\documentclass{amsart}

\usepackage[margin= 60 pt]{geometry}
\usepackage{amsmath,amsthm,amssymb,bm}
\usepackage[colorlinks=true, pdfstartview=FitV, linkcolor=blue, citecolor=blue, urlcolor=blue]{hyperref}
\usepackage{cleveref}
\usepackage{mathtools}
\usepackage{graphicx,color}
\usepackage{ytableau}
\usepackage{tikz}
\usepackage[T1]{fontenc}
\usepackage{tgtermes}
\usepackage{hhline}
\numberwithin{equation}{section}

\newtheorem{thm}{Theorem}[section]
\newtheorem{lem}[thm]{Lemma}
\newtheorem{prop}[thm]{Proposition}

\theoremstyle{definition}
\newtheorem{exam}[thm]{Example}
\newtheorem{defn}[thm]{Definition}

\newtheorem{problem}[thm]{Problem}
\newtheorem{remark}[thm]{Remark}

\newcommand\fs{\operatorname{fs}}
\newcommand\bla{\boldsymbol\lambda}
\newcommand\bmu{\boldsymbol\mu}
\newcommand\ppm{\mathsf{pm}}
\newcommand\rpm{\mathsf{rpm}}
\newcommand\rstep{\mathsf{rstep}}

\newcommand{\ZZ}{\mathbb{Z}}

\newcommand\qand{\quad\text{and}\quad}
\newcommand\Qbinom[3]{\genfrac{[}{]}{0pt}{}{#1}{#2}_{#3}}

\definecolor{darkblue}{rgb}{0.0,0,0.7}
\renewcommand\emph[1]{\textcolor{darkblue}{\it #1}}

\title{Andrews--Gordon and Stanton type identities: bijective and Bailey lemma approaches}

\author{Jehanne Dousse}
\address{Université de Genève, 7--9, rue Conseil Général, 1205 Genève, Switzerland}
\email{jehanne.dousse@unige.ch}

\author{Jihyeug Jang}
\address{Université de Genève, 7--9, rue Conseil Général, 1205 Genève, Switzerland}
\email{jihyeug.jang@unige.ch}

\author{Fr\'ed\'eric Jouhet}
\address{Univ Lyon, Université Claude Bernard Lyon 1, UMR5208, Institut Camille Jordan, F-69622 Villeurbanne, France}
\email{jouhet@math.univ-lyon1.fr}


\begin{document}

\begin{abstract}
In 2018, Stanton proved two types of generalisations of the celebrated Andrews--Gordon and Bressoud identities (in their $q$-series version): one with a similar shape to the original identities, and one involving binomial coefficients. In this paper, we give new proofs of these identities. For the non-binomial identities, we give bijective proofs using the original Andrews--Gordon and Bressoud identities as key ingredients. These proofs are based on particle motion introduced by Warnaar and extended by the first and third authors and Konan. For the binomial identities, we use the Bailey lemma and key lemmas of McLaughlin and Lovejoy, and the order in which we apply the different lemmas plays a central role in the result. We also give an alternative proof of the non-binomial identities using the Bailey lattice. With each of these proofs, new Stanton-type generalisations of classical identities arise naturally, such as generalisations of Kur\c sung\"oz's analogue of Bressoud's identity with opposite parity conditions, and of the Bressoud--G\"ollnitz--Gordon identities.
\end{abstract}

\maketitle


\section{Introduction}\label{sec:intro}

A \emph{partition} \( \lambda \) of a positive integer \( n \) is a
weakly decreasing sequence of positive integers
\( \lambda_1 \ge \lambda_2 \ge  \cdots \ge \lambda_\ell > 0 \) such that
\( \sum_{i = 1} ^\ell\lambda_i = n \). For each partition \( \lambda \),
we associate the \emph{frequency sequence}
\( f = (f_1, f_2, \cdots ) \), where \( f_j \) denotes the number of
times the part \( j \) appears in \( \lambda \). Whenever convenient,
we identify a partition with its corresponding frequency sequence.

Many classical partition identities are elegantly expressed using
\( q \)-series. Throughout this paper, we use standard \(q\)-series
notation which can be found in~\cite{GR}:
\begin{equation*}
(a)_\infty = (a;q)_\infty:=\prod_{j\geq 0}(1-aq^j)\;\;\;\;\mbox{and}\;\;\;\;(a)_k = (a;q)_k:=\frac{(a;q)_\infty}{(aq^k;q)_\infty},
\end{equation*}
where \(k \in \ZZ\), and
\begin{equation*}
(a_1,\ldots,a_m)_k = (a_1,\ldots,a_m;q)_k:=(a_1)_k\cdots(a_m)_k,
\end{equation*}
where \(k\) is an integer or infinity, and as usual \(|q|<1\) to
ensure convergence of infinite products.

One of the most famous results in the theory of partitions is the
Rogers--Ramanujan identities~\cite{RR19}: for \( a \in \{0,1\} \),
\begin{equation}\label{eq:RR}
  \sum_{n \geq 0} \frac{q^{n^2+ (1-a)n}}{(q)_n} = \frac{1}{(q^{2-a},q^{3+a};q^5)_{\infty}}.
\end{equation}
The Rogers--Ramanujan identities have elegant interpretations in terms of partitions. For \( a \in \{0,1\} \), they state
that the number of partitions \( \lambda \) of \( n \) with
the difference condition \( \lambda_i - \lambda_{i+1} \geq 2 \) for
all \( i \), where the part \( 1 \) appears at most \( a \)
times, is equal to the number of partitions of \( n \) into parts
congruent to \( \pm (2-a) \) modulo \( 5 \).

Gordon later extended the Rogers--Ramanujan identities and proved what
is known as Gordon's partition theorem~\cite{Go}. For integers
\( k \geq 1 \) and \( 0 \leq r \leq k \), it states that the number of
partitions \( \lambda \) of \( n \) with the difference condition
\( \lambda_i - \lambda_{i+k} \geq 2 \) for all \( i \), and where the
part \( 1 \) appears at most \( (k-r) \) times, is equal to the
number of partitions of \( n \) into parts not congruent to
\( 0, \pm (k-r+1) \) modulo \( 2k+3 \). The partitions satisfying the
difference condition in Gordon's theorem can also be described easily in terms of
frequency sequences as: 
\begin{equation}\label{eq:comb_AG}
  \text{frequency sequences \( (f_i)_{i \geq 1} \) of \( n \) such that
    \( f_i + f_{i+1} \leq k \) for all \( i \), and
    \( f_1 \leq k-r \).}
\end{equation}
Throughout this paper, we describe such partitions
with difference conditions in terms of frequency sequences.

Andrews~\cite{A74} subsequently proved the following analytic analogue
of Gordon's theorem, now referred to as the Andrews--Gordon
identities.

\begin{thm}[Andrews--Gordon identities~\cite{A74}]\label{thm:AG}
Let \(k \geq 1\) and \(0 \leq r \leq k\) be two integers. We have
\begin{equation}\label{eq:AG}
\sum_{s_1\geq\dots\geq s_{k}\geq0}\frac{q^{s_1^2+\dots+s_{k}^2+s_{k-r+1}+\dots+s_{k}}}{(q)_{s_1-s_2}\dots(q)_{s_{k-1}-s_{k}}(q)_{s_{k}}}=\frac{(q^{2k+3},q^{k+1-r},q^{k+2+r};q^{2k+3})_\infty}{(q)_\infty}.
\end{equation}
\end{thm} 

The case \( k = 1 \) of \eqref{eq:AG} corresponds to the
Rogers--Ramanujan identities~\eqref{eq:RR}. The Andrews--Gordon
identities~\eqref{eq:AG} are also proved in~\cite{Br80} in pair with a
similar formula~\cite[(3.3)]{Br80}, valid for all integers
\(k \geq 1\) and \(0\leq j\leq k\) (we changed notation compared to Bressoud's paper):
\begin{equation}\label{eq:Br3.3}
\sum_{s_1\geq\dots\geq s_{k}\geq0}\frac{q^{s_1^2+\dots+s_{k}^2-s_{1}-\dots-s_{j}}}{(q)_{s_1-s_2}\dots(q)_{s_{k-1}-s_{k}}(q)_{s_{k}}}=\sum_{s=0}^{j}\frac{(q^{2k+3},q^{k+2-j+2s},q^{k+1+j-2s};q^{2k+3})_\infty}{(q)_\infty}.
\end{equation}
Note that there is a small typo in Bressoud's paper: in his
formula~\cite[(3.3)]{Br80}, \(\pm(k-r+i)\) (in his notation) has to be
changed to \(\pm(k-r+i+1)\). The identity~\eqref{eq:Br3.3} is derived
combinatorially from~\eqref{eq:AG} in~\cite{DJK}, while it is used
in~\cite{ADJM} to solve a combinatorial conjecture of Afsharijoo
arising from commutative algebra.

Bressoud also proved an even moduli counterpart of the Andrews--Gordon
identities.

\begin{thm}[Bressoud's identities, {\cite[(3.4)]{Br80}}]\label{thm:2}
  Let \( r \) and \( k \) be integers with \( k \geq 1 \) and
  \( 0 \leq r \leq k \). Then
  \begin{equation}\label{eq:Br}
    \sum_{s_1\geq\dots\geq s_{k}\geq0}\frac{q^{s_1^2+\dots+s_{k}^2+s_{k-r+1}+\dots+s_{k}}}{(q)_{s_1-s_2}\dots(q)_{s_{k-1}-s_{k}}(q^2;q^2)_{s_{k}}}=\frac{(q^{2k+2},q^{k+1-r},q^{k+1+r};q^{2k+2})_\infty}{(q)_\infty}.
  \end{equation}
\end{thm}
From the form of the identities~\eqref{eq:Br}, the partitions with a
congruence condition are readily obtained from the product side.
In~\cite{Br79}, Bressoud introduced the corresponding partitions
satisfying a difference condition along with an additional
restriction. In terms of frequency sequences, the partition identity
corresponding to~\eqref{eq:Br} can be formulated as follows. The
number of partitions of \( n \) into parts not congruent to
\( 0, \pm (k-r+1) \) modulo \( 2k+2 \), is equal to the number of
frequency sequences \( (f_i)_{i \geq 1} \) of \( n \) such that
\( f_i + f_{i+1} \leq k \) for all \( i \), \( f_1 \leq k-r \), and
whenever \( f_i + f_{i+1} = k \) for some \( i \), the parity
condition \( i f_i + (i+1)f_{i+1} \equiv k-r \pmod{2} \) holds.

Bressoud~\cite[(3.5)]{Br80} also proved a similar formula, which was recently derived
combinatorially from~\eqref{eq:Br} in~\cite{DJK}:
\begin{equation}\label{eq:Br3.5}
\sum_{s_1\geq\dots\geq s_{k}\geq0}\frac{q^{s_1^2+\dots+s_{k}^2-s_1-\dots-s_j}}{(q)_{s_1-s_2}\dots(q)_{s_{k-1}-s_{k}}(q^2;q^2)_{s_{k}}}=\sum_{s=0}^{j}\frac{(q^{2k+2},q^{k+1+j-2s},q^{k+1-j+2s};q^{2k+2})_\infty}{(q)_\infty}.
\end{equation}

Kur\c sung\"oz \cite{Kur} considered partitions with the same frequency conditions as in Bressoud's identity, but with the opposite
parity condition \( i f_i + (i+1)f_{i+1} \equiv k-r+1 \pmod{2} \).
He showed that their generating function is equal to
$$\frac{1}{(1+q)(q)_\infty}\left((q^{2k+2},q^{k+2-r},q^{k+r};q^{2k+2})_\infty+q(q^{2k+2},q^{k-r},q^{k+2+r};q^{2k+2})_\infty\right).$$
In~\cite{DJK}, a multisum formula was given for the same generating function, and the identity of Kur\c sung\"oz was proved combinatorially using this multisum and Bressoud's identity, leading to the following identities:

\begin{thm}[Kur\c sung\"oz identities, \cite{Kur} and \cite{DJK}]\label{thm:Kur}
  Let \( r \) and \( k \) be integers with \( k \geq 1 \) and
  \( 0 \leq r \leq k \). Then
  \begin{multline}\label{fij0}
    (1+q)\sum_{s_1\geq\dots\geq s_{k}\geq0}\frac{q^{s_1^2+\dots+s_{k}^2+s_{k-r+1}+\cdots+ s_{k-1}+2s_{k}}}{(q)_{s_1-s_2}\dots(q)_{s_{k-1}-s_{k}}(q^2;q^2)_{s_{k}}}\\
    =\frac{1}{(q)_\infty}\left((q^{2k+2},q^{k+r},q^{k-r+2};q^{2k+2})_\infty+q(q^{2k+2},q^{k+2+r},q^{k-r};q^{2k+2})_\infty\right).
  \end{multline}
  Let \( r \) and \( k \) be integers with \( k \geq 1 \) and
  \( 0 \leq j \leq k \). Then
  \begin{equation}\label{fij}
    \sum_{s_1\geq\dots\geq s_{k}\geq0}\frac{q^{s_1^2+\dots+s_{k}^2-s_1-\cdots- s_j+s_{k}}}{(q)_{s_1-s_2}\dots(q)_{s_{k-1}-s_{k}}(q^2;q^2)_{s_{k}}}=\sum_{s=0}^{j}\frac{(q^{2k+2},q^{k-j+2s},q^{k+2+j-2s};q^{2k+2})_\infty}{(q)_\infty}.
  \end{equation}
\end{thm}

Stanton~\cite{Stant} recently generalised both the
Andrews--Gordon identities~\eqref{eq:AG} and the related
identities~\eqref{eq:Br3.3} in two ways.

\begin{thm}\cite[Theorem~3.1]{Stant}\label{thm:Sta3.1}
  Let \( j,r \geq 0 \) and \( k \geq 1 \) be integers such that
  \( j+r \leq k \). Then
  \begin{multline*}
    \sum_{s_1 \geq \cdots \geq s_k \geq 0}q^{s_1^2 + \cdots + s_k^2 + s_{k-r+1} + \cdots + s_k} \cdot \frac{q^{-s_1 - \cdots - s_j} (1 + q^{s_1 + s_2})(1 + q^{s_2 + s_3}) \cdots (1 + q^{s_{j-1} + s_j})}{(q)_{s_1 - s_2} \cdots (q)_{s_{k-1} - s_k} (q)_{s_k}} \\
    = \sum_{s=0}^j \binom{j}{s} \frac{(q^{2k+3},q^{k+1-r+j-2s}, q^{k+2+r-j+2s} ; q^{2k+3})_\infty}{(q)_\infty}.
  \end{multline*}
  Moreover, the \(j\) factors \(q^{-s_1}\) and
  \(q^{-s_i}(1+q^{s_{i-1}+s_i})\), \(2 \leq i \leq j\) may be replaced
  by any \(j\)-element subset of
  \(\{q^{-s_1}\} \cup \{q^{-s_i}(1+q^{s_{i-1}+s_i}) : 2 \leq i \leq
  k-r\}\).
\end{thm}

\begin{thm}\cite[Theorem~3.2]{Stant}\label{thm:Sta3.2}
  Let \( j,r \geq 0 \) and \( k \geq 1 \) be integers such that
  \( j+r \leq k \). Then
\begin{equation}\label{eq:sAG}
\sum_{s_1\geq\dots\geq s_{k}\geq0}\frac{q^{s_1^2+\dots+s_{k}^2-s_1-\dots-s_j+s_{k-r+1}+\dots+s_{k}}}{(q)_{s_1-s_2}\dots(q)_{s_{k-1}-s_{k}}(q)_{s_{k}}}=\sum_{s=0}^j\frac{(q^{2k+3},q^{k+1-r+j-2s},q^{k+2+r-j+2s};q^{2k+3})_\infty}{(q)_\infty}.
\end{equation}
\end{thm}
We refer to Theorem \ref{thm:Sta3.1} and \ref{thm:Sta3.2} as the
binomial extension and non-binomial extension of the Andrews--Gordon identities, respectively.
As remarked by Stanton, for \(j=0\) both identities reduce to~\eqref{eq:AG}, while for
\(r=0\) Theorem \ref{thm:Sta3.2} yields~\eqref{eq:Br3.3}.

Furthermore, he also presented even moduli versions, generalising both
of Bressoud's identities~\eqref{eq:Br} and \eqref{eq:Br3.5} in a
binomial and a non-binomial extension as well.

\begin{thm}\cite[Theorem~4.1]{Stant}\label{thm:Sta4.1}
  Let \( j,r \geq 0 \) and \( k \geq 1 \) be integers such that
  \( j+r \leq k \). Then
  \begin{multline*}
    \sum_{s_1 \geq \cdots \geq s_k \geq 0}q^{s_1^2 + \cdots + s_k^2 + s_{k-r+1} + \cdots + s_k} \cdot \frac{q^{-s_1 - \cdots - s_j} (1 + q^{s_1 + s_2})(1 + q^{s_2 + s_3}) \cdots (1 + q^{s_{j-1} + s_j})}{(q)_{s_1 - s_2} \cdots (q)_{s_{k-1} - s_k} (q^2;q^2)_{s_k}} \\
    = \sum_{s=0}^j \binom{j}{s} \frac{(q^{2k+2}, q^{k+1-r+j-2s}, q^{k+1+r-j+2s} ; q^{2k+2})_\infty}{(q)_\infty}.
  \end{multline*}
  Moreover, the \(j\) factors \(q^{-s_1}\) and
  \(q^{-s_i}(1+q^{s_{i-1}+s_i})\), \(2 \leq i \leq j\) may be replaced
  by any \(j\)-element subset of
  \(\{q^{-s_1}\} \cup \{q^{-s_i}(1+q^{s_{i-1}+s_i}) : 2 \leq i \leq
  k-r\}\).
\end{thm}

\begin{thm}\cite[Theorem~4.2]{Stant}\label{thm:Sta4.2}
  Let \( j,r \geq 0 \) and \( k \geq 1 \) be integers such that
  \( j+r \leq k \). Then
\begin{equation}\label{eq:sBr}
\sum_{s_1\geq\dots\geq s_{k}\geq0}\frac{q^{s_1^2+\dots+s_{k}^2-s_1-\dots-s_j+s_{k-r+1}+\dots+s_{k}}}{(q)_{s_1-s_2}\dots(q)_{s_{k-1}-s_{k}}(q^2;q^2)_{s_{k}}}=\sum_{s=0}^j\frac{(q^{2k+2},q^{k+1-r+j-2s},q^{k+1+r-j+2s};q^{2k+2})_\infty}{(q)_\infty}.
\end{equation}
\end{thm}

Again, for \(j=0\) these two identities give~\eqref{eq:Br}, while for
\(r=0\) Theorem \ref{thm:Sta4.2} yields Bressoud's
formula~\eqref{eq:Br3.5}. Stanton proved his results by using some
(new properties of) Laurent polynomials \(H_{2n}(z,a|q)\) and some
iterative process.

\medskip
In the first part of this paper, we give alternative proofs of
Stanton's results using the Bailey pair machinery.

The Bailey lemma~\cite{Ba}, whose iterative nature was highlighted by
Andrews~\cite{A1, A2, AAR} through the Bailey chain, provides an
efficient framework to prove \(q\)-series identities. For more details
on the Bailey lemma and its extensions, see \Cref{sec:bailey-pairs}.
Using Bailey-type techniques, we prove a binomial and a non-binomial
generalisation of Theorem~\ref{thm:Kur}, in the same style as
Stanton's results.

\begin{thm}[Binomial extension of the Kur\c sung\"oz identities]\label{thm:binom_Kur}
  Let \( j,r \geq 0 \) and \( k \geq 1 \) be integers such that
  \( j+r \leq k \). Then
  \begin{multline*}
    \sum_{s_1 \geq \cdots \geq s_k \geq 0}q^{s_1^2 + \cdots + s_k^2 + s_{k-r+1} + \cdots + s_{k-1}+2s_k} \cdot \frac{q^{-s_1 - \cdots - s_j} (1 + q^{s_1 + s_2})(1 + q^{s_2 + s_3}) \cdots (1 + q^{s_{j-1} + s_j})}{(q)_{s_1 - s_2} \cdots (q)_{s_{k-1} - s_k} (q^2;q^2)_{s_k}} \\
    = \frac{1}{1+q}\left(\sum_{s=0}^j \binom{j}{s} \left( \frac{(q^{2k+2}, q^{k+2-r+j-2s}, q^{k+r-j+2s} ; q^{2k+2})_\infty}{(q)_\infty}
        + q \frac{(q^{2k+2}, q^{k-r+j-2s}, q^{k+2+r-j+2s} ; q^{2k+2})_\infty}{(q)_\infty}\right) \right).
  \end{multline*}
  Moreover, the \(j\) factors \(q^{-s_1}\) and \(q^{-s_i}(1+q^{s_{i-1}+s_i})\), \(2 \leq i \leq j\) may be replaced by any \(j\)-element subset of \(\{q^{-s_1}\} \cup \{q^{-s_i}(1+q^{s_{i-1}+s_i}) : 2 \leq i \leq k-r\}\).
\end{thm}

\begin{thm}[Non-binomial extension of the Kur\c sung\"oz identities]\label{thm:new_Kur}
  Let \( j,r \geq 0 \) and \( k \geq 1 \) be integers such that
  \( j+r \leq k \). Then
\begin{multline}\label{eq:new_Kur}
\sum_{s_1\geq\dots\geq s_{k}\geq0}\frac{q^{s_1^2+\dots+s_{k}^2-s_1-\dots-s_j+s_{k-r+1}+\dots+s_{k-1}+2s_{k}}}{(q)_{s_1-s_2}\dots(q)_{s_{k-1}-s_{k}}(q^2;q^2)_{s_{k}}}\\
=\frac{1}{1+q}\sum_{s=0}^j\left(\frac{(q^{2k+2},q^{k+2-r+j-2s},q^{k+r-j+2s};q^{2k+2})_\infty}{(q)_\infty}+q\frac{(q^{2k+2},q^{k-r+j-2s},q^{k+2+r-j+2s};q^{2k+2})_\infty}{(q)_\infty}\right).
\end{multline}
\end{thm}

Again,  for \( j = 0 \) both theorems give~\eqref{fij0}, while for
\( r = 0 \) Theorem \ref{thm:new_Kur} yields~\eqref{fij}.

Furthermore, we explore analogues of these constructions for the
Bressoud--G\"ollnitz--Gordon identities. In~\cite[(3.6)--(3.9)]{Br80},
Bressoud proved four extensions of the famous G\"ollnitz--Gordon
identities, introduced in \cite{Go65} and \cite{Gol} independently,
\begin{equation}\label{eq:GG}
  \sum_{n\geq0}\frac{q^{n^2}(-q;q^2)_n}{(q^2;q^2)_n}=\frac{1}{(q,q^4,q^7;q^8)_\infty}, \qand \sum_{n\geq0}\frac{q^{n^2+2n}(-q;q^2)_n}{(q^2;q^2)_n}=\frac{1}{(q^3,q^4,q^5;q^8)_\infty},
\end{equation}
which are modulo \(8\) Rogers--Ramanujan-type
identities. Among them, the identities~\cite[(3.6)]{Br80} state the following.
\begin{thm}[Bressoud--G\"ollnitz--Gordon]\label{thm:BGG}
For all integers \( k \geq 1 \) and \( 0 \leq j \leq k \),
\begin{multline} \label{eq:BGG}
  \sum_{s_1\geq\dots\geq s_{k}\geq0}\frac{q^{2(s_1^2+\dots+s_{k}^2-s_1-\dots-s_{j})}(-q^{1+2s_k};q^2)_\infty}{(q^2;q^2)_{s_1-s_2}\dots(q^2;q^2)_{s_{k-1}-s_{k}}(q^2;q^2)_{s_k}}\\
  =\frac{(-q;q^2)_\infty}{(q^2;q^2)_\infty}\sum_{s=0}^{j}(q^{4k+4},q^{2k+1-2j+2s},q^{2k+3+2j-2s};q^{4k+4})_\infty.
\end{multline}
\end{thm}

Note that taking \(k=1\) and \(j=0\), this formula becomes
\[  
\sum_{s_1\geq0}\frac{q^{2s_1^2}(-q^{1+2s_1};q^2)_\infty}{(q^2;q^2)_{s_1}}=\frac{(q^2;q^4)_\infty}{(q)_\infty}(q^8,q^3,q^5;q^8)_\infty,
\]
therefore the product side is the same as in the first
G\"ollnitz--Gordon identity in~\eqref{eq:GG}. Using the infinite
\(q\)-binomial theorem~\cite[Theorem~10.2.1]{AAR}, the left-hand side
is
\[  
  \sum_{m,\ell\geq0}\frac{q^{2m^2+\ell^2+2m\ell}}{(q^2;q^2)_m(q^2;q^2)_\ell}
  = \sum_{m, \ell \geq 0} \frac{q^{(m+ \ell)^2}}{(q^2; q^2)_{m + \ell}} \cdot q^{m^2}\Qbinom{m + \ell}{m}{q^2}
  =\sum_{n\geq0}\frac{q^{n^2}}{(q^2;q^2)_n}\sum_{m=0}^nq^{m^2} \Qbinom{n}{m}{q^2},
\]
which, by using the finite \(q\)-binomial theorem~\cite[Corollary~10.2.2.~(c)]{AAR}, is the left-hand
side of the first identity in~\eqref{eq:GG}. Similarly, when
\(k=j=1\), Formula~\eqref{eq:BGG} is equivalent to the sum of both
G\"ollnitz--Gordon identities in~\eqref{eq:GG}.

In~\cite{DJK25}, \(m\)-versions of the formulas
\cite[(3.6)--(3.9)]{Br80} are proved using the Bailey lattice, and two
more identities of the same kind are discovered in passing. A natural
question is whether it is possible to prove ``Stanton type'' formulas
generalising these Bressoud--G\"ollnitz--Gordon identities. We answer
positively this question by providing binomial and non-binomial
extensions of the Bresssoud--G\"ollnitz--Gordon identities.

\begin{thm}[Binomial extension of the Bresssoud--G\"ollnitz--Gordon identities]\label{thm:binom_BGG}
  Let \( j,r \geq 0 \) and \( k \geq 1 \) be integers such that
  \( j+r \leq k \). Then
  \begin{multline*}
    \sum_{s_1 \geq \cdots \geq s_{k} \geq 0}
    \frac{q^{2(s_{1}^2 + \cdots + s_{k}^2 + s_{k-r+1} + \cdots + s_{k})} q^{-2s_1}(q^{2s_1}+q^{-2s_2}) \cdots (q^{2s_{j-1}}+q^{-2s_j}) (-q^{1+2s_k}; q^2)_\infty}{(q^2;q^2)_{s_1 - s_{2}}  \cdots (q^2;q^2)_{s_{k-1} - s_{k}} (q^2;q^2)_{s_k} } \\
    = \frac{(-q^3;q^2)_\infty}{(q^2;q^2)_{\infty}} \sum_{s = 0}^{j} \binom{j}{s} \bigg( (q^{4k+4}, q^{2k+3-2r+2j-4s}, q^{2k+1+2r-2j+4s}; q^{4k+4})_\infty  \\
    + q (q^{4k+4}, q^{2k+1-2r+2j-4s}, q^{2k+3+2r-2j+4s}; q^{4k+4})_\infty \bigg).
  \end{multline*}
  Moreover,
  \(q^{-2s_1}(q^{2s_1}+q^{-2s_2}) \cdots (q^{2s_{j-1}}+q^{-2s_j})\) can be
  replaced by any product of \(j\) factors taken from
  \(\{q^{-2s_1}\} \cup \{(q^{2s_{i-1}}+q^{-2s_i}) : 2 \leq i \leq
  k-r\}\).
\end{thm}

\begin{thm}[Non-binomial extension of the Bresssoud--G\"ollnitz--Gordon identities]\label{thm:GGrj}
  Let \( j,r \geq 0 \) and \( k \geq 1 \) be integers such that
  \( j+r \leq k \). Then
\begin{multline}\label{GGrj}
  \sum_{s_1\geq\dots\geq s_{k}\geq0}\frac{q^{2(s_1^2+\dots+s_{k}^2-s_1-\dots-s_{j}+s_{k-r+1}+\dots+s_{k})}(-q^{1+2s_k};q^2)_\infty}{(q^2;q^2)_{s_1-s_2}\dots(q^2;q^2)_{s_{k-1}-s_{k}}(q^2;q^2)_{s_k}}\\ =\frac{(-q^3;q^2)_\infty}{(q^2;q^2)_\infty}\sum_{s=0}^{j}\bigg( (q^{4k+4}, q^{2k+3-2r+2j-4s}, q^{2k+1+2r-2j+4s}; q^{4k+4})_\infty  \\
    + q (q^{4k+4}, q^{2k+1-2r+2j-4s}, q^{2k+3+2r-2j+4s}; q^{4k+4})_\infty \bigg).
\end{multline}
\end{thm}

If we take \(r=0\) in~\eqref{GGrj}, some elementary manipulation
yields~\eqref{eq:BGG}: it is clear for the left-hand sides, while splitting the right-hand side of~\eqref{GGrj} (with $r=0$) into two sums and replacing $s$ by $j-s$ in the second one, we see that both sums are the same, therefore one can factorise by $1+q$. The remaining sum is the one on the left-hand side of~\eqref{eq:BGG} in which even and odd values of $s$ have been split. On the other hand, setting \(j=0\) in~\eqref{GGrj} gives
\begin{multline}\label{GGr0}
\sum_{s_1\geq\dots\geq s_{k}\geq0}\frac{q^{2(s_1^2+\dots+s_{k}^2+s_{k-r+1}+\dots+s_{k})}(-q^{1+2s_k};q^2)_\infty}{(q^2;q^2)_{s_1-s_2}\dots(q^2;q^2)_{s_{k-1}-s_{k}}(q^2;q^2)_{s_k}}=\frac{(-q^3;q^2)_\infty}{(q^2;q^2)_\infty}\\
\times\bigg((q^{4k+4},q^{2k+3-2r},q^{2k+1+2r},q^{4k+4})_\infty+q(q^{4k+4},q^{2k+1-2r},q^{2k+3+2r},q^{4k+4})_\infty\bigg),
\end{multline}
which seems to be new. Note that taking \(k=1\) in~\eqref{GGrj}, we have three possible choices for our integral parameters \(r,j\), namely \((0,0)\), \((1,0)\) and
\((0,1)\). As explained below~\eqref{eq:BGG}, the cases $r=0$ and $j=0$, $j=1$ yield the G\"ollnitz--Gordon identities~\eqref{eq:GG}. For $r=1$ and $j=0$, the formula is the one obtained by taking $k=r=1$ in~\eqref{GGr0}, and we see similarly that it is equivalent to the first G\"ollnitz--Gordon identity plus $q$ times the second one, divided by $1+q$.

 Moreover, we prove two additional general
identities similar to Theorem~\ref{thm:GGrj} (see Theorems~\ref{thm:newSlater} and~\ref{thm:newSlater2}), which
extend Slater's identities~\cite[(8), (12), (13)]{Sl}.

\medskip

In the second part of this paper, we turn to a combinatorial
perspective on Stanton's identities. 
While Bailey pairs are a remarkably powerful and flexible tool for
proving partition identities, they do not provide a combinatorial interpretation of these identities. Combinatorial
approaches to partition identities are generally more difficult. A
famous example is the lack of a simple bijection for the combinatorial version of the Rogers--Ramanujan
identities. Several efforts have been made to understand there
identities from a combinatorial perspective, including work on
bijective proofs (see, e.g., \cite{BP2006}, \cite{BZ1982},
\cite{Corteel2017}, \cite{GM1981}).

It is always simple to interpret the product side of the identities as partitions with congruence conditions. In the case of Rogers--Ramanujan, it is also relatively simple to see that the sum side is the generating function for partitions where parts differ by at least two, or equivalently $f_i +f_{i+1} \leq 1$. However, in the Andrews--Gordon identities and the other identities we consider here, it is not at all obvious combinatorially that the sum side is the generating function for partitions with difference conditions. In the case of Andrews--Gordon, a more natural interpretation has been given by Andrews in terms of Durfee squares \cite{AndrewsDurfee}.

To prove combinatorially that the sum side of the Andrews--Gordon
identities is the generating function for partitions with difference
conditions, Warnaar~\cite{Warnaar1997} introduced a bijection based on
particle motions. He used it to derive a finitisation of the
Andrews--Gordon identities. It was formulated as a one-dimensional
lattice-gas of fermionic particles using slightly different notation.
He then used this finisation to prove the identities. Inspired by
Warnaar's work, the authors of~\cite{DJK}, including the first and
third authors, generalised his approach by adding parts equal to $0$
to the reasoning and introducing an explicit bijection \( \Lambda \)
and its inverse \( \Gamma \), which provided a combinatorial framework
for proving several partition identities. Rather than using a
finitisation, they took a different approach: by combining the
infinite version with the classical Andrews--Gordon and Bressoud
identities, they established several identities, among
which~\eqref{eq:Br3.3},~\eqref{eq:Br3.5},~\eqref{fij0},
and~\eqref{fij}.

Although Stanton
proved Theorems~\ref{thm:Sta3.1}--\ref{thm:Sta4.2} in \cite{Stant}, his proofs did not provide combinatorial interpretations for the sum sides. After reading \cite{DJK}, he asked the authors the following problem, hoping that their techniques could be applied to his identities.

\begin{problem}[Stanton]\label{prob:stanton}
  Give partition-theoretic interpretations and combinatorial proofs of 
  Theorems~\ref{thm:Sta3.1}--\ref{thm:Sta4.2}.
\end{problem}
Indeed, the ideas of~\cite{Warnaar1997} and~\cite{DJK} can also be used to give combinatorial interpretations and proofs of the identities in Theorems~\ref{thm:Sta3.2},
~\ref{thm:Sta4.2}, and~\ref{thm:new_Kur}. We give a suitable interpretation of the sum sides of the identities using particle motion and a combinatorial reasoning on parts equal to $0$. Then we give combinatorial proofs of these identities, using this interpretation and the Andrews--Gordon and Bressoud identities.

Following~\cite{DJK}, we allow partitions to have non-negative parts,
not just positive ones. That is, a \emph{partition} \( \lambda \) of
\( n \) is a weakly decreasing sequence of non-negative integers
\( \lambda_1 \geq \lambda_2 \geq \cdots \geq \lambda_{\ell} \geq 0 \)
whose sum is \( n \). The difference from the previous definition is
that $0$ is now allowed as a part, and is taken into account in the
length of the partition. Accordingly, the associated frequency
sequence becomes \( f = (f_0, f_1, \cdots ) \), rather than
\( (f_1, f_2, \cdots ) \). With this extension, the authors in
\cite{DJK} provided combinatorial proofs of \eqref{eq:Br3.3},
\eqref{eq:Br3.5}, and \eqref{fij}, and gave combinatorial
interpretations in terms of partitions with difference conditions. For
example, they proved that the combinatorial model for \eqref{eq:Br3.3}
is given by
\begin{equation}\label{eq:comb_Br3.3}
  \text{frequency sequences \( (f_i)_{i \geq 0} \) of \( n \) such that
    \( f_i + f_{i+1} \leq k \) for all \( i \), and
    \( f_0 \leq j \).}
\end{equation}

Following this framework, we use particle motion and Theorems~\ref{thm:AG} and~\ref{thm:2} to give combinatorial proofs of Theorems~\ref{thm:Sta3.2},~\ref{thm:Sta4.2}, and~\ref{thm:new_Kur}, and we also provide their following partition-theoretic interpretations.
\begin{thm}\label{thm:Sta3.2_partition}
  The left-hand side of \eqref{eq:sAG} in \Cref{thm:Sta3.2} is the generating
  function for the frequency sequences \( (f_i)_{i \geq 0} \) such
  that
  \begin{itemize}
  \item \( f_i + f_{i+1} \leq k \) for all \( i \), and
  \item \( f_0 \leq j - \max\{f_0 + f_1 - (k-r),0\} \).
  \end{itemize}
\end{thm}

\begin{thm}\label{thm:Sta4.2_partition}
  The left-hand side of \eqref{eq:sBr} in \Cref{thm:Sta4.2} is the generating
  function for the frequency sequences \( (f_i)_{i \geq 0} \) such
  that
  \begin{itemize}
  \item \( f_i + f_{i+1} \leq k \) for all \( i \),
  \item \( f_0 \leq j - \max\{f_0 + f_1 - (k-r),0\} \), and
  \item if \( f_i + f_{i+1} = k \) for some \( i \), then \( i f_i + (i+1)f_{i+1} \equiv k+r-j \pmod{2} \).
  \end{itemize}
\end{thm}
In addition, we obtain new identities of Kur\c sung\"oz type using the
combinatorial model. Interestingly, these formulas coincide with those
previously obtained using the Bailey pairs, thus offering a purely
combinatorial approach.
\begin{thm}\label{thm:Kur_partition}
  The left-hand side of \eqref{eq:new_Kur} in \Cref{thm:new_Kur} is the
  generating function for the frequency sequences
  \( (f_i)_{i \geq 0} \) such that
  \begin{itemize}
  \item \( f_i + f_{i+1} \leq k \) for all \( i \),
  \item \( f_0 \leq j - \max\{f_0 + f_1 - (k-r),0\} \), and
  \item if \( f_i + f_{i+1} = k \) for some \( i \), then \( i f_i + (i+1)f_{i+1} \equiv k+r-j+1 \pmod{2} \).
  \end{itemize}
\end{thm}

The combinatorial interpretation in~\Cref{thm:Sta3.2_partition} is new. When \( j = 0 \), the second
condition becomes \( f_0 = 0 \) and \( f_1 \leq k-r \), which yields
the interpretation~\eqref{eq:comb_AG}, and when \( r = 0 \), it
satisfies \( f_0 \leq j \), yielding~\eqref{eq:comb_Br3.3}.
Furthermore, the same model extends naturally to the Bressoud
(\Cref{thm:Sta4.2_partition}) and Kur\c sung\"oz
(\Cref{thm:Kur_partition}) cases, which shows its flexibility
and confirms that it provides a natural combinatorial interpretation.

This paper is organised as follows. In \Cref{sec:bailey-pairs}, we
recall some basic results and tools from the Bailey machinery that we
will use later in our proofs. In Section~\ref{sec:pf_binom}, we use
combinations of the simplest of these tools (the Bailey lemma and key
lemmas of Lovejoy and McLaughlin) to prove all the binomial extensions
in Theorems~\ref{thm:Sta3.1},~\ref{thm:Sta4.1},~\ref{thm:binom_Kur}
and~\ref{thm:binom_BGG}. In Section~\ref{sec:pf_non-binom}, some
combinations of the Bailey lemma and lattice are used to provide
proofs of all the non-binomial extensions in
Theorems~\ref{thm:Sta3.2},~\ref{thm:Sta4.2},~\ref{thm:new_Kur},~\ref{thm:GGrj},~\ref{thm:newSlater}
and~\ref{thm:newSlater2}. In Section~\ref{sec:insert-map-mult}, we
recall the particle motion bijection to prepare the combinatorial
proofs of the theorems in the next sections. In
Section~\ref{sec:combi-proof-1.2}, using this bijection and the
Andrews--Gordon identities, a combinatorial proof of
Theorem~\ref{thm:Sta3.2} is provided. In Section~\ref{sec:proof3.4},
we use the bijection and the Bressoud identities to give combinatorial
proofs for Theorems~\ref{thm:Sta4.2} and ~\ref{thm:new_Kur}. In
Section~\ref{sec:final-rmk}, we conclude the paper by a list of
remarks and questions.

\section{Bailey pairs}
\label{sec:bailey-pairs}

Fix complex numbers \(a\) and \(q\). Recall~\cite{Ba} that a
\emph{Bailey pair} \(((\alpha_n)_{n\geq0},\,(\beta_n)_{n\geq0})\)
(\((\alpha_n,\,\beta_n)\) for short) relative to \(a\) is a pair of
sequences satisfying:
\begin{equation}\label{bp}
\beta_n=\sum_{\ell=0}^n\frac{\alpha_\ell}{(q)_{n-\ell}(aq)_{n+\ell}}\;\;\;\;\forall\,n\in\mathbb{N}.
\end{equation}
By convention, we set \( \alpha_\ell = 0  \) for all \( \ell < 0 \).

Given a Bailey pair, the Bailey lemma~\cite{Ba} allows one to produce infinitely new Bailey pairs. Bailey~\cite{Ba} originally applied it without iterating it, and Andrews~\cite{A1} generalised Bailey's approach to exhibit its iterative nature with the concept of Bailey chain. 
\begin{thm}[Bailey lemma, Andrews' version]\label{thm:baileylemma}
If \((\alpha_n,\,\beta_n)\) is a Bailey pair relative to \(a\), then so is \((\alpha'_n,\,\beta'_n)\), where
\[
  \alpha'_n=\frac{(\rho,\sigma)_n(aq/\rho\sigma)^n}{(aq/\rho,aq/\sigma)_n}\,\alpha_n,
\]
and
\[
  \beta'_n=\sum_{\ell=0}^n\frac{(\rho,\sigma)_\ell(aq/\rho\sigma)_{n-\ell}(aq/\rho\sigma)^\ell}{(q)_{n-\ell}(aq/\rho,aq/\sigma)_n}\,\beta_\ell.
\]
\end{thm}

In this paper, we use the following two particular cases.
\begin{lem}[Bailey lemma with \( \rho,\sigma \to \infty \)]\label{lem:BL}
  If \((\alpha_n,\,\beta_n)\) is a Bailey pair relative to \(a\), then
  so is \((\alpha'_n,\,\beta'_n)\), where
\begin{equation*}
  \alpha'_n=a^n q^{n^2}\alpha_n, \qand
  \beta'_n=\sum_{\ell=0}^n\frac{a^\ell q^{\ell^2}}{(q)_{n-\ell}}\,\beta_\ell.
\end{equation*}
\end{lem}

\begin{lem}[Bailey lemma with \( \sigma \to \infty \)]\label{lem:BL2}
  If \( (\alpha_n,\,\beta_n) \) is a Bailey pair relative to \(a\),
  then so is \((\alpha'_n,\,\beta'_n)\), where
\begin{equation*}
  \alpha'_n=\frac{(-1)^n (\rho)_n(aq/\rho)^n q^{\binom{n}{2}}}{(aq/\rho)_n}\,\alpha_n , \qand
  \beta'_n=\sum_{\ell=0}^n\frac{(-1)^\ell(\rho)_\ell (aq/\rho)^\ell q^{\binom{\ell}{2}}}{(aq/\rho)_n (q)_{n-\ell}}\,\beta_\ell.
\end{equation*}
\end{lem}

Despite its simple form and quite elementary proof (see~\cite[Appendix~(II.12)]{GR}), the Bailey lemma can be used to prove many $q$-series identities. One of the most widely used Bailey pairs is the unit Bailey pair, defined in~\cite[(2.12) and (2.13)]{A1}.
\begin{defn}[\cite{A1}] \label{def:ubp}
  The \emph{unit Bailey pair}
  relative to \(a\) is the pair \((\alpha_n,\,\beta_n)\) defined by
  \begin{equation}\label{ubp}
    \alpha_n=(-1)^nq^{\binom{n}{2}}\frac{1-aq^{2n}}{1-a}\frac{(a)_n}{(q)_n},  \qand \beta_n=\delta_{n,0}.
  \end{equation}
\end{defn}
The Rogers--Ramanujan identities can be proved easily by applying \Cref{thm:baileylemma} twice to the unit Bailey pair~\eqref{ubp}, and the \(r=0\) and \(r=k\) special instances of the Andrews--Gordon identities in~\Cref{thm:AG} follow by iterating \(k+1\geq 2\) times this process.

But the Bailey chain is not sufficient to prove the cases \(0 < r < k\) of the Andrews--Gordon identities, and the Bailey lattice was developed in~\cite{AAB} to remedy this problem by switching the parameter \(a\) to \(a/q\) at some point before iterating the Bailey lemma.
 
\begin{thm}[Bailey lattice]\label{thm:baileylattice}
If \((\alpha_n,\,\beta_n)\) is a Bailey pair relative to \(a\), then  \((\alpha'_n,\,\beta'_n)\)  is a Bailey pair relative to \(a/q\), where
\[\alpha'_n=\frac{(\rho,\sigma)_n(a/\rho\sigma)^n}{(a/\rho,a/\sigma)_n}(1-a)\left(\frac{\alpha_n}{1-aq^{2n}}-\frac{aq^{2n-2}\alpha_{n-1}}{1-aq^{2n-2}}\right),\]
and
\[\beta'_n=\sum_{\ell=0}^n\frac{(\rho,\sigma)_\ell(a/\rho\sigma)_{n-\ell}(a/\rho\sigma)^\ell}{(q)_{n-\ell}(a/\rho,a/\sigma)_n}\,\beta_\ell.\]
\end{thm}
We use the particular case when \( \rho, \sigma \to \infty \).
\begin{lem}[Bailey lattice with \( \rho, \sigma \to \infty \)]\label{lem:Baileylattice}
  If \((\alpha_n,\,\beta_n)\) is a Bailey pair relative to \(a\), then
  \((\alpha'_n,\,\beta'_n)\) is a Bailey pair relative to \(a/q\),
  where
  \[
    \alpha'_n=a^n q^{n^2-n} (1-a) \left(\frac{\alpha_n}{1-aq^{2n}}-\frac{aq^{2n-2}\alpha_{n-1}}{1-aq^{2n-2}}\right),\qand
    \beta'_n=\sum_{\ell=0}^n \frac{a^\ell q^{\ell^2-\ell}}{(q)_{n-\ell}} \,\beta_\ell.
  \]
\end{lem}

Authors have been interested in ways to avoid using the Bailey lattice. For example,
Andrews, Schilling and Warnaar \cite[Section~3]{ASW} proved~\eqref{eq:AG} 
using the Bailey lemma and bypassing the Bailey lattice, Bressoud, Ismail and Stanton~\cite{BIS} used a change of
base to avoid the lattice, and
McLaughlin~\cite{M} showed that~\eqref{eq:AG} can be proved
easily by combining the Bailey Lemma with a simple lemma,
which gives a Bailey pair relative to \(a/q\) given a Bailey pair
relative to \(a\).

\begin{lem}[McLaughlin, Key lemma 1]\label{lem:key1}
  \label{lem:Mac}
  If \((\alpha_n,\,\beta_n)\) is a Bailey pair relative to \(a\), then
  \((\alpha'_n,\,\beta'_n)\) is a Bailey pair relative to \(a/q\),
  where
  \begin{equation*}
    \alpha'_n=(1-a)\left(\frac{\alpha_n}{1-aq^{2n}}-\frac{aq^{2n-2}\alpha_{n-1}}{1-aq^{2n-2}}\right), \qand \beta'_n=\beta_n.
  \end{equation*}
\end{lem}

In \cite{DJK25}, the first and third authors, together with Konan, showed that the Bailey lattice follows directly from  this key lemma and the Bailey lemma.
They also extended it to a bilateral version and deduced a bilateral
Bailey lattice which they used to prove \(m\)-versions of the
Andrews--Gordon and Bressoud identities. The following lemma, also due
to McLaughlin~\cite{M}, is similar to \Cref{lem:key1} (as it also
transforms \(a\) into \(a/q\)) and led in \cite{DJK25} to a different
bilateral Bailey lattice.

\begin{lem}[McLaughlin, Key lemma 2]\label{lem:key2}
  If \((\alpha_n,\,\beta_n)\) is a Bailey pair relative to \(a\), then
  \((\alpha'_n,\,\beta'_n)\) is a Bailey pair relative to \(a/q\),
  where
  \begin{equation*}
    \alpha'_n=(1-a)\left(\frac{q^n\alpha_n}{1-aq^{2n}}-\frac{q^{n-1}\alpha_{n-1}}{1-aq^{2n-2}}\right), \qand \beta'_n=q^n\beta_n.
  \end{equation*}
  \end{lem}

On the other hand, Lovejoy \cite[(2.4) and (2.5)]{Lo04} proved a lemma
which transforms a Bailey pair relative to \(a\) into a Bailey pair
relative to \(aq\).

\begin{lem}[Lovejoy]\label{lem:love}
  If \((\alpha_n,\,\beta_n)\) is a Bailey pair relative to \(a\), then
  \((\alpha'_n,\,\beta'_n)\) is a Bailey pair relative to \(aq\),
  where
  \[
    \alpha'_n= \frac{(1-aq^{2n+1})(aq/b)_n(-b)^nq^{n(n-1)/2}}{(1-aq)(bq)_n} \sum_{\ell=0}^n \frac{(b)_\ell}{(aq/b)_\ell}(-b)^{-\ell}q^{-\ell(\ell-1)/2}\alpha_\ell,
  \]
  and
  \[
    \beta'_n=\frac{1-b}{1-bq^n}\beta_n.
  \]
\end{lem}

Again, we will mostly need in our calculations a particular case of this lemma (namely
\(b= 0\)), which can be seen as the inverse of Lemma~\ref{lem:key1}.

\begin{lem}[Lovejoy's lemma with $b=0$]\label{lem:L1}
  If \((\alpha_n,\,\beta_n)\) is a Bailey pair relative to \(a\), then
  \((\alpha'_n,\,\beta'_n)\) is a Bailey pair relative to \(aq\),
  where
  \[
    \alpha'_n= \frac{(1-aq^{2n+1})a^nq^{n^2}}{1-aq} \sum_{\ell=0}^n a^{-\ell} q^{-\ell^2} \alpha_\ell,\qand \beta'_n= \beta_n.
  \]
\end{lem}

In~\cite{AAB}, the following result is obtained (in a more general form) by iterating \(r+1\) times \Cref{thm:baileylemma}, using Theorem~\ref{thm:baileylattice}, and concluding with \(k-r-1\) times \Cref{thm:baileylemma} with \(a\) replaced by \(a/q\).
\begin{thm}[Agarwal--Andrews--Bressoud, new notation]\label{thm:corobaileylattice}
If \((\alpha_n,\,\beta_n)\) is a Bailey pair relative to \(a\), then for all integers $k\geq 1$ and  \(-1\leq r\leq k\), we have:
\begin{multline}\label{corolattice}
\sum_{s_1\geq\dots\geq s_{k+1}\geq0}\frac{a^{s_1+\dots+s_{k+1}}q^{s_1^2+\dots+s_{k+1}^2-s_1-\dots-s_{k-r}}}{(q)_{s_1-s_2}\dots(q)_{s_{k}-s_{k+1}}}\beta_{s_{k+1}}\\
=\frac{1}{(aq)_\infty}\sum_{\ell\geq0}a^{(k+1)\ell}q^{(k+1)\ell^2-(k-r)\ell}\frac{1-a^{k-r+1}q^{2\ell(k-r+1)}}{1-aq^{2\ell}}\alpha_\ell.
\end{multline}
\end{thm}

In~\cite{AAB}, Agarwal, Andrews and Bressoud prove the Andrews--Gordon identities~\eqref{eq:AG} by applying Theorem~\ref{thm:corobaileylattice} to the unit Bailey pair~\eqref{ubp} with \(a=q\) and factorising the right-hand side using the Jacobi triple product identity~\cite[Appendix, (II.28)]{GR}
\begin{equation}\label{eq:jtp}
\sum_{\ell\in\mathbb{Z}}(-1)^\ell z^\ell q^{\ell(\ell-1)/2}=(q,z,q/z;q)_\infty.
\end{equation}

Moreover, it is explained in~\cite{DJK25} how~\eqref{eq:Br3.3} is
simply a consequence of~\eqref{corolattice} with \(a=1\)
and~\eqref{eq:jtp}. In the same paper,~\eqref{eq:AG}
and~\eqref{eq:Br3.3} are embedded through a bilateral version of the
Bailey lattice in a single generalisation, which is called
``\(m\)-version of the Andrews--Gordon identities'', where the
parameter \(m\) is a non-negative integer. Interestingly, Stanton's
non-binomial formula~\eqref{eq:sAG} provides another embedding
of~\eqref{eq:AG} and~\eqref{eq:Br3.3}.

The structure of our proofs of Theorems
\ref{thm:Sta3.1}--\ref{thm:GGrj} using Bailey pairs is the following.
We take a well-chosen Bailey pair and apply $k+1$ times either the
Bailey lemma or Bailey lattice or one of the key lemmas to obtain a
new Bailey pair, and then we let $n$ tend to infinity in the relation
\eqref{bp} corresponding to that new Bailey pair. We then simplify the
right-hand side and apply the Jacobi triple product
identity~\eqref{eq:jtp} to obtain the corresponding product
expression. This type of techniques allows us to obtain both the
binomial and non-binomial extensions of the identities under
consideration.

We conclude this section by introducing two Bailey pairs, which will
be used later in our proofs in addition to the unit Bailey pair. The first one is
obtained from the change of base given in~\cite[(D4)]{BIS}, which
asserts that if \((\alpha_n,\,\beta_n)\) is a Bailey pair relative to
\(a\), then so is \((\alpha'_n,\,\beta'_n)\), where
\begin{equation*}
\alpha'_n=\frac{1+a}{1+aq^{2n}}q^n\alpha_n(a^2,q^2)\quad\mbox{and}\quad\beta'_n=\sum_{\ell=0}^n\frac{(-a)_{2\ell}}{(q^2;q^2)_{n-\ell}}q^{\ell}\beta_\ell(a^2,q^2).
\end{equation*}
Applying this to the unit Bailey pair~\eqref{ubp} gives the following Bailey pair.
\begin{defn}[\cite{BIS}]\label{def:D4}
  The \emph{Bailey pair (D'4)} relative to \( a \) is
  \( (\alpha_n, \beta_n) \) where
\begin{equation}\label{D4ubp}
\alpha_n=(-1)^nq^{n^2}\frac{1-aq^{2n}}{1-a}\frac{(a^2;q^2)_n}{(q^2;q^2)_n}, \qand \beta_n=\frac{1}{(q^2;q^2)_{n}}.
\end{equation}
\end{defn}
Similarly, the change of base from~\cite[(D1)]{BIS} states that if
\((\alpha_n,\,\beta_n)\) is a Bailey pair relative to \(a\), then so
is \((\alpha'_n,\,\beta'_n)\), where
\begin{equation*}
\alpha'_n=\alpha_n(a^2,q^2)\quad\mbox{and}\quad\beta'_n=\sum_{\ell=0}^n\frac{(-aq)_{2\ell}}{(q^2;q^2)_{n-\ell}}q^{n-\ell}\beta_\ell(a^2,q^2).
\end{equation*}
Applying this to the unit Bailey pair~\eqref{ubp} gives the following Bailey pair.
\begin{defn}[\cite{BIS}]\label{def:D1}
  The \emph{Bailey pair (D'1)} relative to \( a \) is
  \( (\alpha_n, \beta_n) \) where
\begin{equation}\label{D1ubp}
\alpha_n=(-1)^nq^{n^2-n}\frac{1-a^2q^{4n}}{1-a^2}\frac{(a^2;q^2)_n}{(q^2;q^2)_n}, \qand \beta_n=\frac{q^n}{(q^2;q^2)_{n}}.
\end{equation}
\end{defn}

\section{Proofs of binomial extensions} \label{sec:pf_binom}

\subsection{Preliminary results}
In this section, we prove preliminary results that will be useful in
the proofs of Theorems~\ref{thm:Sta3.1}, \ref{thm:Sta4.1},
\ref{thm:binom_Kur}, and \ref{thm:binom_BGG}. We first combine several
simple lemmas mentioned in Section~\ref{sec:bailey-pairs} to prove a slightly more
involved one.

\begin{lem}\label{lem:star}
  If \((\alpha_n,\,\beta_n)\) is a Bailey pair relative to \(a\), then
  so is \((\alpha'_n,\,\beta'_n)\), where
  \[
    \alpha'_n=a^n q^{n^2-n} \left( (1+q^{2n}) \alpha_n +
      (1-aq^{2n})(1-a^{-1}) \sum_{\ell=0}^{n-1} \alpha_\ell\right),
  \]
  and
  \[
    \beta'_n=\sum_{\ell=0}^n\frac{a^\ell q^{\ell^2}}{(q)_{n-\ell}} (q^n + q^{-\ell})
    \beta_\ell.
  \]
\end{lem}
\begin{proof}
  Let \((\alpha_n,\,\beta_n)\) be a Bailey pair relative to \(a\). By
  Lemma \ref{lem:key1}, \((\alpha^{(1)}_n,\,\beta^{(1)}_n)\) is a Bailey pair
  relative to \(a/q\), where
  \begin{equation*} \alpha^{(1)}_n=(1-a)\left(\frac{\alpha_n}{1-aq^{2n}}-\frac{aq^{2n-2}\alpha_{n-1}}{1-aq^{2n-2}}\right), \quad \beta^{(1)}_n=\beta_n.
  \end{equation*}
  Then by Lemma \ref{lem:BL} with \(a\) replaced by \(a/q\),
  \((\alpha^{(2)}_n,\,\beta^{(2)}_n)\) is a Bailey pair relative to \(a/q\),
  where
  \[
    \alpha^{(2)}_n=a^n q^{n^2-n}
    (1-a)\left(\frac{\alpha_n}{1-aq^{2n}}-\frac{aq^{2n-2}\alpha_{n-1}}{1-aq^{2n-2}}\right),
    \quad \beta^{(2)}_n=\sum_{\ell=0}^n\frac{a^\ell q^{\ell^2-\ell}}{(q)_{n-\ell}}\,\beta_\ell.
  \]
  Finally, by Lemma \ref{lem:L1} with \(a\) replaced by \(a/q\),
  \((\alpha^{(3)}_n,\,\beta^{(3)}_n)\) is a Bailey pair relative to \(a\),
  where
  \begin{equation}
    \label{eq:star1}
    \alpha^{(3)}_n= a^nq^{n^2-n} \left( \alpha_n+ (1-aq^{2n}) \sum_{\ell=0}^{n-1} \alpha_\ell \right), \quad
    \beta^{(3)}_n= \sum_{\ell=0}^n\frac{a^\ell q^{\ell^2-\ell}}{(q)_{n-\ell}}\,\beta_\ell.
  \end{equation}

  Now start again with \((\alpha_n,\,\beta_n)\) Bailey pair relative
  to \(a\). By Lemma \ref{lem:BL}, so is
  \((\tilde{\alpha}^{(1)}_n,\,\tilde{\beta}^{(1)}_n)\), where
  \[
    \tilde{\alpha}^{(1)}_n=a^n q^{n^2}\alpha_n, \quad
    \tilde{\beta}^{(1)}_n=\sum_{\ell=0}^n\frac{a^\ell q^{\ell^2}}{(q)_{n-\ell}}\,\beta_\ell.
  \]
  By Lemma \ref{lem:key2},
  \((\tilde{\alpha}^{(2)}_n,\,\tilde{\beta}^{(2)}_n)\) is a Bailey pair
  relative to \(a/q\), where
  \[
    \tilde{\alpha}^{(2)}_n=(1-a)\left(\frac{a^n
        q^{n^2+n}\alpha_n}{1-aq^{2n}}-\frac{a^{n-1}
        q^{n^2-n}\alpha_{n-1}}{1-aq^{2n-2}}\right), \quad
    \tilde{\beta}^{(2)}_n=q^n\sum_{\ell=0}^n\frac{a^\ell
      q^{\ell^2}}{(q)_{n-\ell}}\,\beta_\ell.
  \]
  Finally, by Lemma \ref{lem:L1} with \(a\) replaced by \(a/q\),
  \((\tilde{\alpha}^{(3)}_n,\,\tilde{\beta}^{(3)}_n)\) is a Bailey pair
  relative to \(a\), where
  \begin{equation}
    \label{eq:star2}
    \tilde{\alpha}^{(3)}_n= a^nq^{n^2-n} \left(q^{2n} \alpha_n - a^{-1}(1-aq^{2n}) \sum_{\ell=0}^{n-1} \alpha_\ell \right), \quad
    \tilde{\beta}^{(3)}_n = q^n\sum_{\ell=0}^n\frac{a^\ell q^{\ell^2}}{(q)_{n-\ell}}\,\beta_\ell.
  \end{equation}
  Adding \eqref{eq:star1} and \eqref{eq:star2} gives the desired
  result, by linearity of the Bailey pair relation~\eqref{bp}.
\end{proof}

When \(a=1\), Lemma \ref{lem:star} simplifies and gives the following,
which we will use in our proofs.

\begin{lem}\label{lem:star1}
  If \((\alpha_n,\,\beta_n)\) is a Bailey pair relative to \(1\), then
  so is \((\alpha'_n,\,\beta'_n)\), where
  \[
    \alpha'_n=q^{n^2-n} (1+q^{2n}) \alpha_n , \qand
    \beta'_n=\sum_{\ell=0}^n\frac{q^{\ell^2}}{(q)_{n-\ell}} (q^n + q^{-\ell})
    \beta_\ell.
  \]
\end{lem}

\begin{remark}\label{rmk:commute}
  If \((\alpha_n,\,\beta_n)\) is a Bailey pair relative to \(1\), and
  we apply first Lemma \ref{lem:BL} and then Lemma \ref{lem:star1} to
  it, we obtain the Bailey pair \((\alpha^{(1)}_n,\,\beta^{(1)}_n)\),
  where
  \[
    \alpha^{(1)}_n= q^{2n^2-n} (1+q^{2n}) \alpha_n , \quad
    \beta^{(1)}_n=\sum_{\ell=0}^n \sum_{m=0}^\ell
    \frac{q^{\ell^2+m^2}}{(q)_{n-\ell}(q)_{\ell-m}} (q^n + q^{-\ell}) \,\beta_m.
  \]

  If we apply first Lemma \ref{lem:star1} and then Lemma \ref{lem:BL}
  to it, we obtain the Bailey pair
  \((\alpha^{(2)}_n,\,\beta^{(2)}_n)\), where
  \[
    \alpha^{(2)}_n= q^{2n^2-n} (1+q^{2n}) \alpha_n , \quad
    \beta^{(2)}_n = \sum_{\ell=0}^n \sum_{m=0}^\ell
    \frac{q^{\ell^2+m^2}}{(q)_{n-\ell}(q)_{\ell-m}} (q^\ell + q^{-m}) \beta_m.
  \]
  Given that \(\alpha^{(1)}_n=\alpha^{(2)}_n\) and that
  \((\alpha^{(1)}_n,\,\beta^{(1)}_n)\) and
  \((\alpha^{(2)}_n,\,\beta^{(2)}_n)\) are Bailey pairs, we know that
  \(\beta^{(1)}_n=\beta^{(2)}_n\), even if it is less obvious from the
  expressions for \(\beta^{(1)}_n\) and \(\beta^{(2)}_n\) given above.
\end{remark}

This remark plays a key role in proving the ``Moreover, the \(j\)
factors \(q^{-s_1}\) and \(q^{-s_i}(1+q^{s_{i-1}+s_i})\),
\(2 \leq i \leq j\) may be replaced by any \(j\)-element subset of
\(\{q^{-s_1}\} \cup \{q^{-s_i}(1+q^{s_{i-1}+s_i}) : 2 \leq i \leq
k-r\}\)" part of the four main theorems.

\medskip
Now we use Lemma \ref{lem:star1} to prove two propositions.

\begin{prop}
  \label{prop:common1}
  If \((\alpha_n,\,\beta_n)\) is a Bailey pair relative to \(q\),
  then, for all integers \(j,r \geq 0\) and \(k\geq1\) such that
  \(r+j\leq k\), \((\alpha^{(k+1)}_n,\,\beta^{(k+1)}_n)\) is a Bailey pair
  relative to \(1\), where
  \begin{equation}
    \label{eq:proof4.1}
    \begin{aligned}
      \alpha^{(k+1)}_n &= (1-q)(1+q^{2n})^j q^{(k+1)n^2+(r+1-j)n} \left(\frac{\alpha_n}{1-q^{2n+1}}-\frac{q^{-2rn-1}\alpha_{n-1}}{1-q^{2n-1}}\right), \\
      \beta^{(k+1)}_n &= \sum_{n \geq s_1 \geq \cdots \geq s_{k+1} \geq 0} \frac{q^{s_{1}^2 + \cdots + s_{k+1}^2 + s_{k-r+1} + \cdots + s_{k+1}} (q^n+q^{-s_1})(q^{s_1}+q^{-s_2}) \cdots (q^{s_{j-1}}+q^{-s_j})}{(q)_{n - s_{1}} (q)_{s_1 - s_{2}}  \cdots (q)_{s_{k} - s_{k+1}}} \,\beta_{s_{k+1}}.
    \end{aligned}
  \end{equation}
  Moreover,
  \((q^n+q^{-s_1})(q^{s_1}+q^{-s_2}) \cdots (q^{s_{j-1}}+q^{-s_j})\)
  can be replaced by any product of \(j\) factors taken from
  \(\{(q^n+q^{-s_1})\} \cup \{(q^{s_{i-1}}+q^{-s_i}) : 2 \leq i \leq
  k-r\}\).
\end{prop}
\begin{proof}
  Let \((\alpha_n,\,\beta_n)\) be a Bailey pair relative to \(q\).
  First apply Lemma \ref{lem:BL} to it \(r+1\) times with \(a=q\). This
  gives a Bailey pair \((\alpha^{(r+1)}_n,\,\beta^{(r+1)}_n)\) relative to
  \(q\), where
  \[
    \alpha^{(r+1)}_n= q^{(r+1)n^2+(r+1)n}\alpha_n,
  \]
  and
  \[
    \beta^{(r+1)}_n=\beta^{(r+1)}_{s_{k-r}}
    = \sum_{s_{k-r} \geq \cdots \geq s_{k+1} \geq 0} \frac{q^{s_{k-r+1}^2 + \cdots + s_{k+1}^2 + s_{k-r+1} + \cdots + s_{k+1}}}{(q)_{s_{k-r} - s_{k-r+1}} \cdots (q)_{s_{k} - s_{k+1}}} \,\beta_{s_{k+1}}.
  \]

  Next apply Lemma \ref{lem:key1} with \(a=q\) once. This yields a
  Bailey pair \((\tilde{\alpha}^{(r+1)}_n,\,\tilde{\beta}^{(r+1)}_n)\)
  relative to \(1\), where
  \[
    \tilde{\alpha}^{(r+1)}_n= (1-q)q^{(r+1)n^2+(r+1)n} \left(\frac{\alpha_n}{1-q^{2n+1}}-\frac{q^{-2rn-1}\alpha_{n-1}}{1-q^{2n-1}}\right),
  \]
  and
  \[
    \tilde{\beta}^{(r+1)}_n=\tilde{\beta}^{(r+1)}_{s_{k-r}}=\sum_{s_{k-r} \geq \cdots \geq s_{k+1} \geq 0} \frac{q^{s_{k-r+1}^2 + \cdots + s_{k+1}^2 + s_{k-r+1} + \cdots + s_{k+1}}}{(q)_{s_{k-r} - s_{k-r+1}} \cdots (q)_{s_{k} - s_{k+1}}} \,\beta_{s_{k+1}}.
  \]

  Now apply Lemma \ref{lem:BL} \(k-r-j\) times with \(a=1\). This
  gives a Bailey pair \((\alpha^{(k-j+1)}_n,\,\beta^{(k-j+1)}_n)\)
  relative to \(1\), where
  \[
    \alpha^{(k-j+1)}_n= (1-q)q^{(k-j+1)n^2+(r+1)n} \left(\frac{\alpha_n}{1-q^{2n+1}}-\frac{q^{-2rn-1}\alpha_{n-1}}{1-q^{2n-1}}\right),
  \]
  and
  \[
    \beta^{(k-j+1)}_n= \beta^{(k-j+1)}_{s_{j}} = \sum_{s_{j} \geq \cdots \geq s_{k+1} \geq 0} \frac{q^{s_{j+1}^2 + \cdots + s_{k+1}^2 + s_{k-r+1} + \cdots + s_{k+1}}}{(q)_{s_{j} - s_{j+1}} \cdots (q)_{s_{k} - s_{k+1}}} \,\beta_{s_{k+1}}.
  \]

  Finally, apply \(j\) times Lemma \ref{lem:star1}. This results in a
  Bailey pair \((\alpha^{(k+1)}_n,\,\beta^{(k+1)}_n)\) relative to \(1\),
  where
  \begin{align*}
    \alpha^{(k+1)}_n
    &= (1-q)(1+q^{2n})^j q^{(k+1)n^2+(r-j+1)n} \left(\frac{\alpha_n}{1-q^{2n+1}}-\frac{q^{-2rn-1}\alpha_{n-1}}{1-q^{2n-1}}\right), \\
    \beta^{(k+1)}_n
    &= \sum_{n \geq s_1 \geq \cdots \geq s_{k+1} \geq 0} \frac{q^{s_{1}^2 + \cdots + s_{k+1}^2 + s_{k-r+1} + \cdots + s_{k+1}} (q^n+q^{-s_1})(q^{s_1}+q^{-s_2}) \cdots (q^{s_{j-1}}+q^{-s_j})}{(q)_{n - s_{1}} (q)_{s_1 - s_{2}}  \cdots (q)_{s_{k} - s_{k+1}}} \,\beta_{s_{k+1}}.
  \end{align*}

  By Remark~\ref{rmk:commute}, the \(k-r-j\) applications of Lemma
  \ref{lem:BL} and the \(j\) applications of Lemma \ref{lem:star1} can
  be done in any order. Hence, the product
  \((q^n+q^{-s_1})(q^{s_1}+q^{-s_2}) \cdots (q^{s_{j-1}}+q^{-s_j})\)
  can be replaced by any product of \(j\) factors taken from
  \(\{(q^n+q^{-s_1})\} \cup \{(q^{s_{i-1}}+q^{-s_i}) : 2 \leq i \leq
  k-r\}\).
\end{proof}

We can now take a limit and obtain the following.
\begin{prop}
  \label{prop:common2}
  Let \((\alpha_n,\,\beta_n)\) be a Bailey pair relative to \(q\).
  Then, for all integers \(j,r \geq 0\) and \(k\geq1\) such that
  \(r+j\leq k\), we have
  \begin{multline*}
    \sum_{s_1 \geq \cdots \geq s_{k+1} \geq 0}  \frac{q^{s_{1}^2 + \cdots + s_{k+1}^2 + s_{k-r+1} + \cdots + s_{k+1}} q^{-s_1}(q^{s_1}+q^{-s_2}) \cdots (q^{s_{j-1}}+q^{-s_j})}{(q)_{s_1 - s_{2}}  \cdots (q)_{s_{k} - s_{k+1}}} \,\beta_{s_{k+1}}\\
                                             = \frac{1}{(q)_{\infty}} \sum_{\ell=0}^{\infty}  q^{(k+1)\ell^2+(r-j+1)\ell} \frac{1-q}{1-q^{2\ell+1}} \left((1+q^{2\ell})^j - (1+q^{2\ell+2})^j q^{(k-r+1)(2\ell+1)-j}\right) \alpha_{\ell}.
  \end{multline*}
  Moreover
  \(q^{-s_1}(q^{s_1}+q^{-s_2}) \cdots (q^{s_{j-1}}+q^{-s_j})\) can be
  replaced by any product of \(j\) factors taken from
  \(\{q^{-s_1}\} \cup \{(q^{s_{i-1}}+q^{-s_i}) : 2 \leq i \leq
  k-r\}\).
\end{prop}
\begin{proof}
  Letting \(n \rightarrow \infty\) in \eqref{bp} gives that if
  \((\alpha_n, \beta_n)\) is a Bailey pair relative to \(1\), then
  \begin{equation}
    \label{eq:lim}
    \beta_{\infty} = \frac{1}{(q)_{\infty}^2} \sum_{\ell=0}^{\infty} \alpha_{\ell}.
  \end{equation}

  Let \((\alpha_n, \beta_n)\) be a Bailey pair relative to \(q\).
  Keeping the notation of Proposition \ref{prop:common1},
  \((\alpha^{(k+1)}_n,\,\beta^{(k+1)}_n)\) is a Bailey pair relative to
  \(1\). Thus we combine \eqref{eq:lim} and \eqref{eq:proof4.1}
  with \(n \rightarrow \infty\) to obtain
  \begin{align}
    \beta^{(k+1)}_{\infty}
    &= \sum_{s_1 \geq \cdots \geq s_{k+1} \geq 0} \frac{q^{s_{1}^2 + \cdots + s_{k+1}^2 + s_{k-r+1} + \cdots + s_{k+1}} q^{-s_1}(q^{s_1}+q^{-s_2}) \cdots (q^{s_{j-1}}+q^{-s_j})}{(q)_{\infty} (q)_{s_1 - s_{2}}  \cdots (q)_{s_{k} - s_{k+1}}} \,\beta_{s_{k+1}} \label{eq:lhs2}\\
    &= \frac{1-q}{(q)_{\infty}^2} \sum_{\ell=0}^{\infty} (1+q^{2\ell})^j q^{(k+1)\ell^2+(r-j+1)\ell} \left(\frac{\alpha_{\ell}}{1-q^{2\ell+1}}-\frac{q^{-2r\ell-1}\alpha_{\ell-1}}{1-q^{2\ell-1}}\right), \nonumber
  \end{align}
  where we recall that \(\alpha_{-\ell}=0\) for \(\ell>0\). Rearranging the last expression
  gives
  \begin{equation}
    \label{eq:rhs2}
    \beta^{(k+1)}_{\infty}= \frac{1}{(q)_{\infty}^2} \sum_{\ell=0}^{\infty}  q^{(k+1)\ell^2+(r-j+1)\ell} \frac{1-q}{1-q^{2\ell+1}} \left((1+q^{2\ell})^j - (1+q^{2\ell+2})^j q^{(k-r+1)(2\ell+1)-j}\right) \alpha_{\ell}.
  \end{equation}
  Equating \eqref{eq:lhs2} and \eqref{eq:rhs2} gives the desired
  result.
\end{proof}

\subsection{Proofs of Theorems~\ref{thm:Sta3.1}, \ref{thm:Sta4.1}, \ref{thm:binom_Kur}, and \ref{thm:binom_BGG}}

The proofs of all four of our main theorems will use Proposition
\ref{prop:common2}, but applied to different Bailey pairs (and some
modifications in \(k\) and \(r\)). We start by proving Theorem
\ref{thm:Sta3.1}.

\begin{proof}[Proof of Theorem \ref{thm:Sta3.1}]
  Apply Proposition \ref{prop:common2} to the unit Bailey pair for
  \(a=q\) (see Definition \ref{def:ubp}):
  \[\alpha_n=(-1)^nq^{\binom{n}{2}}\frac{1-q^{2n+1}}{1-q},\qquad \beta_n=\delta_{n,0}.\]
  This  yields
  \begin{align}\label{eq:proof1}
    \sum_{s_1 \geq \cdots \geq s_{k} \geq 0}
    & \frac{q^{s_{1}^2 + \cdots + s_{k}^2 + s_{k-r+1} + \cdots + s_{k}} q^{-s_1}(q^{s_1}+q^{-s_2}) \cdots (q^{s_{j-1}}+q^{-s_j})}{(q)_{s_1 - s_{2}}  \cdots (q)_{s_{k}}}\\
    &= \frac{1}{(q)_{\infty}} \sum_{\ell=0}^{\infty}  q^{(k+\frac{3}{2})\ell^2+(r-j+\frac{1}{2})\ell} (-1)^{\ell} \left((1+q^{2\ell})^j - (1+q^{2\ell+2})^j q^{(k-r+1)(2\ell+1)-j}\right)\nonumber.
  \end{align}
  Split the sum on the right-hand side into two parts. In the second
  sum, apply the change of variables \( \ell \mapsto -\ell -1 \).
  This transforms the expression into
  \[
    \frac{1}{(q)_{\infty}} \sum_{\ell \in \ZZ}  q^{(k+\frac{3}{2})\ell^2+(r-j+\frac{1}{2})\ell} (-1)^{\ell} (1+q^{2\ell})^j.
  \]
  Now apply the binomial theorem to expand \( (1+q^{2 \ell})^j \):
  \[
    \frac{1}{(q)_{\infty}} \sum_{s=0}^j \binom{j}{s} \sum_{\ell \in \ZZ} (-1)^{\ell} q^{(2k+3)\binom{\ell}{2} +(k+r-j+2s+2)\ell}.
  \]
  Finally, apply the Jacobi triple product identity~\eqref{eq:jtp} with
  \( q \mapsto q^{2k+3} \) and \( z = q^{k+2+r-j+2s} \), yielding
  \[
    \frac{1}{(q)_{\infty}} \sum_{s=0}^j \binom{j}{s} \frac{(q^{2k+3}, q^{k+1-r+j-2s}, q^{k+2+r-j+2s}; q^{2k+3})_\infty}{(q)_\infty}.
  \]
  Note that, in the first sum in \eqref{eq:proof1},
  \(q^{-s_1}(q^{s_1}+q^{-s_2}) \cdots (q^{s_{j-1}}+q^{-s_j})\) can be
  replaced by any product of \(j\) factors taken from
  \(\{q^{-s_1}\} \cup \{(q^{s_{i-1}}+q^{-s_i}) : 2 \leq i \leq
  k-r\}\). This completes the proof.
\end{proof}

We also prove Theorems \ref{thm:Sta4.1}, \ref{thm:binom_Kur}, and
\ref{thm:binom_BGG} in a similar way. We prove \Cref{thm:Sta4.1}
by applying our procedure not to the unit Bailey pair, but to the
Bailey pair (D'4) in~\Cref{def:D4}.

\begin{proof}[Proof of \Cref{thm:Sta4.1}]
  Consider the Bailey pair (D'4) in~\eqref{D4ubp} with \(a=q\):
\begin{equation}\label{D4ubpq}
    \alpha_n = (-1)^n q^{n^2} \frac{1-q^{2n+1}}{1-q}, \qquad \beta_n =
    \frac{1}{(q^2;q^2)_n},
\end{equation}
  and apply Proposition \ref{prop:common2} with $r\to r-1, k\to k-1$ to it. (Note that, as (D'4) is obtained from the unit Bailey pair by one instance of (D4), the ranges for $k,r,j$ are still valid.) This yields
  \begin{align*}
    \sum_{s_1 \geq \cdots \geq s_{k} \geq 0}
    & \frac{q^{s_{1}^2 + \cdots + s_{k}^2 + s_{k-r+1} + \cdots + s_{k}} q^{-s_1}(q^{s_1}+q^{-s_2}) \cdots (q^{s_{j-1}}+q^{-s_j})}{(q)_{s_1 - s_{2}}  \cdots (q)_{s_{k-1} - s_{k}} (q^2;q^2)_{s_k}} \\
    &= \frac{1}{(q)_{\infty}} \sum_{\ell=0}^{\infty}  q^{(k+1)\ell^2+(r-j)\ell}  (-1)^{\ell} \left((1+q^{2\ell})^j - (1+q^{2\ell+2})^j q^{(k-r+1)(2\ell+1)-j}\right)\\
    &= \frac{1}{(q)_{\infty}} \sum_{\ell \in \ZZ}  q^{(k+1)\ell^2+(r-j)\ell} (-1)^{\ell} (1+q^{2\ell})^j\\
    &= \frac{1}{(q)_{\infty}} \sum_{s=0}^j \binom{j}{s} \sum_{\ell \in \ZZ} (-1)^{\ell} q^{(2k+2)\binom{\ell}{2} +(k+r-j+2s+1)\ell}\\
    &= \frac{1}{(q)_{\infty}} \sum_{s=0}^j \binom{j}{s} \frac{(q^{k+1-r+j-2s}, q^{k+2+r-j+2s}, q^{2k+2} ; q^{2k+2})_\infty}{(q)_\infty},
  \end{align*}
  where again the penultimate equality follows from the binomial
  theorem, and the last one from the Jacobi triple
  product~\eqref{eq:jtp}. As before, in the first line, the product
  \(q^{-s_1}(q^{s_1}+q^{-s_2}) \cdots (q^{s_{j-1}}+q^{-s_j})\) may be
  replaced by any product of \(j\) factors chosen from
  \(\{q^{-s_1}\} \cup \{(q^{s_{i-1}}+q^{-s_i}) : 2 \leq i \leq
  k-r\}\).
\end{proof}

Now we give a proof of Theorem \ref{thm:binom_Kur} which comes from
the Bailey pair (D'1) in~\Cref{def:D1}.

\begin{proof}[Proof of Theorem \ref{thm:binom_Kur}]
  Consider the Bailey pair (D'1) in~\eqref{D1ubp} with \(a=q\):
\begin{equation}\label{D1ubpq}
\alpha_n = (-1)^n q^{n^2-n} \frac{1-q^{4n+2}}{1-q^2}, \qquad \beta_n = \frac{q^n}{(q^2;q^2)_n},
\end{equation}
  and apply Proposition \ref{prop:common2} with $r\to r-1, k\to k-1$ to it (again the ranges for $r,j,k$ are valid). This gives
  \begin{align*}
    \sum_{s_1 \geq \cdots \geq s_{k} \geq 0}
    & \frac{q^{s_{1}^2 + \cdots + s_{k}^2 + s_{k-r+1} + \cdots + 2s_{k}} q^{-s_1}(q^{s_1}+q^{-s_2}) \cdots (q^{s_{j-1}}+q^{-s_j})}{(q)_{s_1 - s_{2}}  \cdots (q)_{s_{k-1} - s_{k}} (q^2;q^2)_{s_k}} \\
    &= \frac{1}{(q)_{\infty}} \sum_{\ell=0}^{\infty}  q^{(k+1)\ell^2+(r-j-1)\ell} (-1)^{\ell} \left((1+q^{2\ell})^j - (1+q^{2\ell+2})^j q^{(k-r+1)(2\ell+1)-j}\right)\frac{1+q^{2\ell+1}}{1+q} \\
    &= \frac{1}{(1+q)(q)_{\infty}}  \sum_{\ell \in \ZZ}  q^{(k+1)\ell^2+(r-j-1)\ell} (-1)^{\ell} (1+q^{2\ell})^j (1+q^{2 \ell +1}) \\
    &= \frac{1}{(1+q)(q)_{\infty}}  \sum_{s=0}^j \binom{j}{s} \sum_{\ell \in \ZZ} \left[(-1)^{\ell} q^{(2k+2)\binom{\ell}{2} +(k+r-j+2s)\ell} + q(-1)^{\ell} q^{(2k+2)\binom{\ell}{2} +(k+r-j+2s+2)\ell} \right] \\
    &= \frac{1}{(1+q)(q)_{\infty}} \sum_{s=0}^j \binom{j}{s} \left[ (q^{k+2-r+j-2s}, q^{k+r-j+2s}, q^{2k+2} ; q^{2k+2})_\infty \right.\\ &\qquad \qquad \qquad \qquad \qquad \ \left.  + q (q^{k-r+j-2s}, q^{k+2+r-j+2s}, q^{2k+2} ; q^{2k+2})_\infty \right].
  \end{align*}
  As before, in the first line, the product
  \(q^{-s_1}(q^{s_1}+q^{-s_2}) \cdots (q^{s_{j-1}}+q^{-s_j})\) may be
  replaced by any product of \(j\) factors chosen from
  \(\{q^{-s_1}\} \cup \{(q^{s_{i-1}}+q^{-s_i}) : 2 \leq i \leq
  k-r\}\).
\end{proof}

We conclude with the proof of \Cref{thm:binom_BGG}.

\begin{proof}[Proof of \Cref{thm:binom_BGG}]
  Let \((\alpha_n,\,\beta_n)\) be a Bailey pair relative to \(q\).
  Apply \Cref{prop:common2}  with $r\to r-1, k\to k-1$ to the Bailey pair
  \( (\alpha_n', \beta_n') \) obtained from \Cref{lem:BL2} with
  \( a \) replaced by \( q \):
  \[
    \alpha'_n=\frac{(-1)^n (\rho)_nq^{\frac{n^2}{2} + 3\frac{n}{2}}}{(q^2/\rho)_n \rho^n}\,\alpha_n , \qand
    \beta'_n=\sum_{\ell=0}^n\frac{(-1)^\ell(\rho)_\ell q^{\frac{\ell^2}{2} + 3\frac{\ell}{2}}}{(q^2/\rho)_n (q)_{n-\ell} \rho^\ell}\,\beta_\ell.
  \]
  This gives
  \begin{multline*}
    \sum_{s_1 \geq \cdots \geq s_{k} \geq s_{k+1} \geq 0}  \frac{q^{s_{1}^2 + \cdots + s_{k}^2 + s_{k-r+1} + \cdots + s_{k}} q^{-s_1}(q^{s_1}+q^{-s_2}) \cdots (q^{s_{j-1}}+q^{-s_j})}{(q)_{s_1 - s_{2}}  \cdots (q)_{s_{k-1} - s_{k}} (q)_{s_k - s_{k+1}} (q^2/\rho)_{s_k}} \frac{(-1)^{s_{k+1}}(\rho)_{s_{k+1}} q^{\frac{s_{k+1}^2}{2} + 3\frac{s_{k+1}}{2}}}{\rho^{s_{k+1}}}\beta_{s_{k+1}}\\
    = \frac{1}{(q)_{\infty}} \sum_{\ell=0}^{\infty}  q^{k\ell^2+(r-j)\ell} \frac{1-q}{1-q^{2\ell+1}} \left((1+q^{2\ell})^j - (1+q^{2\ell+2})^j q^{(k-r+1)(2\ell+1)-j}\right) \frac{(\rho)_\ell (-1)^\ell q^{\frac{\ell^2}{2} + 3 \frac{\ell}{2}}}{(q^2/\rho)_\ell \rho^\ell}\alpha_{\ell},
  \end{multline*}
  with the usual ranges for $r,j,k$.   Now take \( (\alpha_n, \beta_n) \) to be the unit Bailey
  pair~\eqref{ubp} with \( a = q \):
  \[
    \alpha_n=(-1)^nq^{\binom{n}{2}}\frac{1-q^{2n+1}}{1-q}, \qand \beta_n=\delta_{n,0}.
  \]
  Inserting this in the above equation yields
  \begin{multline*}
    \sum_{s_1 \geq \cdots \geq s_{k} \geq 0}  \frac{q^{s_{1}^2 + \cdots + s_{k}^2 + s_{k-r+1} + \cdots + s_{k}} q^{-s_1}(q^{s_1}+q^{-s_2}) \cdots (q^{s_{j-1}}+q^{-s_j})}{(q)_{s_1 - s_{2}}  \cdots (q)_{s_{k-1} - s_{k}} (q)_{s_k} (q^2/\rho)_{s_k}} \\
    = \frac{1}{(q)_{\infty}} \sum_{\ell=0}^{\infty}  q^{k\ell^2+(r-j)\ell}  \left((1+q^{2\ell})^j - (1+q^{2\ell+2})^j q^{(k-r+1)(2\ell+1)-j}\right) \frac{(\rho)_\ell q^{ \ell^2 + \ell}}{(q^2/\rho)_\ell \rho^\ell}.
  \end{multline*}
  Now set \( \rho = -q^{3/2} \), and then replace \( q \) by
  \( q^2 \). This gives
  \begin{align*}
    \sum_{s_1 \geq \cdots \geq s_{k} \geq 0}
    &\frac{q^{2(s_{1}^2 + \cdots + s_{k}^2 + s_{k-r+1} + \cdots + s_{k})} q^{-2s_1}(q^{2s_1}+q^{-2s_2}) \cdots (q^{2s_{j-1}}+q^{-2s_j})}{(q^2;q^2)_{s_1 - s_{2}}  \cdots (q^2;q^2)_{s_{k-1} - s_{k}} (q^2;q^2)_{s_k} (-q;q^2)_{s_k}} \\
    &= \frac{1}{(q^2;q^2)_{\infty}} \sum_{\ell=0}^{\infty}  q^{(2k+2)\ell^2+(2r-2j-1)\ell}  \left((1+q^{4\ell})^j - (1+q^{4\ell+4})^j q^{2(k-r+1)(2\ell+1)-2j}\right) (-1)^\ell \frac{1+q^{2 \ell +1}}{1+q } \\
    &= \frac{1}{(1+q)(q^2;q^2)_{\infty}} \sum_{\ell\in \ZZ} q^{(2k+2)\ell^2+(2r-2j-1)\ell}  (1+q^{4\ell})^j (-1)^\ell (1+q^{2 \ell +1}) \\
    &= \frac{1}{(1+q)(q^2;q^2)_{\infty}} \sum_{s = 0}^{j} \binom{j}{s} \sum_{\ell\in \ZZ} \left[ (-1)^\ell  q^{(4k+4)\binom{\ell}{2}+(2k+ 2r-2j+4s+1)\ell} + q (-1)^\ell  q^{(4k+4)\binom{\ell}{2}+(2k + 2r-2j+4s+3)\ell} \right] \\
    &= \frac{1}{(1+q)(q^2;q^2)_{\infty}} \sum_{s = 0}^{j} \binom{j}{s} \Big[ (q^{4k+4}, q^{2k+2r-2j+4s+1}, q^{2k-2r+2j-4s+3}; q^{4k+4})_\infty  \\
    & \qquad\qquad\qquad\qquad\qquad\qquad  + q (q^{4k+4}, q^{2k+2r-2j+4s+3}, q^{2k-2r+2j-4s+1}; q^{4k+4})_\infty \Big].
  \end{align*}
  Multiplying both sides by \( (-q;q^2)_\infty \) completes the proof.
  Here, the product
  \(q^{-2s_1}(q^{2s_1}+q^{-2s_2}) \cdots (q^{2s_{j-1}}+q^{-2s_j})\)
  may be replaced by any product of \(j\) factors chosen from
  \(\{q^{-2s_1}\} \cup \{(q^{2s_{i-1}}+q^{-2s_i}) : 2 \leq i \leq
  k-r\}\).
\end{proof}

\section{Proofs of non-binomial extensions}\label{sec:pf_non-binom}

\subsection{Proofs of Theorems~\ref{thm:Sta3.2}, \ref{thm:Sta4.2}, and \ref{thm:new_Kur}}
\label{sec:proofs-theorems-1.6}

Recall that to get Theorem~\ref{thm:corobaileylattice}, one uses once
the Bailey lattice after a few iterations of the Bailey lemma
(actually the special instances given in Lemmas~\ref{lem:BL}
and~\ref{lem:Baileylattice}). The clue here is to use twice the Bailey
lattice instead of once. More precisely, iterate \(r+1\) times
\Cref{lem:BL}, then use Lemma~\ref{lem:Baileylattice}. Next, iterate
\(k-r-j-1\) times \Cref{lem:BL} with \(a\) replaced by \(a/q\) and use
once again Lemma~\ref{lem:Baileylattice}. Finally, iterate \(j-1\)
times \Cref{lem:BL} with \(a\) replaced by \(a/q^2\). (If $j=0$, the
two last steps have to be omitted. If $r+j=k$, the second and third
steps have to be omitted, and $a$ needs to be replaced by $aq$ in the
remaining steps.) Letting \( n \to \infty \) in the
relation~\eqref{bp} for the resulting Bailey pair
\( (\alpha_n, \beta_n) \) relative to \( a/q^2 \) gives the following.

\begin{prop}[New consequence of the Bailey lattice]\label{prop:coro2baileylattice}
  Let \((\alpha_n,\,\beta_n)\) be a Bailey pair relative to \(a\).
  Then, for all integers \(j,r \geq 0\) and \(k\geq1\) such that
  \(r+j\leq k\), we have
  \begin{multline}\label{coro2lattice}
    \sum_{s_1\geq\dots\geq s_{k+1}\geq0}\frac{a^{s_1+\dots+s_{k+1}}q^{s_1^2+\dots+s_{k+1}^2-2s_1-\dots-2s_{j}-s_{j+1}-\dots-s_{k-r}}}{(q)_{s_1-s_2}\dots(q)_{s_{k}-s_{k+1}}}\beta_{s_{k+1}} \\
    =\frac{1}{(aq)_\infty}\sum_{\ell\geq0}a^{(k+1)\ell}q^{(k+1)\ell^2+(r-j-k)\ell} \frac{1}{1-aq^{2\ell}}\left(\frac{1-(aq^{2\ell-1})^{j+1}}{1-aq^{2\ell-1}}-a^{k+1-r}q^{(2k+2-2r)\ell-j}\frac{1-(aq^{2\ell+1})^{j+1}}{1-aq^{2\ell+1}}\right)\alpha_\ell.
  \end{multline}
\end{prop}

We now prove Theorems~\ref{thm:Sta3.2}, \ref{thm:Sta4.2}, and
\ref{thm:new_Kur} using \Cref{prop:coro2baileylattice}.

\begin{proof}[Proof of \Cref{thm:Sta3.2}]
  We apply \Cref{prop:coro2baileylattice}. Inserting the unit Bailey
  pair~\eqref{ubp} in~\eqref{coro2lattice} gives
\begin{multline}\label{corolatticeubp}
\sum_{s_1\geq\dots\geq s_{k}\geq0}\frac{a^{s_1+\dots+s_{k}}q^{s_1^2+\dots+s_{k}^2-2s_1-\dots-2s_{j}-s_{j+1}-\dots-s_{k-r}}}{(q)_{s_1-s_2}\dots(q)_{s_{k-1}-s_{k}}(q)_{s_{k}}}=\frac{1}{(a)_\infty}\sum_{\ell\geq0}(-1)^\ell a^{(k+1)\ell}q^{(k+1)\ell^2+\binom{\ell}{2}+(r-j-k)\ell}\\
\times\frac{(a)_\ell}{(q)_\ell}\left(\frac{1-(aq^{2\ell-1})^{j+1}}{1-aq^{2\ell-1}}-a^{k+1-r}q^{(2k+2-2r)\ell-j}\frac{1-(aq^{2\ell+1})^{j+1}}{1-aq^{2\ell+1}}\right).
\end{multline}
Now take \(a=q\). Then the left-hand side corresponds to that
of~\eqref{eq:sAG}, while the right-hand side becomes
\[
\frac{1}{(q)_\infty}\sum_{\ell\geq0}(-1)^\ell q^{(k+\frac{3}{2})\ell^2+(r-j+\frac{1}{2})\ell}\left(\frac{1-q^{2\ell(j+1)}}{1-q^{2\ell}}-q^{(2k+2-2r)\ell+k+1-r-j}\frac{1-q^{(2\ell+2)(j+1)}}{1-q^{2\ell+2}}\right).
\]
Split the sum into two parts, and apply the change of variables
\(\ell\mapsto-\ell-1\) in the second sum. This yields
\[
\frac{1}{(q)_\infty}\sum_{\ell\in\mathbb{Z}}(-1)^\ell q^{(k+\frac{3}{2})\ell^2+(r-j+\frac{1}{2})\ell}\frac{1-q^{2\ell(j+1)}}{1-q^{2\ell}}.
\]
Expanding the quotient yields
\[
\frac{1}{(q)_\infty}\sum_{s=0}^j\sum_{\ell\in\mathbb{Z}}(-1)^\ell q^{(2k+3)\binom{\ell}{2}}q^{(k+r-j+2s+2)\ell}.
\]
Finally, applying the Jacobi triple product identity~\eqref{eq:jtp} with
\(q\mapsto q^{2k+3}\) and \(z=q^{k+r-j+2s+2}\), we obtain the right-hand
side of~\eqref{eq:sAG}.
\end{proof}

We give a similar proof of Theorem~\ref{thm:Sta4.2}, starting this time from the Bailey pair (D'4).

\begin{proof}[Proof of \Cref{thm:Sta4.2}]
  We now apply~\Cref{prop:coro2baileylattice} with $a=q$, $k\to k-1$ and $r\to r-1$ to the Bailey pair~\eqref{D4ubpq}.  Note that the ranges $r\geq0$, $k\geq1$ and $r+j\leq k$ are still valid, as the above Bailey pair is obtained by using one instance of the change of base~\cite[(D4)]{BIS} to the unit Bailey pair.   This gives
\begin{multline*}
\sum_{s_1\geq\dots\geq s_{k}\geq0}\frac{q^{s_1^2+\dots+s_{k}^2-s_1-\dots-s_j+s_{k-r+1}+\dots+s_{k}}}{(q)_{s_1-s_2}\dots(q)_{s_{k-1}-s_{k}}(q^2;q^2)_{s_{k}}}=\frac{1}{(q)_\infty}\sum_{\ell\geq0}(-1)^\ell q^{(k+1)\ell^2+(r-j)\ell}\\
\times\left(\frac{1-q^{2\ell(j+1)}}{1-q^{2\ell}}-q^{(2k+2-2r)\ell+k+1-r-j}\frac{1-q^{(2\ell+2)(j+1)}}{1-q^{2\ell+2}}\right).
\end{multline*}
The left-hand side is the one of~\eqref{eq:sBr}. On the right-hand side, split the sum over \(\ell\) into two sums, and shift \(\ell\mapsto-\ell-1\) in the second one to get
\[
\frac{1}{(q)_\infty}\sum_{\ell\in\mathbb{Z}}(-1)^\ell q^{(k+1)\ell^2+(r-j)\ell}\frac{1-q^{2\ell(j+1)}}{1-q^{2\ell}},
\]
which by expanding the quotient yields
\[
\frac{1}{(q)_\infty}\sum_{s=0}^j\sum_{\ell\in\mathbb{Z}}(-1)^\ell q^{(2k+2)\binom{\ell}{2}}q^{(k+r+1-j+2s)\ell}.
\]
We get the right-hand side of~\eqref{eq:sAG} by applying the Jacobi
triple product identity~\eqref{eq:jtp} with \(q\mapsto q^{2k+2}\) and
\(z=q^{k+r+1-j+2s}\).

\end{proof}

We finally prove Theorem~\ref{thm:new_Kur}, starting this time from the Bailey pair (D'1).

\begin{proof}[Proof of \Cref{thm:new_Kur}]
  We now apply~\Cref{prop:coro2baileylattice}  with $a=q$, $k\to k-1$ and $r\to r-1$ to the Bailey pair~\eqref{D1ubpq}.
Again the ranges $r\geq0$, $k\geq1$ and $r+j\leq k$ are still valid, and this gives
\begin{multline*}
\sum_{s_1\geq\dots\geq s_{k}\geq0}\frac{q^{s_1^2+\dots+s_{k}^2-s_1-\dots-s_j+s_{k-r+1}+\dots+s_{k-1}+2s_{k}}}{(q)_{s_1-s_2}\dots(q)_{s_{k-1}-s_{k}}(q^2;q^2)_{s_{k}}}=\frac{1}{(q)_\infty}\sum_{\ell\geq0}(-1)^\ell q^{(k+1)\ell^2+(r-j-1)\ell}\frac{1+q^{2\ell+1}}{1+q}\\
\times\left(\frac{1-q^{2\ell(j+1)}}{1-q^{2\ell}}-q^{(2k+2-2r)\ell+k+1-r-j}\frac{1-q^{(2\ell+2)(j+1)}}{1-q^{2\ell+2}}\right).
\end{multline*}
The left-hand side becomes the one of~\eqref{eq:new_Kur}. The right-hand side can be written as
\begin{multline*}
\frac{1}{(1+q)(q)_\infty}\sum_{\ell\geq0}(-1)^\ell q^{(k+1)\ell^2+(r-j-1)\ell}\left(\frac{1-q^{2\ell(j+1)}}{1-q^{2\ell}}-q^{(2k+4-2r)\ell+k+2-r-j}\frac{1-q^{(2\ell+2)(j+1)}}{1-q^{2\ell+2}}\right)\\
+\frac{1}{(1+q)(q)_\infty}\sum_{\ell\geq0}(-1)^\ell q^{(k+1)\ell^2+(r-j+1)\ell+1}\left(\frac{1-q^{2\ell(j+1)}}{1-q^{2\ell}}-q^{(2k-2r)\ell+k-r-j}\frac{1-q^{(2\ell+2)(j+1)}}{1-q^{2\ell+2}}\right).
\end{multline*}
Split each sum and shift \(\ell\mapsto-\ell-1\) in the second and fourth parts.
This gives
\[
\frac{1}{(1+q)(q)_\infty}\left(\sum_{\ell\in\mathbb{Z}}(-1)^\ell q^{(k+1)\ell^2+(r-j-1)\ell}\frac{1-q^{2\ell(j+1)}}{1-q^{2\ell}}+ q \sum_{\ell\in\mathbb{Z}}(-1)^\ell q^{(k+1)\ell^2+(r-j+1)\ell}\frac{1-q^{2\ell(j+1)}}{1-q^{2\ell}}\right).
\]
Expanding the quotients gives
\[
\frac{1}{(1+q)(q)_\infty}\left(\sum_{s=0}^j\sum_{\ell\in\mathbb{Z}}(-1)^\ell q^{(2k+2)\binom{\ell}{2}}q^{(k+r-j+2s)\ell}+ q \sum_{s=0}^j\sum_{\ell\in\mathbb{Z}}(-1)^\ell q^{(2k+2)\binom{\ell}{2}}q^{(k+r-j+2s+2)\ell}\right).
\]
Applying the Jacobi triple product identity~\eqref{eq:jtp} with
\(q\mapsto q^{2k+2}\) and \(z=q^{k+r-j+2s}\) or \(z=q^{k+r-j+2s+2}\)
yields the right-hand side of~\eqref{eq:sAG}.

\end{proof}

\subsection{Non-binomial extension of the Bresssoud--G\"ollnitz--Gordon identities}\label{sec:GG}

Inspired by the methods in~\cite{DJK25}, we see that an extension of \Cref{prop:coro2baileylattice} is needed, through the following process (the method is the same as for \Cref{prop:coro2baileylattice}, except that at two appropriate steps we keep one finite parameter \(\rho\)) (therefore using Lemma~\ref{lem:BL2} instead of Lemma~\ref{lem:BL}). 

First use Lemma~\ref{lem:BL2} with $\rho=b$ and iterate \(r\) times
Lemma~\ref{lem:BL}. Then use Lemma~\ref{lem:Baileylattice}. Next
iterate \(k-r-j-1\) times Lemma~\ref{lem:BL} with \(a\) replaced by
\(a/q\) and use once again Lemma~\ref{lem:Baileylattice}. Iterate
\(j-2\) times Lemma~\ref{lem:BL} with \(a\) replaced by \(a/q^2\), and
finally once Lemma~\ref{lem:BL2} with \(a\) replaced by \(a/q^2\) and
$\rho=c$. As for Proposition~\ref{prop:coro2baileylattice}, one has to
examine what we should do when the number of steps described above is
negative: If $j=0$ and $r=k$, we do not use any lattice in the $k+1$
steps of the process. If $j=0$ and $r<k$ or if $j=1$, we only use one
lattice. Note that for some extremal cases, one has to replace
Lemma~\ref{lem:BL2} in the first or last step by the use of the Bailey
lattice in Theorem~\ref{thm:baileylattice} in which $\sigma\to\infty$
while we keep $\rho$. Again, letting \( n \to \infty \) in the
relation~\eqref{bp} for the resulting Bailey pair
\( (\alpha_n, \beta_n) \) relative to \( a/q^2 \) gives the following.

\begin{prop}[Another new consequence of the Bailey lattice]\label{prop:coro3baileylattice}
  If \((\alpha_n,\,\beta_n)\) is a Bailey pair relative to \(a\), then for all integers \(k\geq1\) and \(0\leq r,j\leq k\) with \(r+j\leq k\), we have
\begin{multline}\label{coro3lattice}
\sum_{s_1\geq\dots\geq s_{k+1}\geq0}\frac{(-1)^{s_1+s_{k+1}}(b)_{s_1}(c)_{s_{k+1}}}{b^{s_1}c^{s_{k+1}}(aq/c)_{s_k}}\frac{a^{s_1+\dots+s_{k+1}}q^{\frac{s_1^2}{2}+s_2^2+\dots+s_{k}^2+\frac{s_{k+1}^2}{2}+\frac{s_1}{2}-2s_1-\dots-2s_{j}-s_{j+1}-\dots-s_{k-r}+\frac{s_{k+1}}{2}}}{(q)_{s_1-s_2}\dots(q)_{s_{k}-s_{k+1}}}\beta_{s_{k+1}}\\
=\frac{(a/bq)_\infty}{(aq)_\infty}\sum_{\ell\geq0}\frac{a^{(k+1)\ell}}{(bc)^\ell}q^{k\ell^2+(r+1-j-k)\ell}\frac{(b,c)_\ell}{(a/bq,aq/c)_\ell}\frac{1}{1-aq^{2\ell}}\\
\times\left(\frac{1+a^{j+1}q^{(2j+1)\ell-j-1}\frac{1-bq^\ell}{b-aq^{\ell-1}}}{1-aq^{2\ell-1}}+a^{k+1-r}q^{(2k+1-2r)\ell-j}\frac{1-bq^\ell}{b-aq^{\ell-1}}\frac{1+a^{j+1}q^{(2j+1)\ell+j}\frac{1-bq^{\ell+1}}{b-aq^{\ell}}}{1-aq^{2\ell+1}}\right)\alpha_\ell.
\end{multline}
\end{prop}

Applying~\eqref{coro3lattice} to the unit Bailey pair~\eqref{ubp}, we get
\begin{multline}\label{coro3latticeubp}
\sum_{s_1\geq\dots\geq s_{k}\geq0}\frac{(-1)^{s_1}(b)_{s_1}}{b^{s_1}(aq/c)_{s_k}}\frac{a^{s_1+\dots+s_{k}}q^{\frac{s_1^2}{2}+s_2^2+\dots+s_{k}^2+\frac{s_1}{2}-2s_1-\dots-2s_{j}-s_{j+1}-\dots-s_{k-r}}}{(q)_{s_1-s_2}\dots(q)_{s_{k-1}-s_{k}}(q)_{s_k}}\\
=\frac{(a/bq)_\infty}{(a)_\infty}\sum_{\ell\geq0}(-1)^\ell\frac{a^{(k+1)\ell}}{(bc)^\ell}q^{(k+\frac{1}{2})\ell^2+(r-j-k+\frac{1}{2})\ell}\frac{(a,b,c)_\ell}{(q,a/bq,aq/c)_\ell}\\
\times\left(\frac{1+a^{j+1}q^{(2j+1)\ell-j-1}\frac{1-bq^\ell}{b-aq^{\ell-1}}}{1-aq^{2\ell-1}}+a^{k+1-r}q^{(2k+1-2r)\ell-j}\frac{1-bq^\ell}{b-aq^{\ell-1}}\frac{1+a^{j+1}q^{(2j+1)\ell+j}\frac{1-bq^{\ell+1}}{b-aq^{\ell}}}{1-aq^{2\ell+1}}\right).
\end{multline}

As for~\eqref{corolatticeubp}, the appropriate choice for \(a\) to derive Stanton type formulas is \(a=q\) (this is clear by inspecting the powers of \(a\) and \(q\) on the left-hand side of~\eqref{coro3latticeubp}). Note that the alternative classical choice is usually \(a=1\), which creates problems of convergence here. When \(a=q\), we get
\begin{multline}\label{coro3latticeubpa=q}
\sum_{s_1\geq\dots\geq s_{k}\geq0}\frac{(-1)^{s_1}(b)_{s_1}}{b^{s_1}(q^2/c)_{s_k}}\frac{q^{\frac{s_1^2}{2}+s_2^2+\dots+s_{k}^2+\frac{s_1}{2}-s_1-\dots-s_{j}+s_{k-r+1}+\dots+s_{k}}}{(q)_{s_1-s_2}\dots(q)_{s_{k-1}-s_{k}}(q)_{s_k}}\\
=\frac{(1/b)_\infty}{(q)_\infty}\sum_{\ell\geq0}\frac{(-1)^\ell}{(bc)^\ell}q^{(k+\frac{1}{2})\ell^2+(r-j+\frac{3}{2})\ell}\frac{(b,c)_\ell}{(1/b,q^2/c)_\ell}\\
\times\left(\frac{1+q^{(2j+1)\ell}\frac{1-bq^\ell}{b-q^{\ell}}}{1-q^{2\ell}}+q^{(2k+1-2r)\ell+k+1-r-j}\frac{1-bq^\ell}{b-q^{\ell}}\frac{1+q^{(2j+1)(\ell+1)}\frac{1-bq^{\ell+1}}{b-q^{\ell+1}}}{1-q^{2\ell+2}}\right).
\end{multline}

As mentioned in the introduction, \(m\)-versions of Bressoud's extensions~\cite[(3.6)--(3.9)]{Br80} of the G\"ollnitz--Gordon identities are derived in~\cite{DJK25}. The method uses a bilateral version of a simpler case of \Cref{prop:coro3baileylattice} (in which only one instance of the Bailey lattice is used instead of two). As a result, it is shown that, as for the cases of the Andrews--Gordon and Bressoud identities, all results come in pairs,  arising from the two choices \(a=1\) and \(a=q\). Surprisingly, it is for instance noticed that~\cite[(3.6)]{Br80} (which extends one of the G\"ollnitz--Gordon identities) arises in pair with Bressoud's identity~\eqref{eq:Br}, while~\cite[(3.7)]{Br80} (which extends another of the G\"ollnitz--Gordon identities) arises in pair with Bressoud's identity~\eqref{eq:Br3.5}. Similarly,~\cite[(3.8)]{Br80} (resp.~\cite[(3.9)]{Br80} arises in pair with a new formula expressed in~\cite[Corollary 2.10]{DJK25} (resp.~\cite[Corollary 2.12]{DJK25}.

As can be seen in the cases of the Andrews--Gordon and Bressoud
formulas, their extensions discovered by Stanton now embed both
choices \(a=1\) and \(a=q\) in formulas involving two integral
parameters \(j\) and \( r \), instead of only one parameter as usual.
As seen in the previous sections, these formulas can be derived from
the Bailey lattice with only one choice \(a=q\), but one needs to use
twice the Bailey lattice instead of once.

Our goal is to do the same here for Bressoud's extensions of the G\"ollnitz--Gordon identities, by using \Cref{prop:coro3baileylattice}. Nevertheless, as explained above, the only reasonable choice for \(a\) is \(q\), resulting in formula~\eqref{coro3latticeubpa=q}.

 The first specialisation used in~\cite{DJK25}, namely \(b\to\infty,c=-q\), applied to~\eqref{coro3latticeubpa=q}, yields Stanton's formula~\eqref{eq:sBr}. The second specialisation in~\cite{DJK25} is \(b\to\infty,c=-q^{1/2}\), which does not seem to yield interesting formulas. Alternatively, we are able to prove \Cref{thm:GGrj}.

\begin{proof}[Proof of \Cref{thm:GGrj}]
  Take \(b\to\infty,c=-q^{3/2}\) in~\eqref{coro3latticeubpa=q},
  replace \(q\) by \(q^2\) and multiply both sides by
  \((-q;q^2)_\infty\). Then the left-hand side becomes the desired
  expression. On the right-hand side, we obtain
\[
  \frac{(-q^3;q^2)_\infty}{(q^2;q^2)_\infty}
\sum_{\ell\geq0}(-1)^\ell q^{(2k+2)\ell^2+(2r-2j-1)\ell}(1+q^{2\ell+1})\left(\frac{1-q^{4\ell(j+1)}}{1-q^{4\ell}}-\frac{1-q^{(2j+1)(\ell+1)}}{1-q^{2\ell+2}}q^{(2k+1-2r)\ell+k+1-r-j}\right).
\]
Split the sum and apply the change of variables
\( \ell \mapsto - \ell -1 \) in the second sum. This gives
\[
  \frac{(-q^3;q^2)_\infty}{(q^2;q^2)_\infty}
\sum_{\ell\in\mathbb{Z}}(-1)^\ell q^{(2k+2)\ell^2+(2r-2j-1)\ell}(1+q^{2\ell+1})\frac{1-q^{4\ell(j+1)}}{1-q^{4\ell}}.
\]
Expanding the quotient yields
\[
  \frac{(-q^3;q^2)_\infty}{(q^2;q^2)_\infty}
\sum_{s=0}^j\sum_{\ell\in\mathbb{Z}}(-1)^\ell q^{(4k+4)\binom{\ell}{2}+(2k+2r-2j+4s+1)\ell}(1+q^{2\ell+1}).
\]
Splitting the sum over \(\ell\) into two sums and applying twice the Jacobi
triple product identity~\eqref{eq:jtp} gives the desired result.
\end{proof}

Now we turn to the two last specialisations used in~\cite{DJK25} to
extend~\cite[(3.8)--(3.9)]{Br80}, namely \(b=-q^{1/2},c\to\infty\) and
\(b=-q^{1/2},c=-q\), respectively. In view
of~\eqref{coro3latticeubpa=q}, it seems hopeless to handle the
right-hand side nicely with these specialisations. Alternatively, the
choice \(b=-1\) gives the next two results.

\begin{thm}\label{thm:newSlater}
  Let \( j,r \geq 0 \) and \( k \geq 1 \) be integers such that
  \( j+r \leq k \). Then
\begin{multline}\label{newslater}
\sum_{s_1\geq\dots\geq s_{k}\geq0}\frac{q^{\frac{s_1^2}{2}+s_2^2+\dots+s_{k}^2+\frac{s_1}{2}-s_1-\dots-s_{j}+s_{k-r+1}+\dots+s_{k}}(-1)_{s_1}}{(q)_{s_1-s_2}\dots(q)_{s_{k-1}-s_{k}}(q)_{s_k}}\\
=\frac{(-q)_\infty}{(q)_\infty}\sum_{s=0}^{2j}\sum_{t=0}^{2r}(-1)^t(q^{2k+2},q^{k+1-r-j+s+t},q^{k+1+r+j-s-t};q^{2k+2})_\infty.
\end{multline}
\end{thm}

\begin{proof}
  Take \(b=-1,c\to\infty\) in~\eqref{coro3latticeubpa=q}. We get the
  desired left-hand side, while the right-hand side becomes
\[
\frac{(-1)_\infty}{(q)_\infty}\sum_{\ell\geq0}(-1)^\ell q^{(k+1)\ell^2+(r-j+1)\ell}\left(\frac{1-q^{(2j+1)\ell}}{1-q^{2\ell}}-\frac{1-q^{(2j+1)(\ell+1)}}{1-q^{2\ell+2}}q^{(2k+1-2r)\ell+k+1-r-j}\right).
\] 
Split the sum into two parts, and in the second one, apply the change
of variables \(\ell\mapsto -\ell-1\). This transforms the sum into
\[
\frac{(-1)_\infty}{(q)_\infty}\sum_{\ell\in\mathbb{Z}}(-1)^\ell q^{(k+1)\ell^2+(r-j+1)\ell}\frac{1-q^{(2j+1)\ell}}{1-q^{2\ell}}.
\] 
Now apply the substitution \( \ell \mapsto - \ell \) to this expression, obtaining
\[
  \frac{(-1)_\infty}{(q)_\infty}\sum_{\ell\in\mathbb{Z}}(-1)^\ell
  q^{(k+1)\ell^2+(r-j+1)\ell}\frac{1-q^{(2j+1)\ell}}{1-q^{2\ell}}q^{-(2r+1)\ell}.
\]
Adding the two above expressions and dividing by \(2\) gives
\[
\frac{(-q)_\infty}{(q)_\infty}\sum_{\ell\in\mathbb{Z}}(-1)^\ell q^{(2k+2)\binom{\ell}{2}+(k+1-r-j)\ell}\frac{1-q^{(2j+1)\ell}}{1-q^{\ell}}\frac{1+q^{(2r+1)\ell}}{1+q^{\ell}}.
\] 
Expanding the quotients and applying the Jacobi triple product
identitiy~\eqref{eq:jtp} completes the proof.
\end{proof}

For \(r=0\), Formula~\eqref{newslater} is~\cite[Corollary 2.10]{DJK25} with \((r,i)\to (k+1,j)\). When \(k=1\), the three possible cases \((j,r)=(0,0)\), \((1,0)\) and \((0,1)\) are equivalent to Slater's identities~\cite[(8), (12), (13)]{Sl}.

\begin{thm}\label{thm:newSlater2}
  Let \( j,r \geq 0 \) and \( k \geq 1 \) be integers such that
  \( j+r \leq k \). If \(r\geq1\), then we have
\begin{multline}\label{newslater2}
(1+q^{1/2})\sum_{s_1\geq\dots\geq s_{k}\geq0}\frac{q^{\frac{s_1^2}{2}+s_2^2+\dots+s_{k}^2+\frac{s_1}{2}-s_1-\dots-s_{j}+s_{k-r+1}+\dots+s_{k}}(-1)_{s_1}}{(q)_{s_1-s_2}\dots(q)_{s_{k-1}-s_{k}}(q,-q^{1/2})_{s_k}}\\
=\frac{(-q)_\infty}{(q)_\infty}\bigg(\sum_{s=0}^{2j}\sum_{t=0}^{2r-2}(-1)^t(q^{2k+1},q^{k+3/2-r-j+s+t},q^{k-1/2+r+j-s-t};q^{2k+1})_\infty\\
+q^{1/2}\sum_{s=0}^{2j}\sum_{t=0}^{2r}(-1)^t(q^{2k+1},q^{k+1/2-r-j+s+t},q^{k+1/2+r+j-s-t};q^{2k+1})_\infty\bigg),
\end{multline}
and if \(r=0\), then we get~\cite[Corollary~2.12]{DJK25} with
\((r,i)\to (k+1,j)\).
\end{thm}

\begin{proof}
Take \(b=-1,c=-q^{3/2}\) in~\eqref{coro3latticeubpa=q} and multiply both sides by \(1+q^{1/2}\), then we get the desired left-hand side, while the right-hand side is
\[
\frac{(-1)_\infty}{(q)_\infty}\sum_{\ell\geq0}(-1)^\ell q^{(k+1/2)\ell^2+(r-j)\ell}(1+q^{\ell+1/2})\left(\frac{1-q^{(2j+1)\ell}}{1-q^{2\ell}}-\frac{1-q^{(2j+1)(\ell+1)}}{1-q^{2\ell+2}}q^{(2k+1-2r)\ell+k+1-r-j}\right).
\] 
Split the sum into two parts, and in the second sum apply the
substitution \(\ell\mapsto -\ell-1\). This yields
\[
\frac{(-1)_\infty}{(q)_\infty}\left(\sum_{\ell\in\mathbb{Z}}(-1)^\ell q^{(k+1/2)\ell^2+(r-j)\ell}\frac{1-q^{(2j+1)\ell}}{1-q^{2\ell}}+q^{1/2}\sum_{\ell\in\mathbb{Z}}(-1)^\ell q^{(k+1/2)\ell^2+(r-j+1)\ell}\frac{1-q^{(2j+1)\ell}}{1-q^{2\ell}}\right).
\] 
Now replace \( \ell \mapsto - \ell \) in each sum, add to the original
expressions, and divide by \( 2 \). We obtain
\begin{multline*}
\frac{(-q)_\infty}{(q)_\infty}\bigg(\sum_{\ell\in\mathbb{Z}}(-1)^\ell q^{(2k+1)\binom{\ell}{2}+(k-r-j+3/2)\ell}\frac{1-q^{(2j+1)\ell}}{1-q^{\ell}}\frac{1+q^{(2r-1)\ell}}{1+q^\ell}\\
+q^{1/2}\sum_{\ell\in\mathbb{Z}}(-1)^\ell q^{(2k+1)\binom{\ell}{2}+(k-r-j+1/2)\ell}\frac{1-q^{(2j+1)\ell}}{1-q^{\ell}}\frac{1+q^{(2r+1)\ell}}{1+q^\ell}\bigg).
\end{multline*}
If \(r=0\), the two sums are the same, so one can extract a factor
\(1+q^{1/2}\). Expanding both quotients and applying the Jacobi triple
product identity~\eqref{eq:jtp} yields~\cite[Corollary~2.12]{DJK25}
with \((r,i)\to (k+1,j)\). If \(r\geq1\), expanding the four quotients
and applying~\eqref{eq:jtp} twice gives the result.
\end{proof}

\section{Insertion map for multipartitions revisited}
\label{sec:insert-map-mult}

In this section we revisit the particle motion bijection, which was
first used in~\cite{Warnaar1997}, and generalised recently
in~\cite{DJK}, to make it more systematic and easier to apply.
Warnaar's approach had a simpler combinatorial description but was
less general, while the approach of~\cite{DJK} was more general and
gave rise to useful explicit formulas but was less intuitive to
understand. Our goal here is to keep the best of both worlds by
reformulating the approach of~\cite{DJK} in the style of Warnaar,
which leads both to a simple combinatorial description and explicit
formulas.

In the next sections, we will use this bijection to give
partition-theoretic interpretations of Stanton's non-binomial identities and prove Theorems \ref{thm:Sta3.2}, \ref{thm:Sta4.2} and \ref{thm:new_Kur} by combining it with the Andrews--Gordon and Bressoud identities.

\medskip

Recall from the introduction that a \emph{partition} \( \lambda = (\lambda_1,\dots,\lambda_\ell) \) is a
weakly decreasing finite sequence of non-negative integers, that is,
\( \lambda_1 \geq \lambda_2 \ge \cdots \ge \lambda_\ell \geq 0 \).
Each non-negative integer \( \lambda_i \) is called a part of
\( \lambda \) and \( \ell(\lambda) = \ell \) is called the length of
\( \lambda \). A \emph{frequency sequence} \( (f_i)_{i \geq 0} \) is a
sequence of non-negative integers. Given a partition, its frequency
sequence is defined by setting \( f_i \) to be the number of parts of
size \( i \) for all $i \geq 0$. This yields a one-to-one correspondence between
frequency sequences and partitions. The \emph{size} and \emph{length}
of a frequency sequence \( (f_i)_{i \geq 0} \) are defined as
\( |f| := \sum_{i \geq 0} if_i \) and
\( \ell(f) := \sum_{i \geq 0} f_i \), respectively, which coincide
with the size and length of the corresponding partition.

From now on, we use the notation of a multipartition, which simply
refers to a finite sequence of partitions. For an integer $k \geq 1$,
a \emph{\( k \)-multipartition}
\( \bla = (\lambda^{(1)}, \lambda^{(2)}, \dots, \lambda^{(k)}) \) is a
tuple of partitions \( \lambda^{(i)} \), and each of them may be
empty. The \emph{size} \( |\bla| \) of \( \bla \) is defined by
\( \sum_{i=1}^k |\lambda^{(i)}| \), and the \emph{length} \( \ell(\bla) \)
of \( \bla \) is defined by \( \sum_{i=1}^k \ell(\lambda^{(i)}) \).

A \emph{frame sequence} is a frequency sequence
\( (f_i)_{i \geq 0} \) such that \( f_{2i} \geq f_{2i+2} \) and
\( f_{2i+1} = 0 \) for all \( i \geq 0 \). We can associate a frame
sequence to each multipartition as follows. For a
\( k \)-multipartition
\( \bla = (\lambda^{(1)}, \dots, \lambda^{(k)}) \), let
\( s_1,\dots,s_k \) be integers with
\( s_1 \ge \cdots \ge s_k \geq 0 \) such that the length of
\( \lambda^{(i)} \) is \( s_i - s_{i+1} \) for all \( i \), where we
set \( s_{k+1} := 0 \). Then, the frame sequence \( \fs(\bla) \)
corresponding to \( \bla \) is defined as
\( (f_0, f_1, f_2, \dots) \), where the entries
\( f_0, f_2, f_4,\dots \) are, in this order, \( s_k \) copies of \( k \),
\( (s_{k-1} - s_k) \) copies of \( k-1 \), \dots, and
\( (s_1 - s_2) \) copies of \( 1 \). That is,
\[
  \fs(\bla) := (\underbrace{k,0,\dots,k,0}_{s_k ~\text{pairs}},\dots,\underbrace{i,0,\dots,i,0}_{s_i - s_{i+1} ~\text{pairs}},\dots,\underbrace{1,0,\dots,1,0}_{s_1 - s_2 ~\text{pairs}},0, \dots ).
\]

The following two sets are used throughout the paper.

\begin{defn}\label{def:P and A}
  For a non-negative integer $k$, let \( \mathcal{P}_k \) denote the set of all pairs
 \((\bla, \fs(\bla)) \), where \( \bla \) is a
  \( k \)-multipartition and \( \fs(\bla) \) is the frame sequence
  corresponding to \( \bla \). Let \( \mathcal{A}_k \) denote the set
  of all frequency sequences \( (f_i)_{i \geq 0} \) such that
  \( f_i + f_{i+1} \leq k \) for all \( i \geq 0 \).
\end{defn}

In this section, we introduce an \emph{insertion map} \( \Lambda \)
for multipartitions, which defines a size-preserving bijection between
the sets \( \mathcal{P}_k \) and \( \mathcal{A}_k \), where the size
of an element \((\bla, \fs(\bla))  \in \mathcal{P}_k \) is defined as
\( |(\bla, \fs(\bla)) | := |\bla| + |\fs(\bla)| \). Although
\( \fs(\bla) \) is uniquely determined by \( \bla \), we regard
\( \Lambda \) as a map from \( \mathcal{P}_k \) to
\( \mathcal{A}_k \), rather than directly from the set of
multipartitions \( \bla \), in order to emphasize that it preserves the
size.

\subsection*{Particle motion}
Let \( f = (f_0,f_1,\dots ) \) be a frequency sequence. Suppose that
\( u \) is a non-negative integer such that there exists
\( h \geq 1 \) with \( f_u + f_{u+1} = h \) and
\( f_i + f_{i+1} \leq h \) for all \( i \geq u \). We now describe the
procedure for applying \( m \) particle motions in \( f \), starting
from the pair \( (f_u, f_{u+1}) \). Consider the pair
\( (f_u, f_{u+1}) \). If the local condition
\( f_{u+1} + f_{u+2} < h \) is satisfied, we perform the following
(single) \emph{particle motion}:
\[
  (f_u, f_{u+1}) \mapsto (f_u-1, f_{u+1}+1).
\]
As long as the local condition remains satisfied for the current pair
\( (f_u, f_{u+1}) \), we continue to apply this particle motion
repeatedly at the current pair. Once the local condition is no longer
satisfied, that is, if \( f_{u+1} + f_{u+2} = h \), then we increment
\( u \) by \( 1 \), shift our focus to the next pair
\( (f_{u+1}, f_{u+2}) \), and repeat the same procedure. This process
continues until exactly \( m \) particle motions are performed.

The resulting sequence remains a frequency sequence; that is, all
entries are non-negative integers. This follows from the fact that the
single particle motion at \( (f_u, f_{u+1}) \) can only be performed
when \( f_u \geq 1 \).

\begin{defn}\label{def:pm}
  Let \( f = (f_0,f_1,\dots ) \) be a frequency sequence and $m$ be a non-negative integer. Suppose that
  \( u \) is a non-negative integer such that there exists
  \( h \geq 1 \) with \( f_u + f_{u+1} = h \) and
  \( f_i + f_{i+1} \leq h \) for all \( i \geq u \). We define
  \( \ppm_u^{(m)}(f) \) to be the resulting frequency sequence by
  applying \( m \) particle motions starting from \( (f_u, f_{u+1}) \)
  in \( f \). Moreover, if
  \( \ppm_u^{(m)}(f) = (\overline{f}_0, \overline{f}_1,\dots) \) and
  the final focus is on the pair
  \( (\overline{f}_v, \overline{f}_{v+1}) \), then we say that the
  pair \( (f_u, f_{u+1}) \) moves to
  \( (\overline{f}_v, \overline{f}_{v+1}) \).
\end{defn}

Note that shifting the focus is part of the procedure for locating the
next applicable pair and does not affect the frequency sequence. The
single particle motion \( (f_u, f_{u+1}) \mapsto (f_u-1, f_{u+1}+1) \)
implies both \( f_u + f_{u+1} \leq h \) and
\( f_{u+1} + f_{u+2} \leq h \). Therefore, by construction of the
particle motion, the frequency sequence
\( \overline{f} = \ppm_u^{(m)}(f) \) satisfies
\( \overline{f}_i + \overline{f}_{i+1} \leq h \) for all
\( i \geq u \). Note also that if the pair \( (f_u, f_{u+1}) \) moves
to \( (\overline{f}_v, \overline{f}_{v+1}) \), then
\( \overline{f}_v + \overline{f}_{v+1}= f_u + f_{u+1} = h \).

\newcommand{\boxes}[3][white]{
  \foreach \i in {0,...,\numexpr#3-1}{
    \draw[fill=#1, draw = black] (#2,\i) rectangle ++(1,1);
  }
}

\begin{exam}\label{exa:particle_motion}
  A frequency sequence \( (f_i)_{i \geq 0} \) is represented by
  placing boxes above the \( x \)-axis. Starting from the \( 0 \)th
  column, the number of boxes in the \( i \)th column corresponds to
  \( f_i \). The current focus is indicated by shading the
  corresponding boxes in gray and marking the corresponding position
  on the \( x \)-axis with a bold line.

  Let \( f = (4,0,2,0,3,1,0,0,\dots) \) be a frequency sequence.
  By~\Cref{fig:2}, we obtain \( \overline{f} = \ppm_0^{(9)}(f) \)
  where
  \[
    \overline{f} = (2,0,3,1,0,3,1,0,\dots).
  \]
  Here, \( (f_0, f_1) = (4,0) \) moves to
  \( (\overline{f}_5, \overline{f}_6) = (3,1) \).
  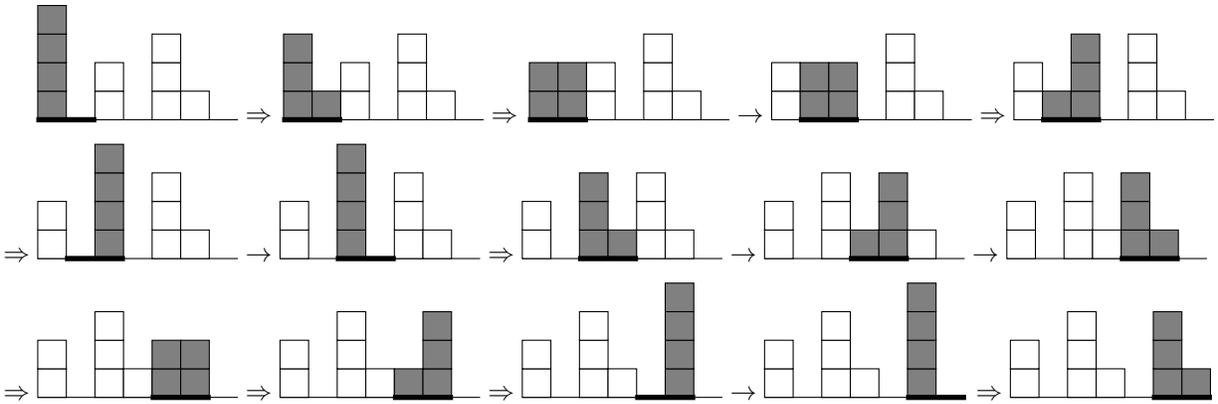
\begin{figure}[h]
    \centering
    \begin{align*}
      \begin{tikzpicture}[scale = .38]
        \draw (1,0) -- (8,0);
        \boxes[gray]{1}{4}   
        \boxes{3}{2}   
        \boxes{5}{3}   
        \boxes{6}{1}   
        \draw[line width = 2pt] (1- .05,0) -- (3 + .05,0);
      \end{tikzpicture}
      &\Rightarrow
        \begin{tikzpicture}[scale = .38]
          \draw (1,0) -- (8,0);
          \boxes[gray]{1}{3}   
          \boxes[gray]{2}{1}   
          \boxes{3}{2}   
          \boxes{5}{3}   
          \boxes{6}{1}   
          \draw[line width = 2pt] (1- .05,0) -- (3 + .05,0);
        \end{tikzpicture}
        \Rightarrow
        \begin{tikzpicture}[scale = .38]
          \draw (1,0) -- (8,0);
          \boxes[gray]{1}{2}   
          \boxes[gray]{2}{2}   
          \boxes{3}{2}   
          \boxes{5}{3}   
          \boxes{6}{1}   
          \draw[line width = 2pt] (1- .05,0) -- (3 + .05,0);
        \end{tikzpicture}
        \rightarrow
        \begin{tikzpicture}[scale = .38]
          \draw (1,0) -- (8,0);
          \boxes{1}{2}   
          \boxes[gray]{2}{2}   
          \boxes[gray]{3}{2}   
          \boxes{5}{3}   
          \boxes{6}{1}   
          \draw[line width = 2pt] (2- .05,0) -- (4 + .05,0);
        \end{tikzpicture}
        \Rightarrow
        \begin{tikzpicture}[scale = .38]
          \draw (1,0) -- (8,0);
          \boxes{1}{2}   
          \boxes[gray]{2}{1}   
          \boxes[gray]{3}{3}   
          \boxes{5}{3}   
          \boxes{6}{1}   
          \draw[line width = 2pt] (2- .05,0) -- (4 + .05,0);
        \end{tikzpicture}\\
      \Rightarrow
      \begin{tikzpicture}[scale = .38]
        \draw (1,0) -- (8,0);
        \boxes{1}{2}   
        \boxes[gray]{3}{4}   
        \boxes{5}{3}   
        \boxes{6}{1}   
        \draw[line width = 2pt] (2- .05,0) -- (4 + .05,0);
      \end{tikzpicture}
      &\rightarrow
        \begin{tikzpicture}[scale = .38]
          \draw (1,0) -- (8,0);
          \boxes{1}{2}   
          \boxes[gray]{3}{4}   
          \boxes{5}{3}   
          \boxes{6}{1}   
          \draw[line width = 2pt] (3- .05,0) -- (5 + .05,0);
        \end{tikzpicture}
        \Rightarrow
        \begin{tikzpicture}[scale = .38]
          \draw (1,0) -- (8,0);
          \boxes{1}{2}   
          \boxes[gray]{3}{3}   
          \boxes[gray]{4}{1}   
          \boxes{5}{3}   
          \boxes{6}{1}   
          \draw[line width = 2pt] (3- .05,0) -- (5 + .05,0);
        \end{tikzpicture}
        \rightarrow
        \begin{tikzpicture}[scale = .38]
          \draw (1,0) -- (8,0);
          \boxes{1}{2}   
          \boxes{3}{3}   
          \boxes[gray]{4}{1}   
          \boxes[gray]{5}{3}   
          \boxes{6}{1}   
          \draw[line width = 2pt] (4- .05,0) -- (6 + .05,0);
        \end{tikzpicture}
        \rightarrow
        \begin{tikzpicture}[scale = .38]
          \draw (1,0) -- (8,0);
          \boxes{1}{2}   
          \boxes{3}{3}   
          \boxes{4}{1}   
          \boxes[gray]{5}{3}   
          \boxes[gray]{6}{1}   
          \draw[line width = 2pt] (5- .05,0) -- (7 + .05,0);
        \end{tikzpicture}\\
      \Rightarrow
      \begin{tikzpicture}[scale = .38]
        \draw (1,0) -- (8,0);
        \boxes{1}{2}   
        \boxes{3}{3}   
        \boxes{4}{1}   
        \boxes[gray]{5}{2}   
        \boxes[gray]{6}{2}   
        \draw[line width = 2pt] (5- .05,0) -- (7 + .05,0);
      \end{tikzpicture} 
      &\Rightarrow
        \begin{tikzpicture}[scale = .38]
          \draw (1,0) -- (8,0);
          \boxes{1}{2}   
          \boxes{3}{3}   
          \boxes{4}{1}   
          \boxes[gray]{5}{1}   
          \boxes[gray]{6}{3}   
          \draw[line width = 2pt] (5- .05,0) -- (7 + .05,0);
        \end{tikzpicture}
        \Rightarrow
        \begin{tikzpicture}[scale = .38]
          \draw (1,0) -- (8,0);
          \boxes{1}{2}   
          \boxes{3}{3}   
          \boxes{4}{1}   
          \boxes[gray]{6}{4}   
          \draw[line width = 2pt] (5- .05,0) -- (7 + .05,0);
        \end{tikzpicture}
        \rightarrow
        \begin{tikzpicture}[scale = .38]
          \draw (1,0) -- (8,0);
          \boxes{1}{2}   
          \boxes{3}{3}   
          \boxes{4}{1}   
          \boxes[gray]{6}{4}   
          \draw[line width = 2pt] (6- .05,0) -- (8 + .05,0);
        \end{tikzpicture}
        \Rightarrow
        \begin{tikzpicture}[scale = .38]
          \draw (1,0) -- (8,0);
          \boxes{1}{2}   
          \boxes{3}{3}   
          \boxes{4}{1}   
          \boxes[gray]{6}{3}   
          \boxes[gray]{7}{1}   
          \draw[line width = 2pt] (6- .05,0) -- (8 + .05,0);
        \end{tikzpicture}
    \end{align*}
    \caption{Illustration of applying \( 9 \) particle motions
      starting from \( (f_0, f_1) \) in the frequency sequence
      \( f = (4,0,2,0,3,1,0,0, \cdots ) \). The symbol
      \( \Rightarrow \) indicates a single particle motion, and
      \( \rightarrow \) indicates a focus shift.}
    \label{fig:2}
  \end{figure}
\end{exam}

In this paper, we only apply the particle motions to pairs
\( (f_u,f_{u+1}) \) of the form \( (h,0) \) with a positive integer
\( h \) that occurs in the frame sequence. Therefore, from now on, we
assume that the starting pair \( (f_u, f_{u+1}) \) is of the form
\( (h,0) \) with a positive integer \( h \).

In \cite{DJK}, the authors describe the frequency sequence
obtained by iteratively applying the particle motion. As a result, the
resulting frequency sequence can be computed directly in a single
step, without performing each individual particle motion. In the next proposition, we show that applying $m$ particle motions as described above gives exactly the same frequency sequence.

\begin{prop}\label{pro:explicit_pm}
  Let \( f = (f_0,f_1,\dots ) \) be a frequency sequence. Suppose that
  \( u \) is a non-negative integer such that there exists
  \( h \geq 1 \) with \( (f_u, f_{u+1}) = (h,0) \) and
  \( f_i + f_{i+1} \leq h \) for all \( i \geq u \). For a non-negative integer $m$, let
  \( \overline{f} = (\overline{f}_0, \overline{f}_1, \dots ) \) be the
  frequency sequence obtained by applying \( m \) particle motions in
  \( f \), starting from the pair \( (f_u, f_{u+1}) \). To determine
  the position where the pair \( (f_u, f_{u+1}) \) moves, define
  \begin{equation}\label{eq:u_bar}
    v := \min \left\{ t \geq u+2 : \sum_{i = u+2}^{t} \left( h - (f_{i -1} + f_{i}) \right) \geq m \right\}.
  \end{equation}
  Then \( (f_u, f_{u+1}) \) moves to
  \( (\overline{f}_{v-2}, \overline{f}_{v -1})
  \). The frequency sequence \( \overline{f} = \ppm_u^{(m)}(f) \) is
  given explicitly by:
  \begin{align}
    \label{eq:f_bar}
    \overline{f}_i =
    \begin{cases}
      f_i & \mbox{if \( 0 \leq i < u \),}\\
      f_{i+2} & \mbox{if \( u \leq i < v -2 \),} \\
      f_{v} + \sum_{j = u+2}^{v} (h - (f_{j-1} + f_j)) - m & \mbox{if \( i = v -2 \),} \\
      f_{v -1} + m - \sum_{j = u+2}^{v -1} (h - (f_{j-1} + f_j)) & \mbox{if \( i = v -1 \),} \\
      f_i & \mbox{if \( i \geq v \).}
    \end{cases}
  \end{align}
\end{prop}
\begin{proof}
  Let \( (f_u, f_{u+1}) = (h,0) \) move to
  \( (\overline{f}_{v-2}, \overline{f}_{v-1}) \) for some
  \( v-2 \geq u \). We use the following facts: (1)
  \( \overline{f}_{v-2} + \overline{f}_{v-1} = h \); (2)~the entries
  originally between \( f_{u+2} \) and \( f_{v-1} \) are shifted two
  steps to the left; and (3) all other entries remain unchanged.
  Facts~(1) and (3) are straightforward. Fact~(2) requires a brief
  explanation. The focus moves from \( (f_u, f_{u+1}) \) to
  \( (f_{u+1}, f_{u+2}) \) only when
  \( f_u + f_{u+1} = f_{u+1} + f_{u+2} \), that is, when
  \( f_u = f_{u+2} \). From this point on, the value of the \( u \)th
  entry remains equal to \( f_{u+2} \). Therefore, we may regard the
  original value of \( f_{u+2} \) as being shifted two steps to the
  left.

  Using the facts above, we have
  \begin{align*}
    |\overline{f}|
    &= \sum_{i =0}^{u-1} i f_i + \sum_{i = u}^{v-3} i f_{i+2} + (v-2)\overline{f}_{v-2} + (v-1) \overline{f}_{v-1} + \sum_{i \geq v} i f_i \\
    &= \sum_{i \geq 0} i f_i - 2 \sum_{i = u}^{v-3} f_{i+2} - u f_u - (u+1) f_{u+1} + (v-2)\overline{f}_{v-2} + (v-1) \overline{f}_{v-1} \\
    &= |f| - 2\sum_{i = u+2}^{v-1}f_i - uh + (v-2)\overline{f}_{v-2} + (v-1)\overline{f}_{v-1}.   
  \end{align*}
  Since \( \overline{f} \) is obtained from \( f \) via \( m \)
  particle motions, its size is \( | \overline{f} | = | f | + m \).
  Therefore, we obtain
  \[
    (v-2)\overline{f}_{v-2} + (v-1)\overline{f}_{v-1} = m + uh + \sum_{i = u+2}^{v-1} f_i.
  \]
  Using \( \overline{f}_{v-2} + \overline{f}_{v-1} = h \), we can
  express the left-hand side as either
  \( (v-2)h + \overline{f}_{v-1} \) or
  \( (v-1)h - \overline{f}_{v-2} \). These two expressions determine
  \( \overline{f}_{v-2} \) and \( \overline{f}_{v-1} \), as given
  in~\eqref{eq:f_bar}. One can also check the formula~\eqref{eq:u_bar}.
\end{proof}

\begin{exam}\label{exa:3}
  Let \( f = (4,0,2,0,3,1,0,0, \dots) \), \( u = 0 \) and \( m = 9 \).
  Then \( f_0 + f_1 = 4 = h \). We have \( v = 7 \), since
  \[
    \sum_{i = 2}^{6}(h-(f_{i-1} + f_i)) = 8 < 9, \qand \sum_{i = 2}^{7}(h-(f_{i-1} + f_i)) = 12 \geq 9.
  \]
  By the explicit formula \eqref{eq:f_bar},
  \[
    \overline{f}_i =
    \begin{cases}
      f_{i+2} & \mbox{if \( 0 \leq i < 5 \),} \\
      0 + 12 - 9 = 3 & \mbox{if \( i = 5 \),} \\
      0 + 9 - 8 = 1 & \mbox{if \( i = 6 \),} \\
      0 & \mbox{if \( i \geq 7 \).}
    \end{cases}
  \]
  Hence, we obtain
  \[
    \ppm_0^{(9)}((4,0,2,0,3,1,0,0,\dots)) = (2,0,3,1,0,3,1,0,\dots),
  \]
  which coincides with the result in \Cref{exa:particle_motion}.
\end{exam}

Now we reformulate the explicit map \( \Lambda \) defined in \cite{DJK} in terms of particle motions.

\subsection*{The map \( \Lambda \)}
We now construct the map
\( \Lambda : \mathcal{P}_k \to \mathcal{A}_k \). For a
\( k \)-multipartition
\( \bla = (\lambda^{(1)},\dots,\lambda^{(k)}) \), define a sequence
\( (s_1,\dots,s_k) \) of non-negative integers with
\( s_1 \ge \cdots \ge s_k \geq 0 \) such that, for all \( i \),
\( \lambda^{(i)} \) has a length \( s_i - s_{i+1} \) where we set
\( s_{k+1} := 0 \). Recall that the frame sequence \( \fs(\bla) \)
consists of \( s_i - s_{i+1} \) pairs equal to \( (i,0) \) for
\( i = 1,\dots,k \). The map \( \Lambda \) produces the frequency
sequence from \( \fs(\bla) \) by applying the particle motion starting
from \( (1,0) \) with step sizes given by the parts of
\( \lambda^{(1)} \), from \( (2,0) \) with step sizes given by the
parts of \( \lambda^{(2)} \), \dots, and from \( (k,0) \) with step
sizes given by the parts of \( \lambda^{(k)} \). We give a formal
definition of this map.

Let \( \lambda = (\lambda_0,\dots,\lambda_{s_1-1}) \) be the
sequence of non-negative integers defined by
\begin{equation}\label{eq:lambda}
  (\lambda_{s_1-1},\lambda_{s_1 -2},\dots,\lambda_1,\lambda_0) =
  (\lambda_1^{(1)},\dots,\lambda_{s_1-s_2}^{(1)},\dots,\lambda_{1}^{(k)},\dots,\lambda_{s_k}^{(k)}).
\end{equation}
That is, let
\( (\lambda_{s_i-1},\lambda_{s_i-2}\dots,\lambda_{s_{i+1}}) =
(\lambda^{(i)}_1,\dots,\lambda^{(i)}_{s_i - s_{i+1}}) \) for each
\( i = 1,\dots,k \). Since \( \bla \) can be obtained directly from
\( \fs(\bla) \) and \( \lambda \), we regard the pairs
\( (\bla, \fs(\bla)) \) and \( (\lambda,\fs(\bla)) \) as essentially
the same.

The map \( \Lambda \) associates to a pair \( (\bla, \fs(\bla)) \) ,
where \( \bla \) is a \( k \)-multipartition and \( \fs(\bla) \) is
the frequency sequence corresponding to \( \bla \), the frequency
sequence
\[
  \Lambda( \bla, \fs(\bla)) =\left(\ppm_{0}^{(\lambda_0)} \circ \ppm_{2}^{(\lambda_1)} \circ \cdots \circ \ppm_{2(s_1 -2)}^{(\lambda_{s_1 -2})} \circ \ppm_{2(s_1 -1)}^{(\lambda_{s_1 -1})}\right) \left(\fs(\bla)\right).
\]
This map \( \Lambda \) can also be described, as was done in \cite{DJK}, by using a sequence
of intermediate frequency sequences, constructed recursively by
\begin{equation}\label{eq:theta}
   \fs(\bla) =: \theta^{(s_1)} , \theta^{(s_1-1)} , \dots  , \theta^{(1)}, \theta^{(0)} := \Lambda(\bla,\fs(\bla)),
\end{equation}
where each \( \theta^{(i)} \) is obtained from \( \theta^{(i+1)} \) by
\begin{equation}\label{eq:recursion_theta}
  \theta^{(i)} = \ppm_{2i}^{(\lambda_{i})}\left(\theta^{(i+1)}\right) \quad\mbox{for \( i = s_1 - 1 , \dots,1,0 \)}.
\end{equation}

We index the parts \( \lambda_{s_1-1},\dots,\lambda_0 \) and the
recursive steps \( \theta^{(s_1)},\dots,\theta^{(0)} \) in reverse
order to ensure consistency with the starting positions
\( 2(s_1 - 1),\dots,2,0 \) of the particle motions. Note
that if \( (\theta^{(i+1)}_{2i},\theta^{(i+1)}_{2i+1}) = (h,0) \),
then \( \lambda_{i} \) is an entry of the partition
\( \lambda^{(h)} \).

\begin{exam}\label{exa:2}
  Let
  \( \bla = (\lambda^{(1)}, \lambda^{(2)}, \lambda^{(3)},
  \lambda^{(4)}) = ((3,1), \emptyset, (6,6,5,3), (19,0)) \) be a
  \( 4 \)-multipartition. Then the corresponding frame sequence is
  \( \fs(\bla) = (4,0,4,0,3,0,3,0,3,0,3,0,1,0,1,0, \dots ) =
  \theta^{(8)} \). See~\Cref{fig:1} for the process of applying the
  map \( \Lambda \) to \( (\fs(\bla), \bla) \).
  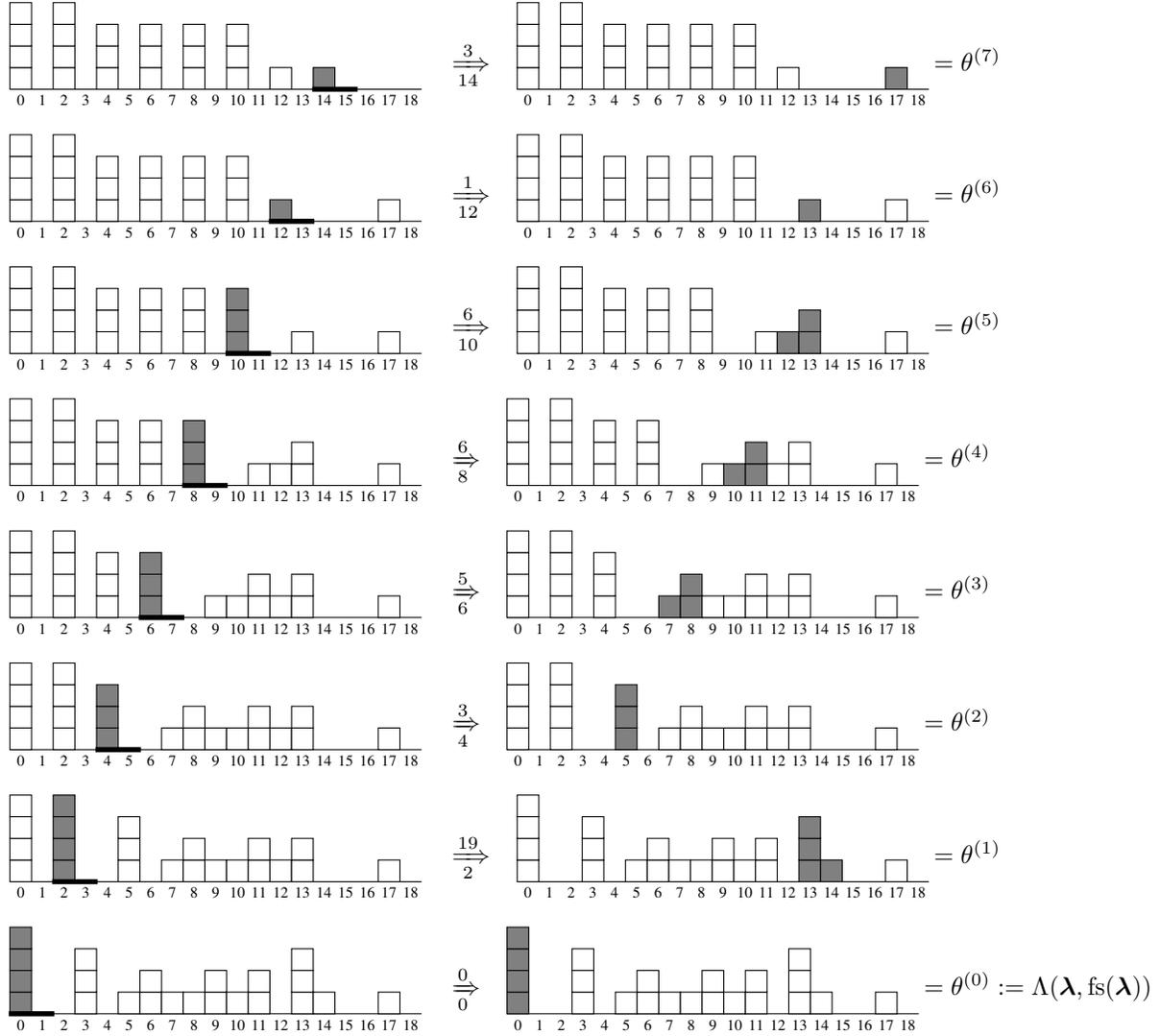
\begin{figure}
    \centering
\begin{align*}
    \begin{tikzpicture}[scale = .3]
      \draw (1,0) -- (20,0);
      \foreach \x in {1,...,19}
      \node at (\x+ .5, -0.5) {\tiny \number\numexpr\x - 1\relax};
      \boxes{1}{4}   
      \boxes{3}{4}   
      \boxes{5}{3}   
      \boxes{7}{3}   
      \boxes{9}{3}   
      \boxes{11}{3}   
      \boxes{13}{1}   
      \boxes[gray]{15}{1}   
      \draw[line width = 2pt] (15- .05,0) -- (17 + .05,0);
    \end{tikzpicture}
    \quad&\raisebox{+.6cm}{\( \xRightarrow[14]{3} \)}\quad
    \begin{tikzpicture}[scale = .3]
      \draw (1,0) -- (20,0);
      \foreach \x in {1,...,19}
      \node at (\x+ .5, -0.5) {\tiny \number\numexpr\x - 1\relax};
      \boxes{1}{4}   
      \boxes{1+2}{4}   
      \boxes{3+2}{3}   
      \boxes{5+2}{3}   
      \boxes{7+2}{3}   
      \boxes{9+2}{3}   
      \boxes{11+2}{1}   
      \boxes[gray]{16+2}{1}   
    \end{tikzpicture}
    \raisebox{+.6cm}{\( = \theta^{(7)} \)} \\
    \begin{tikzpicture}[scale = .3]
      \draw (1,0) -- (20,0);
      \foreach \x in {1,...,19}
      \node at (\x+ .5, -0.5) {\tiny \number\numexpr\x - 1\relax};
      \boxes{1}{4}   
      \boxes{1+2}{4}   
      \boxes{3+2}{3}   
      \boxes{5+2}{3}   
      \boxes{7+2}{3}   
      \boxes{9+2}{3}   
      \boxes[gray]{11+2}{1}   
      \boxes{16+2}{1}   
      \draw[line width = 2pt] (11+2- .05,0) -- (13+2 + .05,0);
    \end{tikzpicture}
    \quad &\raisebox{+.6cm}{\( \xRightarrow[12]{1} \)} \quad
    \begin{tikzpicture}[scale = .3]
      \draw (1,0) -- (20,0);
      \foreach \x in {1,...,19}
      \node at (\x+ .5, -0.5) {\tiny \number\numexpr\x - 1\relax};
      \boxes{1}{4}   
      \boxes{1+2}{4}   
      \boxes{3+2}{3}   
      \boxes{5+2}{3}   
      \boxes{7+2}{3}   
      \boxes{9+2}{3}   
      \boxes[gray]{12+2}{1}   
      \boxes{16+2}{1}   
    \end{tikzpicture}
    \raisebox{+.6cm}{\( = \theta^{(6)} \)} \\
    \begin{tikzpicture}[scale = .3]
      \draw (1,0) -- (20,0);
      \foreach \x in {1,...,19}
      \node at (\x+ .5, -0.5) {\tiny \number\numexpr\x - 1\relax};
      \boxes{1}{4}   
      \boxes{1+2}{4}   
      \boxes{3+2}{3}   
      \boxes{5+2}{3}   
      \boxes{7+2}{3}   
      \boxes[gray]{9+2}{3}   
      \boxes{12+2}{1}   
      \boxes{16+2}{1}   
      \draw[line width = 2pt] (9+2- .05,0) -- (11+2 + .05,0);
    \end{tikzpicture}
    \quad &\raisebox{+.6cm}{\( \xRightarrow[10]{6} \)} \quad
    \begin{tikzpicture}[scale = .3]
      \draw (1,0) -- (20,0);
      \foreach \x in {1,...,19}
      \node at (\x+ .5, -0.5) {\tiny \number\numexpr\x - 1\relax};
      \boxes{1}{4}   
      \boxes{1+2}{4}   
      \boxes{3+2}{3}   
      \boxes{5+2}{3}   
      \boxes{7+2}{3}   
      \boxes{10+2}{1}   
      \boxes[gray]{11+2}{1}   
      \boxes[gray]{12+2}{2}   
      \boxes{16+2}{1}   
    \end{tikzpicture}
    \raisebox{+.6cm}{\( = \theta^{(5)} \)} \\
    \begin{tikzpicture}[scale = .3]
      \draw (1,0) -- (20,0);
      \foreach \x in {1,...,19}
      \node at (\x+ .5, -0.5) {\tiny \number\numexpr\x - 1\relax};
      \boxes{1}{4}   
      \boxes{1+2}{4}   
      \boxes{3+2}{3}   
      \boxes{5+2}{3}   
      \boxes[gray]{7+2}{3}   
      \boxes{10+2}{1}   
      \boxes{11+2}{1}   
      \boxes{12+2}{2}   
      \boxes{16+2}{1}   
      \draw[line width = 2pt] (7+2- .05,0) -- (9+2 + .05,0);
    \end{tikzpicture}
    \quad &\raisebox{+.6cm}{\( \xRightarrow[8]{6} \)} \quad
    \begin{tikzpicture}[scale = .3]
      \draw (1,0) -- (20,0);
      \foreach \x in {1,...,19}
      \node at (\x+ .5, -0.5) {\tiny \number\numexpr\x - 1\relax};
      \boxes{1}{4}   
      \boxes{1+2}{4}   
      \boxes{3+2}{3}   
      \boxes{5+2}{3}   
      \boxes{8+2}{1}   
      \boxes[gray]{9+2}{1}   
      \boxes[gray]{10+2}{2}   
      \boxes{11+2}{1}   
      \boxes{12+2}{2}   
      \boxes{16+2}{1}   
    \end{tikzpicture}
    \raisebox{+.6cm}{\( = \theta^{(4)} \)} \\
    \begin{tikzpicture}[scale = .3]
      \draw (1,0) -- (20,0);
      \foreach \x in {1,...,19}
      \node at (\x+ .5, -0.5) {\tiny \number\numexpr\x - 1\relax};
      \boxes{1}{4}   
      \boxes{1+2}{4}   
      \boxes{3+2}{3}   
      \boxes[gray]{5+2}{3}   
      \boxes{8+2}{1}   
      \boxes{9+2}{1}   
      \boxes{10+2}{2}   
      \boxes{11+2}{1}   
      \boxes{12+2}{2}   
      \boxes{16+2}{1}   
      \draw[line width = 2pt] (5+2- .05,0) -- (7+2 + .05,0);
    \end{tikzpicture}
    \quad &\raisebox{+.6cm}{\( \xRightarrow[6]{5} \)} \quad
    \begin{tikzpicture}[scale = .3]
      \draw (1,0) -- (20,0);
      \foreach \x in {1,...,19}
      \node at (\x+ .5, -0.5) {\tiny \number\numexpr\x - 1\relax};
      \boxes{1}{4}   
      \boxes{1+2}{4}   
      \boxes{3+2}{3}   
      \boxes[gray]{6+2}{1}   
      \boxes[gray]{7+2}{2}   
      \boxes{8+2}{1}   
      \boxes{9+2}{1}   
      \boxes{10+2}{2}   
      \boxes{11+2}{1}   
      \boxes{12+2}{2}   
      \boxes{16+2}{1}   
    \end{tikzpicture}
    \raisebox{+.6cm}{\( = \theta^{(3)} \)} \\
    \begin{tikzpicture}[scale = .3]
      \draw (1,0) -- (20,0);
      \foreach \x in {1,...,19}
      \node at (\x+ .5, -0.5) {\tiny \number\numexpr\x - 1\relax};
      \boxes{1}{4}   
      \boxes{1+2}{4}   
      \boxes[gray]{3+2}{3}   
      \boxes{6+2}{1}   
      \boxes{7+2}{2}   
      \boxes{8+2}{1}   
      \boxes{9+2}{1}   
      \boxes{10+2}{2}   
      \boxes{11+2}{1}   
      \boxes{12+2}{2}   
      \boxes{16+2}{1}   
      \draw[line width = 2pt] (3+2- .05,0) -- (5+2 + .05,0);
    \end{tikzpicture}
    \quad &\raisebox{+.6cm}{\( \xRightarrow[4]{3} \)} \quad
    \begin{tikzpicture}[scale = .3]
      \draw (1,0) -- (20,0);
      \foreach \x in {1,...,19}
      \node at (\x+ .5, -0.5) {\tiny \number\numexpr\x - 1\relax};
      \boxes{1}{4}   
      \boxes{1+2}{4}   
      \boxes[gray]{4+2}{3}   
      \boxes{6+2}{1}   
      \boxes{7+2}{2}   
      \boxes{8+2}{1}   
      \boxes{9+2}{1}   
      \boxes{10+2}{2}   
      \boxes{11+2}{1}   
      \boxes{12+2}{2}   
      \boxes{16+2}{1}   
    \end{tikzpicture}
    \raisebox{+.6cm}{\( = \theta^{(2)} \)} \\
    \begin{tikzpicture}[scale = .3]
      \draw (1,0) -- (20,0);
      \foreach \x in {1,...,19}
      \node at (\x+ .5, -0.5) {\tiny \number\numexpr\x - 1\relax};
      \boxes{1}{4}   
      \boxes[gray]{1+2}{4}   
      \boxes{4+2}{3}   
      \boxes{6+2}{1}   
      \boxes{7+2}{2}   
      \boxes{8+2}{1}   
      \boxes{9+2}{1}   
      \boxes{10+2}{2}   
      \boxes{11+2}{1}   
      \boxes{12+2}{2}   
      \boxes{16+2}{1}   
      \draw[line width = 2pt] (1+2- .05,0) -- (3+2 + .05,0);
    \end{tikzpicture}
    \quad &\raisebox{+.6cm}{\( \xRightarrow[2]{19} \)} \quad
    \begin{tikzpicture}[scale = .3]
      \draw (1,0) -- (20,0);
      \foreach \x in {1,...,19}
      \node at (\x+ .5, -0.5) {\tiny \number\numexpr\x - 1\relax};
      \boxes{1}{4}   
      \boxes{2+2}{3}   
      \boxes{4+2}{1}   
      \boxes{5+2}{2}   
      \boxes{6+2}{1}   
      \boxes{7+2}{1}   
      \boxes{8+2}{2}   
      \boxes{9+2}{1}   
      \boxes{10+2}{2}   
      \boxes[gray]{12+2}{3}   
      \boxes[gray]{13+2}{1}   
      \boxes{16+2}{1}   
    \end{tikzpicture}
            \raisebox{+.6cm}{\( = \theta^{(1)} \)}  \\
    \begin{tikzpicture}[scale = .3]
      \draw (1,0) -- (20,0);
      \foreach \x in {1,...,19}
      \node at (\x+ .5, -0.5) {\tiny \number\numexpr\x - 1\relax};
      \boxes[gray]{1}{4}   
      \boxes{2+2}{3}   
      \boxes{4+2}{1}   
      \boxes{5+2}{2}   
      \boxes{6+2}{1}   
      \boxes{7+2}{1}   
      \boxes{8+2}{2}   
      \boxes{9+2}{1}   
      \boxes{10+2}{2}   
      \boxes{12+2}{3}   
      \boxes{13+2}{1}   
      \boxes{16+2}{1}   
      \draw[line width = 2pt] (1 - .05,0) -- (3 + .05,0);
    \end{tikzpicture}
    \quad &\raisebox{+.6cm}{\( \xRightarrow[0]{0} \)} \quad
    \begin{tikzpicture}[scale = .3]
      \draw (1,0) -- (20,0);
      \foreach \x in {1,...,19}
      \node at (\x+ .5, -0.5) {\tiny \number\numexpr\x - 1\relax};
      \boxes[gray]{1}{4}   
      \boxes{2+2}{3}   
      \boxes{4+2}{1}   
      \boxes{5+2}{2}   
      \boxes{6+2}{1}   
      \boxes{7+2}{1}   
      \boxes{8+2}{2}   
      \boxes{9+2}{1}   
      \boxes{10+2}{2}   
      \boxes{12+2}{3}   
      \boxes{13+2}{1}   
      \boxes{16+2}{1}   
    \end{tikzpicture}
    \raisebox{+.6cm}{\( = \theta^{(0)} := \Lambda( \bla, \fs(\bla)) \)} 
  \end{align*}

  \caption{The process of applying the map \( \Lambda \) to
    \( ( \bla, \fs(\bla)) \) where
    \( \bla = ((3,1), \emptyset, (6,6,5,3), (19,0)) \). Here, the
    notation \( f \xRightarrow[u]{m} \overline{f} \) indicates that
    \( \overline{f} = \ppm_u^{(m)}(f) \).}
    \label{fig:1}
  \end{figure}
  Hence, we have
  \begin{align*}
    \Lambda(\fs(\bla), \bla)
    &= \left(\ppm_0^{(0)}\circ \ppm_2^{(19)}\circ \ppm_4^{(3)}\circ \ppm_6^{(5)}\circ \ppm_8^{(6)}\circ \ppm_{10}^{(6)}\circ \ppm_{12}^{(1)}\circ \ppm_{14}^{(3)}\right)\left(\fs(\bla))\right) \\
    &= (4,0,0,3,0,1,2,1,1,2,1,2,0,3,1,0,0,1,0,\dots).
  \end{align*}
  The size of the frequency sequence \( \Lambda(\fs(\bla), \bla) \) is
  equal to
  \[
    |\fs(\bla)| + |\lambda^{(1)}| + |\lambda^{(2)}| + |\lambda^{(3)}| + |\lambda^{(4)}| = 118 + 4 + 0 + 20 + 19 = 171.
  \]
\end{exam}

By construction, \( \Lambda(\fs(\bla), \bla) \) is a frequency
sequence satisfying \( f_i + f_{i+1} \leq k \) for all \( i \geq 0 \).
Hence the map \( \Lambda : \mathcal{P}_k \to \mathcal{A}_k \) is
well-defined. In addition, the map \( \Lambda \) is in fact a
bijection from \( \mathcal{P}_k \) to \( \mathcal{A}_k \). The
following property proved in \cite{DJK} plays an important role in constructing the inverse
map of \( \Lambda \).

\begin{prop}[reformulation of Proposition 3.15 of \cite{DJK}]\label{pro:property of Lambda}
  Let \( \bla \) be a \( k \)-multipartition with
  \( \ell(\bla) = s \). Let
  \( \theta^{(s)},\dots,\theta^{(1)}, \theta^{(0)} \) be the sequence
  of frequency sequences in the construction~\eqref{eq:theta} of
  \( \Lambda(\fs(\bla), \bla) \), that is,
  \[
    \theta^{(i)} = \ppm_{2i}^{(\lambda_i)}(\theta^{(i+1)}) \qand (\theta^{(i+1)}_{2i}, \theta^{(i+1)}_{2i+1}) = (h,0),    
  \]
  for some \( 1 \leq h \leq k \). Assume that
  \( (\theta^{(i+1)}_{2i}, \theta^{(i+1)}_{2i+1}) \) moves to
  \( (\theta^{(i)}_v, \theta^{(i)}_{v+1}) \) via the map
  \( \ppm_{2i}^{(\lambda_i)} \). Then, the integer \( h \) is the
  largest value of \( \theta^{(i)}_u + \theta^{(i)}_{u+1} \) for all
  \( u \geq 2i \), and \( v \) is the smallest integer \( u \) such
  that \( u \geq 2i \) and
  \( \theta^{(i)}_u + \theta^{(i)}_{u+1} = h \).
\end{prop}

We briefly describe the construction of the inverse map \( \Gamma \)
of \( \Lambda \). Given a frequency sequence, find the leftmost pair
of adjacent entries whose sum is maximal. Apply reverse particle
motions to move this pair to the first and second positions, until the
second position becomes zero. Record the number of reverse particle
motions applied during this step. Then, excluding the first and second
entries, repeat the same procedure: find the next leftmost maximal
pair among the remaining entries, move it to the third and fourth
position using reverse particle motions, until the fourth position
becomes zero. Again, record how many reverse particle motions are
applied. Continue this process until all entries of the remaining
sequence are zero. The resulting sequence represents the frame
sequence \( \fs(\bmu) \), and the recorded numbers form the sequence
\( (\mu_{s -1},\dots,\mu_0) \), from which the multipartition
\( \bmu \) can be recovered.

Based on the above description, we now formulate the construction of
\( \Gamma \) more precisely.

\subsection*{Reverse particle motions}

We define the reverse step of the particle motions to
construct the map \( \Gamma \). Let \( f = (f_0, f_1,\dots) \) be a
frequency sequence and \( u \) be a non-negative integer such that
\( f_{u-1} = 0 \), where we set \( f_{-1} = 0 \). We now describe the
procedure for applying \emph{reverse particle motions in \( f \)
  ending at \( u \)}.

Let \( h \) be the largest value of \( f_i + f_{i+1} \) for all
\( i \geq u \). Choose the smallest index \( v \geq u \) such that
\( f_v + f_{v+1} = h \). Consider the pair \( (f_v, f_{v+1}) \). If
the reverse local condition \( f_{v-1} + f_v < h \) is satisfied, we
perform the following (single) \emph{reverse particle motion}:
\[
  (f_v, f_{v+1}) \mapsto (f_v +1, f_{v+1} -1).
\]
As long as the reverse local condition remains satisfied at the
current pair, we continue applying this reverse particle motion
repeatedly at the same pair. Once the condition is no longer
satisfied, we decrement \( v \) by \( 1 \), shift our focus to the
previous pair \( (f_{v-1}, f_v) \), and resume the same procedure.
This process continues until the focus reaches the pair
\( (f_u, f_{u+1}) \) and the condition \( f_{u+1} = 0 \) is satisfied.
Record the resulting frequency sequence and the total number of
reverse particle motions applied during the process.

\begin{defn}\label{def:rpm}
  Let \( f = (f_0, f_1,\dots) \) be a frequency sequence and \( u \)
  be a non-negative integer such that \( f_{u-1} = 0 \). We define
  \( \rpm_u(f) \) to be the frequency sequence obtained by applying
  reverse particle motions in \( f \) ending at \( u \), and
  \( \rstep_u(f) \) to be the total number of reverse particle motions
  applied, as described above.
\end{defn}

Note that in general, the reverse particle motion is not the inverse
of the particle motion. Suppose that, for a given frequency sequence
\( f \) and non-negative integer \( u \), the pair
\( (f_u, f_{u+1}) = (h,0) \) moves to
\( (\overline{f}_v, \overline{f}_{v+1}) \) by applying particle
motions multiple times in \( f \) starting from \( (f_u, f_{u+1}) \).
To recover \( f \) via the reverse particle motion of
\( \overline{f} \) ending at \( u \), the index \( v \geq u \) must be
the smallest index such that \( f_v + f_{v+1} = h \). For instance,
in~\Cref{exa:particle_motion}, we have
\( (2,0,3,1,0,3,1,0, \cdots ) = \ppm_0^{(9)}((4,0,2,0,3,1,0, \cdots ))
\), but
\[
  \rpm_0((2,0,3,1,0,3,1,0, \cdots )) = (4,0,2,0,0,3,1,0, \cdots ) \not= (4,0,2,0,3,1,0, \cdots ).
\]
In this case, the reverse particle motion is not the inverse of the
particle motion. On the other hand, in~\Cref{exa:2}, we have
\( \rpm_{2i}(\theta^{(i)}) = \theta^{(i+1)} \) for all
\( i = 0,\dots,7 \). Thus, in these eight cases, the reverse particle
motion is indeed the inverse of the particle motion. In fact, by
\Cref{pro:property of Lambda}, all particle motions appearing in
\( \Lambda \) have reverse particle motion as their inverse map.

Similar to the particle motions, the authors in
\cite{DJK} gave explicit formulas for \( \rpm_u(f) \) and
\( \rstep_u(f) \).

\begin{prop}\label{pro:explicit_rpm}
  Let \( f \) be a frequency sequence, and let \( u \) be a
  non-negative integer such that \( f_{u-1} = 0 \), where
  \( f_{-1} := 0 \). Set 
  \[
    h = \max \{f_i+ f_{i+1} : i \geq u \}, \qand v = \min \{ i \geq u+2 : f_{i-2} + f_{i-1} = h \}.
  \]
  Then the frequency sequence
  \( \rpm_u(f) = (\overline{f}_0 , \overline{f}_1, \dots) \) and the
  non-negative integer \( \rstep_u(f) \) are given explicitly by:
  \begin{align}
    \label{eq:rev_f_bar}
    \overline{f}_i =
    \begin{cases}
      f_i & \mbox{if \( 0 \leq i < u \),}\\
      f_{i-2} & \mbox{if \( u+2 \leq i < v  \),} \\
      f_i & \mbox{if \( i \geq v \),}
    \end{cases}
    \qand
    (\overline{f}_u, \overline{f}_{u+1}) = (h,0),
  \end{align}
  and
  \begin{equation}\label{eq:rstep}
    \rstep_u(f) = h - f_u + \sum_{i = u}^{v-3}(h - (f_i + f_{i+1})). 
  \end{equation}
\end{prop}
\begin{proof}
  The proof uses a similar idea as in the proof of
  \Cref{pro:explicit_pm}. By applying reverse particle motion, the
  intermediate entries originally between \( f_{u} \) and
  \( f_{v-3} \) are shifted two steps to the right. From this fact,
  \eqref{eq:rev_f_bar} follows. The formula~\eqref{eq:rstep} then
  follows directly from \( |\rpm_u(f)| + \rstep_u(f) = |f| \).
\end{proof}

\subsection*{The map \( \Gamma \)}
We now construct the map
\( \Gamma : \mathcal{A}_k \to \mathcal{P}_k \). Let \( f \) be a
frequency sequence such that \( f_i + f_{i+1} \leq k \) for all
\( i \geq 0 \). The image of \( f \) by the map \( \Gamma \) is
defined recursively as follows: Let \( \eta^{(0)} = f \). Construct
\( \eta^{(i+1)} \) and \( \mu_i \) for \( i = 0,1,\dots \) recursively
by
\begin{equation}\label{eq:eta}
  \eta^{(i+1)} = \rpm_{2i}(\eta^{(i)}), \qand \mu_i = \rstep_{2i}(\eta^{(i)}).
\end{equation}
Since \( f \) has finitely many nonzero entries, the smallest integer
\( s \) such that \( \eta^{(s)}_i = 0 \) for all \( i \geq 2s \) is
well-defined. The sequence \( \eta^{(s)} \) is a frame sequence, and
from this together with \( (\mu_0,\dots,\mu_{s-1}) \), one can
immediately obtain the \( k \)-multipartition \( \bmu \). Therefore,
we define \( \Gamma(f) := (\bmu, \fs(\bmu)) \). As an example, let
\( f = (4,0,0,3,0,1,2,1,1,2,1,2,0,3,1,0,0,1,0, \cdots ) \in
\mathcal{A}_4 \). Then \( \Gamma(f) \) is obtained by reversing the
procedure described in \Cref{exa:2}.

\medskip
We conclude this section with the main result: the map
\( \Lambda : \mathcal{P}_k \to \mathcal{A}_k \) is a bijection and its
inverse is given by \( \Gamma \). This was proved in~\cite{DJK}, but the proof there was quite tedious. With the new formulation of this paper, the proof is only a few lines.

\begin{thm}[~\cite{DJK}]\label{thm:main bijection}
  The map \( \Lambda : \mathcal{P}_k \to \mathcal{A}_k \) is a
  size-preserving bijection, with inverse map \( \Gamma \). More
  precisely, let \( \bla \) be a multipartition with
  \( \ell(\bla) = s \), and let
  \( (\theta^{(s)},\dots,\theta^{(0)}) \) be the sequence of frequency
  sequences obtained in the process~\eqref{eq:theta} of applying
  \( \Lambda \) to \( (\bla, \fs(\bla)) \). Let
  \( (\eta^{(0)},\dots,\eta^{(t)}) \) be the sequence of frequency
  sequences obtained in the process~\eqref{eq:eta} of applying
  \( \Gamma \) to \( \Lambda(\bla, \fs(\bla))\). Then we have
  \( s = t \), and \( \theta^{(i)} = \eta^{(i)} \) for each
  \( i = 0,\dots,s \).
\end{thm}
\begin{proof}
  Recall that the sequence \( (\theta^{(s)},\dots,\theta^{(0)}) \) of
  frequency sequences is defined by
  \[
    \fs(\bla) =: \theta^{(s)} , \theta^{(s-1)} , \dots  , \theta^{(1)}, \theta^{(0)} := \Lambda(\bla, \fs(\bla)),
  \]
  where
  \begin{align*}
    \theta^{(i)} = \ppm_{2i}^{(\lambda_{i})}(\theta^{(i+1)}) \quad\mbox{for \( i = s -1 , \dots,1,0 \)}. 
  \end{align*}

  Assume that \( (\theta^{(i+1)}_{2i}, \theta^{(i+1)}_{2i+1}) \) moves
  to \( (\theta^{(i)}_v, \theta^{(i)}_{v+1}) \) via the map
  \( \ppm_{2i}^{(\lambda_i)} \). Then, by \Cref{pro:property of
    Lambda}, \( v \) is the smallest index \( u \geq 2i \) such that
  \( \theta^{(i)}_u + \theta^{(i)}_{u+1} = \max\{ \theta^{(i)}_u +
  \theta^{(i)}_{u+1} : u \geq 2i\} \). By the construction of
  \( \Gamma \), it follows that
  \[
    \rpm_{2i}(\theta^{(i)}) = \theta^{(i+1)}, \qand \rstep_{2i}(\theta^{(i)}) = \lambda_{i} .
  \]
  The converse follows directly from the construction. Therefore, for
  each \( i = 0,\dots,s-1 \), we have
  \[
    \theta^{(i)} = \ppm_{2i}^{(\lambda_{i})}(\theta^{(i+1)})
    \Longleftrightarrow
    \left(\rpm_{2i}(\theta^{(i)}), \rstep_{2i}(\theta^{(i)})\right)
    = \left(\theta^{(i+1)}, \lambda_{i}\right) .
  \]
  Since every step in the construction of \( \Lambda \) is invertible,
  the map \( \Lambda \) is a bijection, with the inverse map
  \( \Gamma \). Moreover, the following holds:
  \[
    |\Lambda(\bla, \fs(\bla))| = |\fs(\bla)| + \sum_{i = 0 }^{s-1} \lambda_i = |\fs(\bla)| + |\bla| = |(\bla, \fs(\bla))|.
  \]
  Hence, \( \Lambda \) is size-preserving, which completes the
  proof.
\end{proof}

\section{A combinatorial proof of \Cref{thm:Sta3.2}}
\label{sec:combi-proof-1.2}

Our strategy is as follows. We first define a subset
\( \mathcal{X}_{j,r,k} \subseteq \mathcal{P}_k \) whose generating
function corresponds to the sum side of \Cref{thm:Sta3.2}
(\Cref{pro:gf X}). On the other hand, we define a subset
\( \mathcal{Y}_{j,r,k} \subseteq \mathcal{A}_k \) whose generating
function corresponds to the product side of \Cref{thm:Sta3.2}
(\Cref{pro:gf Y}), using the Andrews--Gordon identities. However, the
set \( \mathcal{Y}_{j,r,k} \) is an artifact introduced solely
to match the desired generating function. To connect
\( \mathcal{X}_{j,r,k} \) and \( \mathcal{Y}_{j,r,k} \), we define a
new subset \( \mathcal{Z}_{j,r,k} \subseteq \mathcal{A}_k \), and show
that there exists a size-preserving bijection between
\( \mathcal{Y}_{j,r,k} \) and
\( \mathcal{Z}_{j,r,k} \subseteq \mathcal{A}_k \)
(\Cref{prop:bij_YZ}). We then prove that the map \( \Lambda \) gives a
bijection between \( \mathcal{X}_{j,r,k} \) and
\( \mathcal{Z}_{j,r,k} \) (\Cref{pro:bij_XZ}). Combining these results
gives a partition interpretation of the identity, and this also
provides a combinatorial proof of \Cref{thm:Sta3.2}.

\subsection{The sum side of \Cref{thm:Sta3.2}}
\label{sec:sum-side}

\begin{defn}\label{def:X}
  Let \( j,r \), and \( k \) be non-negative integers with
  \( j+r \leq k \). Define \( \mathcal{X}_{j,r,k} \) to be the set of
  all pairs \( (\bla, \fs(\bla))\), where
  \( \bla = (\lambda^{(1)},\dots,\lambda^{(k)}) \) is a
  \( k \)-multipartition and \( \fs(\bla) \) is the frame sequence
  corresponding to \( \bla \), subject to the condition that each part
  of \( \lambda^{(m)} \) is at least \( m - j + \max\{m-(k-r), 0\} \)
  for each \( m = 1,\dots,k \).
\end{defn}
The size of \( (\bla,(f_i)_{i \geq 0}) \in \mathcal{X}_{j,r,k} \) is
defined by \( \sum_{i=1}^{k}|\lambda^{(i)}| + \sum_{i \geq 0}if_i  \).
We give a combinatorial model for the left-hand side of the equation
in \Cref{thm:Sta3.2}. We use the following facts. A simple calculation
(as shown in \cite[(2.15)]{DJK}) shows that the weight of
\( \fs(s_1,\dots,s_k) \) is
\begin{equation}\label{eq:wt_fs}
  |\fs(s_1,\dots,s_k)| = s_1^2 + \cdots + s_k^2 - (s_1 + \cdots + s_k).
\end{equation}
Let \( P_{\ell,m}(n) \) be the number of partitions of \( n \) of
length \( \ell \) into parts at least \( d \). Then
\begin{equation}\label{eq:p_lm}
  \sum_{n \geq 0}P_{\ell, d}(n) q^n = \frac{q^{d \ell}}{(q)_\ell}.
\end{equation}

\begin{prop}\label{pro:gf X}
  For non-negative integers \( j,r \), and \( k \) with
  \( j+r \leq k \), we have
  \[
    \sum_{(\bla, \fs(\bla)) \in \mathcal{X}_{j,r,k}} q^{|(\bla,\fs(\bla))|} = \sum_{s_1 \geq \dots \geq s_k \geq 0}\frac{q^{s_1^2 + \cdots + s_k^2 - (s_1 + \cdots + s_j) + (s_{k-r+1} + \cdots + s_k)}}{(q)_{s_1-s_2} \cdots (q)_{s_{k-1}-s_k}(q)_{s_k}}.
  \]
\end{prop}
\begin{proof}
  For non-negative integers \( s_1,\dots,s_k \) with
  \( s_1 \ge \cdots \ge s_k \geq 0 \), define
  \( X_{j,r}(s_1,\dots,s_k) \) to be the set of
  \( k \)-multipartitions
  \( \bla = (\lambda^{(1)},\dots,\lambda^{(k)}) \) such that, for each
  \( m \),
  \begin{itemize}
  \item the length of \( \lambda^{(m)} \) is \( s_m - s_{m+1} \), and
  \item each part of \( \lambda^{(m)} \) is at least
    \( m - j + \max\{m-(k-r), 0\} \); that is,
    \[
      \lambda^{(m)}_1 \ge \cdots \ge \lambda^{(m)}_{s_m - s_{m+1}} \geq m - j + \max\{m-(k-r), 0\}.
    \]
  \end{itemize}
  Let \( \fs(s_1,\dots,s_k) \) be the frequency sequence corresponding
  to \( k \)-multipartitions in \( X_{j,r}(s_1,\dots,s_k) \). We can
  express \( \mathcal{X}_{j,r,k} \) as
  \[
    \mathcal{X}_{j,r,k}=\bigsqcup_{s_1 \geq \dots \geq s_k \geq 0} \{\fs(s_1,\dots,s_k)\} \times X_{j,r}(s_1,\dots,s_k).
  \]
  By \eqref{eq:p_lm}, we have
  \begin{align*}
    \sum_{\bla \in X_{j,r}(s_1,\dots,s_k)} q^{|\lambda|}
    &= \prod_{m=1}^j \frac{1}{(q)_{s_m - s_{m+1}}} \prod_{m = j+1}^{k-r} \frac{q^{(m-j)(s_m - s_{m+1})}}{(q)_{s_m - s_{m+1}}} \prod_{m = k-r+1}^{k-1} \frac{q^{(2m-k+r-j)(s_m - s_{m+1})}}{(q)_{s_m - s_{m+1}}} \frac{q^{(k+r-j)s_k}}{(q)_{s_k}} \\
    &= \frac{q^{(s_{j+1} + \cdots + s_{k-r}) + 2(s_{k-r+1} + \cdots + s_k)}}{(q)_{s_1-s_2} \cdots (q)_{s_{k-1}-s_k}(q)_{s_k}}.
  \end{align*}
  This, together with \eqref{eq:wt_fs}, completes the proof.
\end{proof}

\subsection{The product side of \Cref{thm:Sta3.2}}
\label{sec:prod-side}

\begin{defn}\label{def:Y}
  For non-negative integers \( j,r \), and \( k \) with
  \( j+r \leq k \), we define \( \mathcal{Y}_{j,r,k} \) to be the set
  of all frequency sequences \( (f_i)_{i \geq 0} \) such that
  \( f_i + f_{i+1} \leq k \) for all \( i \geq 0 \), subject to the
  condition that
  \[
    f_0 \in \{\ell+\max\{\ell-(j-r),0 \} : 0 \leq \ell \leq j\}.
  \]
  We also define \( Y_{s,k} \) to be the set of all frequency
  sequences \( (f_i)_{i \geq 0} \) such that
  \( f_i + f_{i+1} \leq k \) for all \( i \), and \( f_0 = s \).
\end{defn}
Using Theorem \ref{thm:AG}, we immediately have the following.

\begin{lem}\label{lem:Y}
  For non-negative integers \( s \) and \( k \) with
  \( 0 \leq s \leq k \), we have
  \[
    \sum_{f \in Y_{s,k}}q^{|f|} = \frac{(q^{2k+3},q^{k+1-s},q^{k+2+s};q^{2k+3})_\infty}{(q)_\infty}.
  \]
\end{lem}

We give a combinatorial model for the right-hand side of the equation
in \Cref{thm:Sta3.2}.
\begin{prop}\label{pro:gf Y}
  For non-negative integers \( j,r \), and \( k \) with
  \( j+r \leq k \), we have
  \[
    \sum_{f \in \mathcal{Y}_{j,r,k}} q^{|f|}  
    = \sum_{s=0}^j \frac{(q^{2k+3},q^{k+1-r+j-2s}, q^{k+2+r-j+2s} ; q^{2k+3})_\infty}{(q)_\infty}.
  \]
\end{prop}
\begin{proof}
  If \( j \leq r \), then the values
  \( \ell+\max\{\ell-(j-r),0 \} = r-j+ 2 \ell \) for
  \( 0 \leq \ell \leq j \), are between \( 0 \) and \( k \). The proof
  follows immediately from \Cref{lem:Y} and
  \[
    \{\ell+\max\{\ell-(j-r),0 \} : 0 \leq \ell \leq j\} = \{ r-j+2s: 0 \leq s \leq j\}.
  \]
  Now suppose \( j>r \).
  Then the set
  \[
    \{\ell+\max\{\ell-(j-r),0 \} : 0 \leq \ell \leq j\}
    = \{ 0, 1,\dots,j-r,j-r+2,\dots,j+r \} 
  \]
  can be expressed as the disjoint union of the two sets
  \[
    \{ j+r, j+r-2, j+r-4 , \cdots  \} \sqcup \{ j-r-1, j-r-3, j-r-5 , \cdots  \}.
  \]
  The first set is given by
  \[
    \{r-j+2s : s = \lfloor (j-r)/2 \rfloor,\lfloor (j-r)/2 \rfloor+1,\dots,j\},
  \]
  and the second set is
  \[
    \{j-r-1-2s: s = 0,1,\dots,\lfloor (j-r)/2 \rfloor -1\}.
  \]
  We divide \( \mathcal{Y}_{j,r,k} \) into two parts depending on
  whether \( f_0 \) belongs to the first set or the second set. By
  \Cref{lem:Y}, the corresponding generating functions for these two
  parts are, respectively,
  \begin{align}
\label{eq:first}    \sum_{s=\lfloor (j-r)/2 \rfloor}^j \frac{(q^{2k+3}, q^{k+1-(r-j+2s)}, q^{k+2+(r-j+2s)}; q^{2k+3})_\infty}{(q)_\infty},
  \end{align}
  and
  \begin{multline}\label{eq:second}
    \sum_{s=0}^{\lfloor (j-r)/2 \rfloor -1} \frac{(q^{2k+3}, q^{k+1-(j-r-1-2s)}, q^{k+2+(j-r-1-2s)} ; q^{2k+3})_\infty}{(q)_\infty}\\
    =
    \sum_{s=0}^{\lfloor (j-r)/2 \rfloor -1} \frac{(q^{2k+3}, q^{k+1-r+j-2s}, q^{k+2+r-j+2s} ; q^{2k+3})_\infty}{(q)_\infty},
  \end{multline}
  where the equality in \eqref{eq:second} follows from switching the
  first two terms in the Pochhammer symbol in the numerator. Adding
  \eqref{eq:first} and \eqref{eq:second} completes the proof.
\end{proof}

\subsection{A combinatorial proof of \Cref{thm:Sta3.2}}
\label{sec:proof3.2}

\begin{defn}\label{def:Z}
  For non-negative integers \( j,r \), and \( k \) with
  \( j+r \leq k \), we define \( \mathcal{Z}_{j,r,k} \) to be the set
  of all frequency sequences \( (f_i)_{i \geq 0} \) such that
  \( f_i + f_{i+1} \leq k \) for all \( i \geq 0 \), subject to the
  condition that
  \[
    f_0 \leq j - \max\{f_0 + f_1 - (k-r),0\}.
  \]
\end{defn}

Note that this condition is equivalent to
\[
  f_0 \leq j \qand 2f_0 + f_1 \leq k-r+j.
\]

\begin{prop}\label{prop:bij_YZ}
  There exists a size-preserving bijection from
  \( \mathcal{Y}_{j,r,k} \) to \( \mathcal{Z}_{j,r,k} \).
\end{prop}
\begin{proof}
  The additional conditions on \( \mathcal{Y}_{j,r,k} \) and
  \( \mathcal{Z}_{j,r,k} \) are given respectively by
  \[
    f_0 \in \{\ell+\max\{\ell-(j-r),0 \} : 0 \leq \ell \leq j\}, \qand f_0 \leq j - \max\{f_0 + f_1 - (k-r),0\}.
  \]
  We construct a bijection between \( \mathcal{Y}_{j,r,k} \) and
  \( \mathcal{Z}_{j,r,k} \) by modifying only the value of \( f_0 \).

  Suppose that \( j \geq r \). Then the condition on
  \( \mathcal{Y}_{j,r,k} \) is
  \( f_0 \in \{ 0,1,\dots,j-r, j-r+2,\dots,j+r\} \). Define a map
  \( \phi \) on \( \mathcal{Y}_{j,r,k} \) by
  \[
    (f_0,f_1,f_2, \dots) \mapsto (f_0', f_1, f_2,\dots),
  \]
  where
  \[
    f_0' =
    \begin{cases}
      f_0 & \mbox{if \( f_0 \leq j-r \)},\\
      j-r+\ell & \mbox{if \( f_0 = j-r + 2 \ell \) for some \( \ell \in \{1,\dots,r\} \).}
    \end{cases}
  \]
  We first claim that
  \( \phi(\mathcal{Y}_{j,r,k}) \subseteq \mathcal{Z}_{j,r,k} \). Since
  all entries except the first one remain unchanged,
  \( f_i + f_{i+1} \leq k \) for \( i \geq 1 \). For \( i = 0 \), we
  have \( f_0' + f_1 \leq f_0 + f_1 \leq k \). It remains to show that
  \( f_0' \leq j - \max\{f_0' + f_1 - (k-r), 0 \} \). We verify this
  inequality case by case. If \( f_0' + f_1 \leq k-r \), then
  \( \max\{f_0' + f_1 - (k-r), 0 \} = 0 \). By construction of
  \( \phi \), we have \( f_0' \leq j \), hence,
  \( f_0' \leq j - \max\{f_0' + f_1 - (k-r), 0 \} \) holds. Suppose
  \( f_0' + f_1 > k-r \). If \( f_0 \leq j-r \), then \( f_0' = f_0 \)
  and we have
  \begin{align*}
    f_0' + \max\{f_0' + f_1 - (k-r), 0 \}
    &= 2f_0 + f_1 - (k-r) \\
    &= f_0 + (f_0 + f_1) - (k-r)\\
    &\leq (j-r) + k - (k-r) = j.
  \end{align*}
  If \( f_0 > j-r \), then \( f_0 = j-r+ 2 \ell \) and
  \( f_0' = j-r+ \ell \) for some \( \ell \in \{1,\dots,r\}\). We have
  \begin{align*}
    f_0' + \max\{f_0' + f_1 - (k-r), 0 \}
    &= 2f_0' + f_1 - (k-r) \\
    &= 2(j-r+ \ell) + f_1 - (k-r)\\
    &= f_0 + (j-r) + f_1 - (k-r)\\
    &= (f_0 + f_1) - k + j \\
    &\leq k-k+j = j.
  \end{align*}
  Hence, the claim follows.

  We now construct the inverse map of \( \phi \). Define a map
  \( \pi \) on \( \mathcal{Z}_{j,r,k} \) by
  \[
    (g_0,g_1,g_2, \dots) \mapsto (g_0', g_1, g_2,\dots),
  \]
  where
  \[
    g_0' =
    \begin{cases}
      g_0 & \mbox{if \( g_0 \leq j-r \)},\\
      j-r+2\ell & \mbox{if \( g_0 = j-r + \ell \) for some \( \ell \in \{1,\dots,r\}\).}
    \end{cases}
  \]
  Similarly, we claim that
  \( \pi(\mathcal{Z}_{j,r,k}) \subseteq \mathcal{Y}_{j,r,k} \). Since
  \( (g_0, g_1,\dots) \in \mathcal{Z}_{j,r,k} \), we have
  \( g_i + g_{i+1} \leq k \) for \( i \geq 1 \). If
  \( g_0 \leq j-r \), then
  \[
    g_0' + g_1 = g_0 + g_1 \leq k.
  \]
  If \( g_0 = j - r + \ell \), then \( g_0' = j - r + 2 \ell \) and
  \[
    g_0' + g_1 = (j-r+2 \ell) + g_1 = 2g_0 -(j-r) + g_1 \leq k,
  \]
  where the last equality follows from the condition of
  \( (g_0, g_1,\dots) \in \mathcal{Z}_{j,r,k} \) that
  \( 2g_0 + g_1 \geq k-r+j \). By the construction of \( \pi \), we
  have
  \[
    g_0' \in \{\ell+\max\{\ell-(j-r),0 \} : 0 \leq \ell \leq j\}.
  \]
  Hence, the claim follows. Since the map
  \( \pi : \mathcal{Z}_{j,r,k} \to \mathcal{Y}_{j,r,k} \) is the
  inverse of \( \phi \), it follows that \( \phi \) is a bijection
  from \( \mathcal{Y}_{j,r,k} \) to \( \mathcal{Z}_{j,r,k} \).
  Moreover, since this map changes only the \( 0 \)th entry, it does
  not affect the size of the partition; hence, it is
  size-preserving.
  
  The case \( j < r \) is proved in a similar way.
  Suppose that \( j < r \).
  Then \( f_0 \in \{ r-j, r-j+2,\dots, r+j\} \).
  Define a map \( \phi \) on \( \mathcal{Y}_{j,r,k} \) by
  \[
    (f_0,f_1,f_2, \dots) \mapsto (f_0', f_1, f_2,\dots),
  \]
  where \( f_0' = \ell \) for \( f_0 = r-j+2 \ell \) with
  \( \ell \in \{0,\dots,j\}\). Then, we have
  \( \phi(\mathcal{Y}_{j,r,k}) \subseteq \mathcal{Z}_{j,r,k} \), since
  \[
    f_0' + f_1 = \ell + f_1 \leq 2 \ell + (r-j) + f_1 = f_0 + f_1 \leq k,
  \]
  and
  \begin{align*}
    f_0' + \max\{f_0' + f_1 - (k-r), 0 \}
    &\leq 2f_0' + f_1 - (k-r) \\
    &= 2 \ell + f_1 - (k-r)\\
    &= f_0 - (r-j) + f_1 - (k-r)\\
    &= (f_0 + f_1) - k + j \\
    &\leq k-k+j = j.
  \end{align*}
  The inverse map \( \pi \) on \( \mathcal{Z}_{j,r,k} \) is defined by
  \[
    (g_0,g_1,g_2, \dots) \mapsto (g_0', g_1, g_2,\dots),
  \]
  where \( g_0' = r-j+2 \ell \) for \( g_0 = \ell \) with
  \( \ell \in \{0,\dots,j\}\). Then we have
  \( \pi(\mathcal{Z}_{j,r,k}) \subseteq \mathcal{Y}_{j,r,k} \) since
  \[
    g_0' \in \{ r-j, r-j+2,\dots, r+j\},
  \]
  and
  \[
    g_0' + g_1 = r-j+2 \ell + g_1 = r-j + 2g_0 + g_1 \leq k,
  \]
  where the last equality follows from the condition of
  \( (g_0, g_1,\dots) \in \mathcal{Z}_{j,r,k} \) that
  \[
    j \geq g_0 + \max\{g_0 + g_1 - (k-r) , 0\} \geq 2g_0 + g_1 - (k-r).
  \]
  The map \( \phi \) is a bijection from \( \mathcal{Y}_{j,r,k} \) to
  \( \mathcal{Z}_{j,r,k} \) in the case \( j < r \) as well.
\end{proof}

\begin{prop}\label{pro:bij_XZ}
  The map \( \Lambda \) is a size-preserving bijection from
  \( \mathcal{X}_{j,r,k} \) to \( \mathcal{Z}_{j,r,k} \).
\end{prop}

It is shown in \Cref{thm:main bijection} that the map \( \Lambda \) is
size-preserving and has an inverse map \( \Gamma \). To complete the
proof, it suffices to show that
\( \Lambda(\mathcal{X}_{j,r,k}) \subseteq \mathcal{Z}_{j,r,k} \) and
\( \Gamma(\mathcal{Z}_{j,r,k}) \subseteq \mathcal{X}_{j,r,k} \), which
we prove in Lemmas~\ref{lem:Lambda} and \ref{lem:Gamma}.

\begin{lem}\label{lem:Lambda}
  Let \( (\bla, \fs(\bla)) \in \mathcal{X}_{j,r,k} \) with
  \( \ell(\bla) = s \). Suppose that
  \( (\theta^{(s)},\dots,\theta^{(0)}) \) denotes the sequence of
  frequency sequences obtained recursively from
  \( (\bla, \fs(\bla))\) via the map \( \Lambda \), as in
  \eqref{eq:theta}. Then, for all \( i \in \{s,\dots,0\} \), we have
  \begin{equation}\label{eq:cond_theta}
    \theta_{2i}^{(i)} \leq j - \max\{\theta_{2i}^{(i)} + \theta_{2i+1}^{(i)} - (k-r), 0 \}.
  \end{equation}
  Moreover, \( \Lambda(\mathcal{X}_{j,r,k}) \subseteq \mathcal{Z}_{j,r,k} \).
\end{lem}
\begin{proof}
  The proof follows a similar approach to that of
  \cite[Proposition~4.1]{DJK}, using backward induction on
  \( i \in \{ s,\dots,0\} \). The base case \( i = s \) holds
  trivially, since \( \theta^{(s)} = \fs(\bla) \) and the both entries
  at positions \( 2s \) and \( 2s+1 \) in the sequence \( \fs(\bla) \)
  are zero. Assume that
  \( \theta_{2i+2}^{(i+1)} \leq j - \max\{\theta_{2i+2}^{(i+1)} +
  \theta_{2i+3}^{(i+1)} - (k-r), 0 \} \) holds. Recall
  \eqref{eq:recursion_theta} that
  \[
    \theta^{(i)} = \ppm_{2i}^{(\lambda_i)} \left( \theta^{(i+1)} \right).
  \]
  Suppose that
  \( (\theta^{(i+1)}_{2i},\theta^{(i+1)}_{2i+1}) = (h,0) \) for some
  \( h \geq 1 \). Then \( \lambda_i \) is a part of the partition
  \( \lambda^{(h)} \). By the condition of \( \mathcal{X}_{j,r,k} \),
  we have \( \lambda_i \geq h - j + \max\{h - (k-r),0\} \). It
  suffices to prove that
  \( \theta_{2i}^{(i)} \leq j - \max\{\theta_{2i}^{(i)} +
  \theta_{2i+1}^{(i)} - (k-r), 0 \} \), which is equivalent to showing
  that \( \theta_{2i}^{(i)} \leq j \) and
  \( 2\theta_{2i}^{(i)} + \theta_{2i+1}^{(i)} \leq k-r+j \).

  Let \( v \) be the value defined in \eqref{eq:u_bar}, so that
  \( (\theta^{(i+1)}_{2i},\theta^{(i+1)}_{2i+1}) \) moves to
  \( (\theta^{(i)}_{v-2},\theta^{(i)}_{v-1}) \). We use the explicit
  formula \eqref{eq:f_bar} for \( \theta^{(i)} \) to determine
  \( \theta^{(i)}_{2i} \) and \( \theta^{(i)}_{2i+1} \) .

\begin{itemize}
\item If \( v > 2i+3 \), then
  \( (\theta_{2i}^{(i)}, \theta_{2i+1}^{(i)}) =
  (\theta_{2i+2}^{(i+1)}, \theta_{2i+3}^{(i+1)}) \), and the claim
  follows immediately from the induction hypothesis.
\item If \( v = 2i+3 \), then 
  \begin{align*}
    \theta_{2i}^{(i)} &= \theta_{2i+2}^{(i+1)}, \\
    \theta_{2i+1}^{(i)} &= \theta^{(i+1)}_{2i+3} + (h-\theta_{2i+1}^{(i+1)}-\theta_{2i+2}^{(i+1)}) + (h-\theta_{2i+2}^{(i+1)}-\theta_{2i+3}^{(i+1)}) - \lambda_i\\
                      &= 2h - \lambda_i - \theta_{2i+1}^{(i+1)} - 2\theta_{2i+2}^{(i+1)}.
  \end{align*}
  We have \( \theta_{2i}^{(i)} = \theta_{2i+2}^{(i+1)} \leq j \), by
  the induction hypothesis. Using two facts that
  \( \theta_{2i+1}^{(i+1)} = 0 \) and
  \( \lambda_i \geq h - j + \max\{ h - (k-r), 0 \} \), we have
  \begin{align*}
    2\theta_{2i}^{(i)} + \theta_{2i+1}^{(i)}
    &= 2\theta_{2i+2}^{(i+1)} + (2h - \lambda_i - \theta_{2i+1}^{(i+1)} - 2\theta_{2i+2}^{(i+1)}) \\
    &= 2h - \lambda_i \\
    &\leq 2h - (h - j + \max\{ h - (k-r), 0 \}) \\
    &\leq k-r+j.
  \end{align*}
\item If \( v = 2i + 2 \), then
  \begin{align*}
    \theta_{2i}^{(i)} &= \theta_{2i+2}^{(i+1)} + (h-\theta_{2i+1}^{(i+1)}-\theta_{2i+2}^{(i+1)}) - \lambda_i = h - \lambda_i, \\
    \theta_{2i+1}^{(i)} &= \theta_{2i+1}^{(i+1)} + \lambda_i = \lambda_i.
  \end{align*}
  Similarly, we have
  \[
    \theta_{2i}^{(i)} = h - \lambda_i \leq h - (h - j + \max\{ h - (k-r), 0 \}) \leq j,
  \]
  and
  \[
    2\theta_{2i}^{(i)} + \theta_{2i+1}^{(i)} = 2h - \lambda_i \leq k-r+j.
  \]
\end{itemize}
Therefore, in all cases, we have \( \theta^{(i)}_{2i} \leq j \) and
\( 2\theta^{(i)}_{2i} + \theta^{(i)}_{2i+1} \leq k-r+j \). This
completes the proof by induction. Moreover, by \eqref{eq:cond_theta}
with \( i = 0 \), we obtain
\( \Lambda(\mathcal{X}_{j,r,k}) \subseteq \mathcal{Z}_{j,r,k} \).

\end{proof}

\begin{lem}\label{lem:Gamma}
  Let \( f \in \mathcal{Z}_{j,r,k} \). Suppose that
  \( (\eta^{(0)},\dots,\eta^{(s)}) \) denotes the sequence of
  frequency sequences obtained recursively from \( f \) via the map
  \( \Gamma \), as in \eqref{eq:eta}.
  Then, for all \( i \in \{0,\dots,s\} \), we have
  \begin{equation}\label{eq:condition eta}
    \eta_{2i}^{(i)} \leq j - \max\{\eta_{2i}^{(i)} + \eta_{2i+1}^{(i)} - (k-r), 0 \}.
  \end{equation}
  Moreover,
  \( \Gamma(\mathcal{Z}_{j,r,k}) \subseteq \mathcal{X}_{j,r,k} \).
\end{lem}
\begin{proof}
  The proof follows a similar approach to \cite[Proposition~4.3 and
  Corollary~4.4]{DJK}, using induction on
  \( i \in \{0,\dots,s\} \). The base case \( i = 0 \) holds clearly
  from \( \eta^{(0)} = f \in \mathcal{Z}_{j,r,k} \). Assume that
  \( \eta_{2i}^{(i)} \leq j - \max\{\eta_{2i}^{(i)} +
  \eta_{2i+1}^{(i)} - (k-r), 0 \} \) holds. To obtain
  \( \eta^{(i+1)} \) from \( \eta^{(i)} \), set
  \[
    h = \max\{\eta^{(i)}_j + \eta^{(i)}_{j+1} : j \geq 2i\}, \qand v = \min\{j \geq 2i+2 : \eta^{(i)}_{j-2} + \eta^{(i)}_{j-1} = h \},
  \]
  and recall from \eqref{eq:eta} that
  \( \eta^{(i+1)} = \rpm_{2i} ( \eta^{(i)} )\) and
  \( \mu_i = \rstep_{2i} ( \eta^{(i)} ) \). We first prove that
  \( \eta_{2i+2}^{(i+1)} \leq j - \max\{\eta_{2i+2}^{(i+1)} +
  \eta_{2i+3}^{(i+1)} - (k-r), 0 \}\), which is equivalent to showing
  that \( \eta_{2i+2}^{(i+1)} \leq j \) and
  \( 2\eta_{2i+2}^{(i+1)} + \eta_{2i+3}^{(i+1)} \leq k-r+j \). We use
  the explicit formula \eqref{eq:rev_f_bar} to compute
  \( \eta^{(i+1)}_{2i+2} \) and \( \eta^{(i+1)}_{2i+3} \).
  \begin{itemize}
  \item If \( v > 2i+3 \), then
    \( (\eta_{2i+2}^{(i+1)}, \eta_{2i+3}^{(i+1)}) = (\eta_{2i}^{(i)},
    \eta_{2i+1}^{(i)}) \), and the claim follows immediately from the
    induction hypothesis.
  \item If \( v = 2i+3 \), then
    \( (\eta_{2i+2}^{(i+1)}, \eta_{2i+3}^{(i+1)}) = (\eta_{2i}^{(i)},
    \eta_{2i+3}^{(i)}) \). The inequality
    \( \eta_{2i+2}^{(i)} + \eta_{2i+3}^{(i)} \leq h =
    \eta_{2i+1}^{(i)} + \eta_{2i+2}^{(i)} \) implies
    \( \eta_{2i+3}^{(i)} \leq \eta_{2i+1}^{(i)} \). Using this
    together with the induction hypothesis, we have
    \begin{align*}
      \eta_{2i+2}^{(i+1)} &= \eta_{2i}^{(i)} \leq j, \\
      2\eta_{2i+2}^{(i+1)} + \eta_{2i+3}^{(i+1)} &\leq 2\eta_{2i}^{(i)} + \eta_{2i+1}^{(i)} \leq k-r+j.
    \end{align*} 
  \item If \( v = 2i+2 \),
    then\( (\eta_{2i+2}^{(i+1)}, \eta_{2i+3}^{(i+1)}) =
    (\eta_{2i+2}^{(i)}, \eta_{2i+3}^{(i)}) \). Using inequalities
    \( \eta_{2i+1}^{(i)} + \eta_{2i+2}^{(i)} \leq h = \eta_{2i}^{(i)}
    + \eta_{2i+1}^{(i)} \) and
    \( \eta_{2i+2}^{(i)} + \eta_{2i+3}^{(i)} \leq h = \eta_{2i}^{(i)}
    + \eta_{2i+1}^{(i)} \) together with the induction hypothesis, we
    have
    \begin{align*}
      \eta_{2i+2}^{(i+1)} &= \eta_{2i+2}^{(i)} \leq \eta_{2i}^{(i)} \leq j, \\
      2\eta_{2i+2}^{(i+1)} + \eta_{2i+3}^{(i+1)} &= 2\eta_{2i+2}^{(i)}
      + \eta_{2i+3}^{(i)} = (\eta_{2i+2}^{(i)} + \eta_{2i+3}^{(i)}) +
      \eta_{2i+2}^{(i)} \leq 2\eta_{2i}^{(i)} + \eta_{2i+1}^{(i)}
      \leq k-r+j.
    \end{align*} 
  \end{itemize}

  Let \( \bmu = (\mu^{(1)},\dots,\mu^{(k)}) \) be the
  \( k \)-partition obtained from \( f \) via the map \( \Gamma \). It
  remains to prove that
  \( (\fs(\bmu), \bmu) \in \mathcal{X}_{j,r,k} \). Since \( \mu_i \)
  is a part of the partition \( \mu^{(h)} \), it suffices to show that
  \( \mu_i = \rstep_{2i}(\eta^{(i)}) \geq h-j + \max\{ h -(k-r),0 \}
  \). We use the formula \eqref{eq:rstep} to compute \( \mu_i \), and
  consider two cases: when \( v \geq 2i+3 \) and when \( v = 2i+2 \).
  If \( v \geq 2i+3 \), then
  \[
    \mu_i = \rstep_{2i}(\eta^{(i)})
    = h - \eta^{(i)}_{2i} + \sum_{j = 2i}^{v-3}(h - (\eta^{(i)}_j + \eta^{(i)}_{j+1}))
    \geq 2h - 2 \eta_{2i}^{(i)} - \eta_{2i+1}^{(i)}.
  \]
  Using \( h \geq \eta_{2i}^{(i)} + \eta_{2i+1}^{(i)} \), and the
  induction hypotheses \( \eta^{(i)}_{2i} \leq j \) and
  \( 2\eta^{(i)}_{2i} + \eta^{(i)}_{2i+1} \leq k-r+j \), we obtain the
  following two inequalities:
  \begin{align*}
    \mu_i &\geq 2h - 2 \eta_{2i}^{(i)} - \eta_{2i+1}^{(i)} \geq h - \eta_{2i}^{(i)} \geq h - j, \quad\text{and}\\
    \mu_i &\geq 2h - 2 \eta_{2i}^{(i)} - \eta_{2i+1}^{(i)} \geq 2h - (k-r+j) = h - j + (h - (k-r)).
  \end{align*}
  Combining the two inequalities above, we obtain
  \( \mu_i \geq h-j+\max\{h-(k-r),0\} \). If \( v = 2i+2 \), then
  \( \mu_i = h - \eta_{2i}^{(i)} \) and
  \( h = \eta_{2i}^{(i)} + \eta_{2i+1}^{(i)} \). Hence, by
  \eqref{eq:condition eta}, we have
  \begin{align*}
    \mu_i
    &= h - \eta_{2i}^{(i)}\\
    &\geq h -j + \max\{ \eta_{2i}^{(i)} + \eta_{2i+1}^{(i)} - (k-r), 0 \} \\
    &=  h -j + \max\{ h - (k-r), 0 \},
  \end{align*}
  which completes the proof.
\end{proof}

\section{A combinatorial proof of Theorems~\ref{thm:Sta4.2} and~\ref{thm:new_Kur}}
\label{sec:proof3.4}

The combinatorial proof of \Cref{thm:Sta4.2} follows the same approach
as that of \Cref{thm:Sta3.2}, but with Bressoud's identity in place of
Andrews--Gordon's identity as key ingredient. The sum and product
sides are obtained as the generating functions of new sets
\( \mathcal{X}_{j,r,k}' \) and \( \mathcal{Z}_{j,r,k}' \),
respectively. We then define a new set \( \mathcal{Y}_{j,r,k}' \), and
give bijections between \( \mathcal{X}_{j,r,k}' \) and
\( \mathcal{Y}_{j,r,k}' \), and between \( \mathcal{Y}_{j,r,k}' \) and
\( \mathcal{Z}_{j,r,k}' \). However, the case of the Kur\c sung\"oz
type identities is somewhat different. One can define
\( \widetilde{\mathcal{X}}_{j,r,k}' \),
\( \widetilde{\mathcal{Y}}_{j,r,k}' \), and
\( \widetilde{\mathcal{Z}}_{j,r,k}' \) each satisfying the opposite
parity condition, as analogues of \( \mathcal{X}_{j,r,k}' \),
\( \mathcal{Y}_{j,r,k}' \), and \( \mathcal{Z}_{j,r,k}' \). While the
sum side arises as the generating function of
\( \widetilde{\mathcal{X}}_{j,r,k}' \), the product side does not
correspond to the generating function of
\( \widetilde{\mathcal{Z}}_{j,r,k}' \). Unlike in the previous cases,
\( \widetilde{\mathcal{Z}}_{j,r,k}' \) is not in bijection with
\( \widetilde{\mathcal{Y}}_{j,r,k}' \). However, by using its relation
to \( \mathcal{Z}_{j,r,k}' \), we can determine the generating
function for \( \widetilde{\mathcal{Z}}_{j,r,k}' \), which in turn
yields a new identity of the Kur\c sung\"oz type, namely Theorem
\ref{thm:new_Kur}.

\begin{defn}\label{def:X'}
  Let \( j,r \), and \( k \) be non-negative integers with
  \( j+r \leq k \). Define \( \mathcal{X}_{j,r,k}' \) (resp.
  \( \widetilde{\mathcal{X}}_{j,r,k}' \)) to be the set of all pairs
  \( (\bla,\fs(\bla)) \), where
  \( \bla = (\lambda^{(1)},\dots,\lambda^{(k)}) \) is a
  \( k \)-multipartition and \( \fs(\bla) \) is the frame sequence
  corresponding to \( \bla \), subject to the conditions that
  \begin{itemize}
  \item each part of the partition \( \lambda^{(m)} \) is at least
    \( m - j + \max\{m-(k-r), 0\} \) for each \( m = 1,\dots,k \), and
  \item each part of the last partition \( \lambda^{(k)} \) has the same
    parity as \( k+r-j \) (resp. \( k+r-j+1 \)).
  \end{itemize}
\end{defn}

\begin{prop}\label{pro:gf X'}
  Let \( j,r \), and \( k \) be non-negative integers with
  \( j+r \leq k \).
  Then we have
  \begin{align}
    \label{eq:gfX'}
    \sum_{(\mu,\lambda) \in \mathcal{X}_{j,r,k}'} q^{|(\mu,\lambda)|}
    &= \sum_{s_1 \geq \dots \geq s_k \geq 0}\frac{q^{s_1^2 + \cdots + s_k^2 - (s_1 + \cdots + s_j) + (s_{k-r+1} + \cdots + s_k)}}{(q)_{s_1-s_2} \cdots (q)_{s_{k-1}-s_k}(q^2;q^2)_{s_k}}, \\
    \label{eq:gfX'til}
    \sum_{(\mu,\lambda) \in \widetilde{\mathcal{X}}_{j,r,k}'} q^{|(\mu,\lambda)|}
    &= \sum_{s_1 \geq \dots \geq s_k \geq 0}\frac{q^{s_1^2 + \cdots + s_k^2 - (s_1 + \cdots + s_j) + (s_{k-r+1} + \cdots + s_{k-1} + 2s_k)}}{(q)_{s_1-s_2} \cdots (q)_{s_{k-1}-s_k}(q^2;q^2)_{s_k}}.
  \end{align}
\end{prop}
\begin{proof}
  We use the following identity. Let \( P_{\ell, d, s}(n) \) be the
  number of partitions \( \lambda \) of \( n \) of length \( \ell \) into
  parts at least \( d \), such that each part of \( \lambda \) has the
  same parity as \( s \). Then
  \begin{equation}\label{eq:p_lmm}
    \sum_{n \geq 0} P_{\ell,d,d}(n) q^n = \frac{q^{d \ell}}{(q^2;q^2)_\ell}, \qand \sum_{n \geq 0} P_{\ell,d,d+1}(n) q^n = \frac{q^{(d+1) \ell}}{(q^2;q^2)_\ell}.
  \end{equation}
  The proofs are the same as that of \Cref{pro:gf X}, except for the
  last term in the generating functions. By \eqref{eq:p_lm} and
  \eqref{eq:p_lmm}, we have
  \begin{align*}
    \sum_{\lambda \in X_{j,r}'(s_1,\dots,s_k)} q^{|\lambda|}
    &= \prod_{m=1}^j \frac{1}{(q)_{s_m - s_{m+1}}} \prod_{m = j+1}^{k-r} \frac{q^{(m-j)(s_m - s_{m+1})}}{(q)_{s_m - s_{m+1}}} \prod_{m = k-r+1}^{k-1} \frac{q^{(2m-k+r-j)(s_m - s_{m+1})}}{(q)_{s_m - s_{m+1}}} \frac{q^{(k+r-j)s_k}}{(q^2;q^2)_{s_k}} \\
    &= \frac{q^{(s_{j+1} + \cdots + s_{k-r}) + 2(s_{k-r+1} + \cdots + s_k)}}{(q)_{s_1-s_2} \cdots (q)_{s_{k-1}-s_k}(q^2;q^2)_{s_k}},
  \end{align*}
  and
  \begin{align*}
    \sum_{\lambda \in \widetilde{X}_{j,r}'(s_1,\dots,s_k)} q^{|\lambda|}
    &= \prod_{m=1}^j \frac{1}{(q)_{s_m - s_{m+1}}} \prod_{m = j+1}^{k-r} \frac{q^{(m-j)(s_m - s_{m+1})}}{(q)_{s_m - s_{m+1}}} \prod_{m = k-r+1}^{k-1} \frac{q^{(2m-k+r-j)(s_m - s_{m+1})}}{(q)_{s_m - s_{m+1}}} \frac{q^{(k+r-j+1)s_k}}{(q^2;q^2)_{s_k}} \\
    &= \frac{q^{(s_{j+1} + \cdots + s_{k-r}) + 2(s_{k-r+1} + \cdots + s_{k-1})+3s_k}}{(q)_{s_1-s_2} \cdots (q)_{s_{k-1}-s_k}(q^2;q^2)_{s_k}}.
  \end{align*}
  This, together with \eqref{eq:wt_fs}, completes the proof.
\end{proof}

\begin{defn}\label{def:Y'}
  For non-negative integers \( j,r \), and \( k \) with
  \( j+r \leq k \), we define \( \mathcal{Y}_{j,r,k}' \) (resp.
  \( \widetilde{\mathcal{Y}}_{j,r,k}' \)) to be the set of all
  frequency sequences \( (f_i)_{i \geq 0} \) such that
  \( f_i + f_{i+1} \leq k \) for all \( i \geq 0 \), subject to the
  conditions that
  \begin{itemize}
  \item \( f_0 \in \{\ell+\max\{\ell-(j-r),0 \} : 0 \leq \ell \leq j\} \), and
  \item if \( f_u + f_{u+1} = k \), then \( u f_u + (u+1) f_{u+1} \) has the same parity as \( k+r-j \) (resp. \( k+r-j+1 \)).
  \end{itemize}
\end{defn}

The following lemma was obtained in \cite{DJK} using Theorem \ref{thm:2}.

\begin{lem}[{\cite[Equation~(2.12) and (2.13)]{DJK}}]\label{lem:Y'}
  Let \( Y_{s,k}' \) (resp. \( \widetilde{Y}_{s,k}' \)) be the set of
  all frequency sequences \( (f_i)_{i \geq 0} \) such that
  \( f_i + f_{i+1} \leq k \) for all \( i \geq 0 \), \( f_0 = s \),
  and if \( f_u + f_{u+1} = k \), then \( u f_u + (u+1) f_{u+1} \) has
  the same parity as \( k-s \) (resp. \( k-s+1 \)).
  Then
  \begin{align}
\label{eq:gf Y'_sk}    \sum_{\lambda \in Y_{s,k}'}q^{|\lambda|} &= \frac{(q^{2k+2}, q^{k+1-s},q^{k+1+s};q^{2k+2})_\infty}{(q)_\infty}, \\
 \label{eq:gf Y'_sk_til}   \sum_{\lambda \in \widetilde{Y}_{s,k}'}q^{|\lambda|} &= \frac{(q^{2k+2}, q^{k-s},q^{k+2+s};q^{2k+2})_\infty}{(q)_\infty}.
  \end{align}
\end{lem}

Using this lemma, we deduce expressions for the generating functions
of \(\mathcal{Y}_{j,r,k}' \) and
\( \widetilde{\mathcal{Y}}_{j,r,k}' \) as sums of infinite products.

\begin{prop}\label{pro:gf Y'}
  For non-negative integers \( j,r \), and \( k \) with
  \( j+r \leq k \), we have
  \begin{align}
    \label{eq:gf Y'}
    \sum_{\lambda \in \mathcal{Y}_{j,r,k}'} q^{|\lambda|}  
    &= \sum_{s=0}^j \frac{(q^{2k+2}, q^{k+1-r+j-2s}, q^{k+1+r-j+2s} ; q^{2k+2})_\infty}{(q)_\infty}, \\
    \label{eq:gf Y'til}
    \sum_{\lambda \in \widetilde{\mathcal{Y}}_{j,r,k}'} q^{|\lambda|}  
    &= \sum_{s=0}^j \frac{(q^{2k+2}, q^{k-r+j-2s}, q^{k+2+r-j+2s} ; q^{2k+2})_\infty}{(q)_\infty}.
  \end{align}
\end{prop}
\begin{proof}
  We use a similar idea to the proof of \Cref{pro:gf Y}, with the key
  difference that, in this case, the parity of \( f_0 \) needs to be
  taken into account. The proof of the second identity \eqref{eq:gf
    Y'til} is essentially the same as that of the first \eqref{eq:gf
    Y'}, so we prove only the first identity.

  If \( j \leq r \), then
  \( \ell+\max\{\ell-(j-r),0 \} = r-j+ 2 \ell \) for
  \( 0 \leq \ell \leq j \). The values of \( f_0 \) are
  \[
    \{\ell+\max\{\ell-(j-r),0 \} : 0 \leq \ell \leq j\} = \{ r-j+2s: 0 \leq s \leq j\},
  \]
  and they all have same parity as \( r-j \). Hence, the proof follows
  immediately from \eqref{eq:gf Y'_sk}. Now suppose \( j>r \). Then
  the set
  \[
    \{\ell+\max\{\ell-(j-r),0 \} : 0 \leq \ell \leq j\}
    = \{ 0, 1,\dots,j-r,j-r+2,\dots,j+r \} 
  \]
  can be expressed as the disjoint union of the two sets
  \[
    \{ j+r, j+r-2, j+r-4 , \cdots  \} \sqcup \{ j-r-1, j-r-3, j-r-5 , \cdots  \}.
  \]
  Every element of the first set
  \( \{r-j+2s : s = \lfloor (j-r)/2 \rfloor,\lfloor (j-r)/2
  \rfloor+1,\dots,j\} \) has the same parity as \( r-j \). On the
  other hand, every element of the second set
  \( \{j-r-1-2s: s = 0,1,\dots,\lfloor (j-r)/2 \rfloor -1\} \) has the
  same parity as \( r-j+1 \). We apply \eqref{eq:gf Y'_sk} to the
  first set and \eqref{eq:gf Y'_sk_til} to the second set. The
  corresponding generating functions for these two parts are,
  respectively,
  \begin{align}
\label{eq:first'}    \sum_{s=\lfloor (j-r)/2 \rfloor}^j \frac{(q^{2k+2}, q^{k+1-(r-j+2s)}, q^{k+1+(r-j+2s)} ; q^{2k+2})_\infty}{(q)_\infty},
  \end{align}
  and
  \begin{multline}\label{eq:second'}
    \sum_{s=0}^{\lfloor (j-r)/2 \rfloor -1} \frac{(q^{2k+2}, q^{k-(j-r-1-2s)}, q^{k+2+(j-r-1-2s)}; q^{2k+2})_\infty}{(q)_\infty}\\
    =
    \sum_{s=0}^{\lfloor (j-r)/2 \rfloor -1} \frac{(q^{2k+2}, q^{k+1-r+j-2s}, q^{k+1+r-j+2s}; q^{2k+2})_\infty}{(q)_\infty}.
  \end{multline}
  Therefore, adding \eqref{eq:first'} and \eqref{eq:second'} completes the proof.
\end{proof}

\begin{defn}\label{def:Z'}
  For non-negative integers \( j,r \), and \( k \) with
  \( j+r \leq k \), we define \( \mathcal{Z}_{j,r,k}' \) (resp.
  \( \widetilde{\mathcal{Z}}_{j,r,k}' \)) to be the set of all
  frequency sequences \( (f_i)_{i \geq 0} \) such that
  \( f_i + f_{i+1} \leq k \) for all \( i \geq 0 \), subject to the
  conditions that
  \begin{itemize}
  \item \( f_0 \leq j - \max\{f_0 + f_1 - (k-r),0\} \), and
  \item if \( f_u + f_{u+1} = k \), then \( u f_u + (u+1) f_{u+1} \) has the same parity as \( k+r-j \) (resp. \( k+r-j+1 \)).
  \end{itemize}
\end{defn}

\begin{prop}\label{pro:bij X' and Z'}
  The map \( \Lambda \) gives a bijection between
  \( \mathcal{X}_{j,r,k}' \) (resp.
  \( \widetilde{\mathcal{X}}_{j,r,k}' \)) and
  \( \mathcal{Z}_{j,r,k}' \) (resp.
  \( \widetilde{\mathcal{Z}}_{j,r,k}' \)).
\end{prop}
\begin{proof}
  The proof is similar to the arguments in \cite[(4.2) and
  (4.3)]{DJK}. Let
  \( \theta^{(s)},\dots,\theta^{(1)}, \theta^{(0)} \) be the sequence
  of frequency sequences in the construction of
  \( \Lambda(\bla,\fs(\bla)) \). We have
  \[
    \theta^{(i)} = \ppm_{2i}^{(\lambda_i)}(\theta^{(i+1)}) \qand (\theta^{(i+1)}_{2i}, \theta^{(i+1)}_{2i+1}) = (h,0),    
  \]
  for some \( h \geq 1 \), and the pair
  \( (\theta^{(i+1)}_{2i}, \theta^{(i+1)}_{2i+1}) \) moves to
  \( (\theta^{(i)}_v,\theta^{(i)}_{v+1}) \) for some \( v \geq 2i \).
  We have the explicit formula \eqref{eq:f_bar} of \( \theta^{(i)} \)
  from \( \theta^{(i+1)} \) and \( \lambda_i \), with the property
  that
  \[
    | \theta^{(i)} | = | \theta^{(i+1)} | + \lambda_i.
  \]
  An entry in the sequence that remains unchanged or is shifted two
  steps to the left does not affect the parity. Hence, we obtain
  \begin{align*}
    v\theta^{(i)}_v + (v+1)\theta^{(i)}_{v+1}
    &\equiv (2i)\cdot h + (2i+1)\cdot 0 + \lambda_i \pmod{2} \\
    & \equiv  \lambda_i \pmod{2}.
  \end{align*}

  In~\Cref{pro:property of Lambda}, we showed that consecutive
  particle motions starting from pairs of the form \( (h,0) \) do not
  interfere with each other by verifying that \( v+2 \leq v_0 \). Let
  \( \lambda^{(k)} = (\lambda^{(k)}_1,\dots,\lambda^{(k)}_\ell) \).
  Then, the last \( \ell \) steps in the construction of \( \Lambda \)
  consist of \( \ell \) particle motions starting from the pair
  \( (k,0) \). The pairs moved in these steps remain unchanged for the
  rest of the process, until the final frequency sequence
  \( f = \Lambda(\bla,\fs(\bla)) \) is obtained. From this, we
  immediately deduce the following: for any pair \( (f_i, f_{i+1}) \)
  satisfying \( f_i + f_{i+1} = k \), there exists a part
  \( \lambda^{(k)}_u \) of \( \lambda^{(k)} \) such that
  \( i \cdot f_i + (i+1)\cdot f_{i+1} \equiv \lambda^{(k)}_u \pmod{2}
  \).

  Conversely, the first \( \ell \) steps in the construction of
  \( \Gamma \) satisfy the corresponding parity condition. More
  precisely, for any pair \( (f_i, f_{i+1}) \) satisfying
  \( f_i + f_{i+1} = k \), there exists \( \lambda^{(k)}_{u+1} \) for
  some \( u = 0,\dots,\ell -1 \) such that
  \( \rstep_{2u}(\eta^{(u)}) = \lambda^{(k)}_{u+1} \equiv i\cdot f_i +
  (i+1)\cdot f_{i+1} \pmod{2} \).

  It follows that the bijection
  \( \Lambda : \mathcal{X}_{j,r,k} \to \mathcal{Y}_{j,r,k} \)
  naturally restricts to bijections
  \( \Lambda : \mathcal{X}_{j,r,k}' \to \mathcal{Y}_{j,r,k}' \) and
  \( \Lambda : \widetilde{\mathcal{X}}_{j,r,k}' \to
  \widetilde{\mathcal{Z}}_{j,r,k}' \).
\end{proof}

\begin{prop}\label{pro:gf Z'}
  There exists a size-preserving bijection from
  \( \mathcal{Y}_{j,r,k}' \) to \( \mathcal{Z}_{j,r,k}' \).
\end{prop}
\begin{proof}
  By \Cref{pro:gf Y'}, the right-hand side is the generating function
  for \( \mathcal{Y}_{j,r,k}' \). Consider the bijection \( \phi \)
  from \( \mathcal{Y}_{j,r,k} \) to \( \mathcal{Z}_{j,r,k} \) and its
  inverse map \( \pi \), described in the proof of \Cref{prop:bij_YZ}.
  We now prove that the bijection
  \( \phi : \mathcal{Y}_{j,r,k} \to \mathcal{Z}_{j,r,k} \) restricts
  to a bijection from \( \mathcal{Y}_{j,r,k}' \) to
  \( \mathcal{Z}_{j,r,k}' \). In other words, the parity condition is
  preserved under the map \( \phi \).

  Suppose \( f = (f_i)_{i \geq 0} \in \mathcal{Y}_{j,r,k}' \). Recall
  that the map \( \phi \) is defined by
  \( (f_0,f_1,f_2, \dots) \mapsto (f_0', f_1, f_2,\dots) \), where
  \[
    f_0' =
    \begin{cases}
      f_0 & \mbox{if \( f_0 \leq j-r \)},\\
      j-r+\ell & \mbox{if \( f_0 = j-r + 2 \ell \) for \( \ell = 1,\dots,r \),}
    \end{cases}
  \]
  if \( j \geq r \), and \( f_0' = \ell \) if \( f_0 = r-j+2 \ell \)
  for some \( \ell \in \{0,\dots,j\} \) if \( j < r \). We show that
  \( \phi(f) \in \mathcal{Z}_{j,r,k}' \). Since \( \phi \) modifies
  only the \( 0 \)th entry, it suffices to show that if
  \( f_0' + f_1 = k \), then
  \( 0\cdot f_0 + 1\cdot f_1 = f_1 \equiv k-r+j \pmod{2} \). By the
  definition of \( \phi \), we have \( f_0 \geq f_0' \), and hence
  \( f_0' + f_1 \leq f_0 + f_1 \leq k \). Therefore, if
  \( f_0' + f_1 = k \), then \( f_0' \) and \( f_0 \) are equal. Since
  \( f \in \mathcal{Y}_{j,r,k}' \), we have
  \( \phi(f) \in \mathcal{Z}_{j,r,k}' \), as required.

  Suppose \( g = (g_i)_{i \geq 0}\in \mathcal{Z}_{j,r,k}' \). Recall
  that the inverse map \( \pi \) is defined by
  \( (g_0,g_1,g_2, \dots) \mapsto (g_0', g_1, g_2,\dots) \), where
  \[
    g_0' =
    \begin{cases}
      g_0 & \mbox{if \( g_0 \leq j-r \)},\\
      j-r+2\ell & \mbox{if \( g_0 = j-r + \ell \) for \( \ell = 1,\dots,r \).}
    \end{cases}
  \]
  if \( j \geq r \), and \( g_0' = r-j+2 \ell \) if \( g_0 = \ell \)
  for some \( \ell \in \{0,\dots,j\} \) if \( j < r \). It suffices to
  show that if \( g_0' + g_1 = k \), then
  \( g_1 \equiv k-r+j \pmod{2} \). By construction, we have
  \( g_0 \leq g_0' \). If \( g_0 = g_0' \), then we have
  \( \pi(g) \in \mathcal{Y}_{j,r,k}' \) by the assumption that
  \( g \in \mathcal{Z}_{j,r,k}' \). Now suppose \( g_0 < g_0' \). Then
  the pair \( (g_0, g_0') \) is either of the form
  \[
    (j-r + \ell, j-r+ 2 \ell) \quad\text{or}\quad (\ell, r-j+2 \ell),
  \]
  for some \( \ell \), depending on whether \( j \geq r \) or
  \( j < r \), respectively. In both cases, we have
  \( g_0' \equiv j-r \pmod{2} \). Hence, if \( g_0' + g_1 = k \), then
  \( g_1 = k - g_0' \equiv k+r-j \pmod{2} \). Therefore,
  \( \pi(g) \in \mathcal{Y}_{j,r,k}' \), as desired.
\end{proof}

The generating function for \( \mathcal{X}_{j,r,k}' \) is given in
\Cref{pro:gf X'}. By \Cref{pro:gf Z'}, the generating function for
\( \mathcal{Z}_{j,r,k}' \) is given by
\[
  \sum_{\lambda \in \mathcal{Z}_{j,r,k}'} q^{|\lambda|} = 
  \sum_{s=0}^j \frac{(q^{2k+2}, q^{k+1-r+j-2s}, q^{k+1+r-j+2s} ; q^{2k+2})_\infty}{(q)_\infty}.
\]
The result then follows from \Cref{pro:bij X' and Z'}. Therefore,
this gives a combinatorial proof of \Cref{thm:Sta4.2}.

\begin{prop}\label{thm:gf Z tilde'}
  For non-negative integers \( j,r \), and \( k \) with
  \( j+r \leq k \), we have
\begin{multline*}
(1+q)\sum_{\lambda \in \widetilde{\mathcal{Z}}_{j,r,k}'} q^{|\lambda|} = 
  \sum_{s=0}^j \frac{(q^{2k+2}, q^{k+2-r+j-2s}, q^{k+r-j+2s} ; q^{2k+2})_\infty}{(q)_\infty}\\
    + q\sum_{s=0}^j \frac{(q^{2k+2}, q^{k-r+j-2s}, q^{k+2+r-j+2s} ; q^{2k+2})_\infty}{(q)_\infty}.
\end{multline*} 
\end{prop}

\begin{proof}
  Observe that
  \begin{itemize}
  \item \( \mathcal{Z}_{j,r+1,k}' \subseteq \widetilde{\mathcal{Z}}_{j,r,k}' \subseteq \mathcal{Z}_{j,r-1,k}' \), and
  \item The map \( (f_0, f_1,f_2, \dots) \mapsto (f_0, f_1-1,f_2, \cdots ) \) is a bijection from \( \widetilde{\mathcal{Z}}_{j,r,k}' \setminus \mathcal{Z}_{j,r+1,k}'\) to \( \mathcal{Z}_{j,r-1,k}' \setminus \widetilde{\mathcal{Z}}_{j,r,k}' \).
  \end{itemize}
  We have
  \begin{align*}
    (1+q) \sum_{\lambda \in \widetilde{\mathcal{Z}}_{j,r,k}'}q^{|\lambda|}
    &= \sum_{\lambda \in \widetilde{\mathcal{Z}}_{j,r,k}'}q^{|\lambda|} + \sum_{\lambda \in \widetilde{\mathcal{Z}}_{j,r,k}' \setminus \mathcal{Z}_{j,r+1,k}'}q^{|\lambda|+1} + q\sum_{\lambda \in \mathcal{Z}_{j,r+1,k}'}q^{|\lambda|}\\
    &= \sum_{\lambda \in \widetilde{\mathcal{Z}}_{j,r,k}'}q^{|\lambda|} + \sum_{\lambda \in \mathcal{Z}_{j,r-1,k}' \setminus \widetilde{\mathcal{Z}}_{j,r,k}'}q^{|\lambda|} + q\sum_{\lambda \in \mathcal{Z}_{j,r+1,k}'}q^{|\lambda|} \\
    &= \sum_{\lambda \in \mathcal{Z}_{j,r-1,k}'}q^{|\lambda|} + q\sum_{\lambda \in \mathcal{Z}_{j,r+1,k}'}q^{|\lambda|},
  \end{align*}
  which completes the proof.

\end{proof}

Together with the bijection \( \Lambda \) between
\( \widetilde{\mathcal{X}}_{j,r,k}' \) and
\( \widetilde{\mathcal{Z}}_{j,r,k}' \) in~\Cref{pro:bij X' and Z'},
and the generating functions for these sets given in~\Cref{pro:gf X'}
and \Cref{thm:gf Z tilde'}, this yields a combinatorial proof of
\Cref{thm:new_Kur}.

\section{Final remarks}
\label{sec:final-rmk}

We conclude with a few final remarks.

\begin{enumerate}
\item Stanton~\cite{Stant} proved \Cref{thm:Sta3.2} using
  \Cref{thm:Sta3.1}, and similarly showed that \Cref{thm:Sta4.1}
  implies \Cref{thm:Sta4.2}. In contrast, our proofs rely on different
  Bailey pair constructions. From the perspective of Bailey pairs, it
  would be interesting to investigate whether \Cref{thm:Sta3.1}
  implies \Cref{thm:Sta3.2}, or \Cref{thm:Sta4.1} implies
  \Cref{thm:Sta4.2}, and similarly regarding Theorems~\ref{thm:binom_BGG} and~\ref{thm:GGrj}.
\item As mentioned in \Cref{prob:stanton}, Stanton posed the challenge
  of finding a partition-theoretic interpretation not only for the
  non-binomial extension but also for the binomial extension. While we
  have treated only the non-binomial case, the question of a
  partition-theoretic interpretation for the binomial case remains
  open and intriguing.
\item There are combinatorial interpretations for certain
  generalisations of the G\"ollnitz--Gordon identities, see for instance~\cite{Br80} or~\cite{HZ}.
  However, no study has applied the particle motion framework to these
  identities. Exploring such an approach may lead to new identities or
  provide combinatorial proofs of the Bresssoud--G\"ollnitz--Gordon
  identities~\eqref{eq:BGG}, \Cref{thm:binom_BGG} and
  \Cref{thm:GGrj}. We have not yet explored this direction in depth.
\item One could also wonder if there could exist binomial versions of Theorems~\ref{thm:newSlater} and~\ref{thm:newSlater2}.
\end{enumerate}

\section*{Acknowledgements}

The first two authors are supported by the SNSF Eccellenza grant PCEFP2 202784.

\bibliographystyle{alpha}

\end{document}